\documentclass[10pt, reqno]{amsart}
\usepackage{graphicx, amssymb, amsmath, amsthm}
\numberwithin{equation}{section}
\usepackage{xcolor}
\usepackage{amscd}
\usepackage{amssymb}
\usepackage{latexsym}

\usepackage{amsmath}
\usepackage{amscd}
\usepackage{cite}
\usepackage{color}
\usepackage{enumerate}
\usepackage{amsfonts}
\usepackage{graphicx}
\usepackage{mathrsfs}
\usepackage{setspace}
\usepackage{hyperref}
\def\pa{\partial}

\let\Re=\undefined\DeclareMathOperator*{\Re}{Re}
\let\Im=\undefined\DeclareMathOperator*{\Im}{Im}

\newcommand{\qtq}[1]{\quad\text{#1}\quad}
\newcommand{\R}{\mathbb{R}}
\newcommand{\C}{\mathbb{C}}

\newcommand{\eps}{\varepsilon}

\newtheorem{theorem}{Theorem}[section]

\newtheorem{lemma}[theorem]{Lemma}

\newtheorem{corollary}[theorem]{Corollary}

\newtheorem{proposition}[theorem]{Proposition}

\theoremstyle{definition}
\newtheorem{definition}[theorem]{Definition}
\newtheorem{remark}[theorem]{Remark}

\begin{document}

\title[Rigidity result for the mass-critical Hartree equation]{Classification of minimal-mass blowup solutions to the mass-critical Hartree equation at the $L^2$-critical level}

\author[Z. Ma]{Zuyu Ma}
\address{Zuyu Ma, Graduate School of China Academy of Engineering Physics, Beijing 100088, China}
\email{mazuyu23@gscaep.ac.cn}

\author[C. Miao]{Changxing Miao}
\address{Changxing Miao, School of Mathematics and Physics, University of Science and Technology Beijing, Beijing 100083, China}
\email{miao\_changxing@ustb.edu.cn, miao\_changxing@iapcm.ac.cn}

\author[J. Zheng]{Jiqiang Zheng}
\address{Jiqiang Zheng, Institute of Applied Physics and Computational Mathematics and National Key Laboratory of Computational Physics, Beijing 100088, China}
\email{zheng\_jiqiang@iapcm.ac.cn, zhengjiqiang@gmail.com}

\begin{abstract}
In this paper, we study the focusing mass-critical Hartree equation
\[
iu_t+\Delta u=-\bigl(|x|^{-2}*|u|^2\bigr)u
\]
at the ground-state mass $M(Q)$. We prove that, up to the symmetries of the equation, every global minimal-mass non-scattering solution is a solitary wave, while every finite-time minimal-mass blowup solution is a pseudo-conformal image of a solitary wave. This extends the results of \cite{LZ28,Miao-Xu-Zhao3,Wang-Gao} to the scaling-critical space $L^2$ and removes the radial assumption in the global case.
\end{abstract}
\maketitle

\section{Introduction}

\subsection{Setting of the problem}\label{intro:setting}

We consider the Cauchy problem for the focusing mass-critical Hartree
equation
\begin{equation*}\tag{Hartree}\label{1.1}
\begin{cases}
 iu_t+\Delta u=-\bigl(|x|^{-2}*|u|^2\bigr)u=:-F(u),
 \quad (t,x)\in\R\times\R^d,\\
 u|_{t=0}=u_0\in L^2(\R^d),
\end{cases}
\end{equation*}
where $d\geq3$ and $u$ is complex-valued. The goal of this paper is
to characterize the non-scattering solutions of \eqref{1.1} with
ground-state mass at the scaling-critical regularity $L^2(\R^d)$.

Hartree equations and local nonlinear Schr\"odinger equations (NLS) both
arise as effective descriptions of interacting bosons. In the mean-field
regime, the condensate evolves under a self-consistent potential generated
by its density. In the absence of external fields, this leads to the
Hartree equation
\[
 i\partial_t\psi+\Delta\psi=(w*|\psi|^2)\psi,
\]
where $w$ is the two-body interaction potential and $|\psi|^2$ is the
condensate density; see \cite{Frohlich-Lenzmann}. For short-range
interactions, a contact approximation $w\approx g\delta_0$ formally gives
the local cubic NLS
\[
 i\partial_t\psi+\Delta\psi=g|\psi|^2\psi.
\]
In the Gross--Pitaevskii regime for dilute three-dimensional Bose gases
with repulsive interactions, this local cubic equation can be derived
rigorously, with a coupling constant determined by the scattering length
of the pair potential \cite{Erdos-Schlein-Yau}.

The local model describes a nonlinear response determined by the density
at the same point, while the Hartree convolution retains the spatial
structure of the pair interaction. This is essential for Coulomb and
Newtonian interactions, whose kernel in three dimensions is proportional
to $|x|^{-1}$. Such algebraically decaying potentials couple spatially
separated wave packets even when their overlap is negligible; see
\cite{Gomez-Schmid-Wu}. Related Hartree models with semi-relativistic dispersion
describe self-gravitating bosonic matter \cite{Elgart-Schlein}.

In \eqref{1.1}, we choose the attractive inverse-square potential
$w(x)=-|x|^{-2}$. The Newtonian kernel in dimension $d\geq3$ is
proportional to $|x|^{2-d}$, so \eqref{1.1} has a
Schr\"odinger--Poisson interpretation when $d=4$. We study the
inverse-square interaction in every dimension $d\geq3$, where its
homogeneity yields the mass-critical scaling described below.

Equation \eqref{1.1} is invariant under the scaling
\[
 u(t,x)\mapsto u_\lambda(t,x)
 :=\lambda^{d/2}u(\lambda^2t,\lambda x),\qquad \lambda>0.
\]
Solutions conserve the mass
\[
 M(u(t)):=\int_{\R^d}|u(t,x)|^2\,dx,
\]
and $H^1$ solutions also conserve the energy
\[
 E(u(t)):=\frac12\int_{\R^d}|\nabla u(t,x)|^2\,dx
 -\frac14\iint_{\R^d\times\R^d}
 \frac{|u(t,x)|^2|u(t,y)|^2}{|x-y|^2}\,dx\,dy.
\]
Since the scaling preserves the $L^2$ norm, \eqref{1.1} is called
mass-critical.

Within the local focusing power-type NLS family, the same mass-preserving
scaling selects the mass-critical equation
\begin{equation*}\tag{NLS}\label{NLS}
 iu_t+\Delta u=-|u|^{4/d}u.
\end{equation*}
The cubic contact model is mass-critical in dimension two, whereas the
inverse-square Hartree interaction has this property in every dimension
$d\geq3$. The mass-critical NLS thus provides a natural local counterpart
to \eqref{1.1} for the study of threshold dynamics.
Both equations are invariant under phase rotations, space-time
translations, time reversal, and Galilean transformations
\[
 u(t,x)\mapsto e^{ix\cdot\xi-it|\xi|^2}u(t,x-2t\xi).
\]
They also admit the pseudo-conformal transformation
\[
 (\mathcal Cu)(t,x)
 :=|t|^{-d/2}e^{i|x|^2/(4t)}
 \overline{u\left(\frac1t,\frac{x}{t}\right)}.
\]

We work with strong $L^2$ solutions. A function $u:I\times\R^d\to\C$
is a strong solution if
\[
 u\in C_tL_x^2(J\times\R^d)\cap
 L_{t,x}^{\frac{2(d+2)}d}(J\times\R^d)
\]
for every compact interval $J\subset I$, and
\[
 u(t)=e^{i(t-t_0)\Delta}u(t_0)
 +i\int_{t_0}^t e^{i(t-s)\Delta}F(u(s))\,ds,
 \qquad t,t_0\in I.
\]
For each $u_0\in L^2$, the local theory gives a unique solution on a
maximal open interval $I$ containing zero; see
\cite{Miao-Xu-Zhao2,MMRZ}. We write
\[
 S_I(u):=\|u\|_{L_{t,x}^{\frac{2(d+2)}d}(I\times\R^d)}.
\]
If this norm is finite in the forward time direction, then
$\sup I=\infty$ and there exists $u_+\in L^2$ such that
\[
 \lim_{t\to\infty}\|u(t)-e^{it\Delta}u_+\|_{L^2}=0.
\]
In this case, we say that $u$ scatters forward in time. The analogous
statement holds backward in time. We say that $u$ blows up in a time
direction if the corresponding spacetime norm is infinite. This
includes finite-time blowup and global solutions which do not scatter.
Small $L^2$ initial data give global solutions which scatter in both
time directions.

The ground state determines the mass threshold for scattering. Let
$Q$ be a positive radial ground state satisfying
\begin{equation}\label{grounstate}
 \Delta Q-Q+\bigl(|x|^{-2}*|Q|^2\bigr)Q=0.
\end{equation}
The variational characterization of $Q$ gives the sharp
Gagliardo--Nirenberg inequality
\begin{equation}\label{sharp-HGN-intro}
 \int_{\R^d}\bigl(|x|^{-2}*|f|^2\bigr)|f|^2\,dx
 \leq\frac{2M(f)}{M(Q)}\|\nabla f\|_{L^2}^2,
 \qquad f\in H^1(\R^d);
\end{equation}
see \cite{Miao-Xu-Zhao3,Miao-Xu-Zhao2}. It follows that
\[
 E(f)\geq\frac12\left(1-\frac{M(f)}{M(Q)}\right)
 \|\nabla f\|_{L^2}^2.
\]
Hence conservation of mass and energy gives global well-posedness
for $H^1$ data with mass below $M(Q)$. At mass $M(Q)$, the standing
wave $e^{it}Q$ is global and does not scatter. Its pseudo-conformal
transform has the same mass and blows up in finite time.

Miao, Xu, and Zhao \cite{Miao-Xu-Zhao2} proved global well-posedness
and scattering below the ground-state mass for radial $L^2$ data.
In our recent work with Rosenzweig \cite{MMRZ}, we removed the
radial assumption and proved the following result.

\begin{theorem}[Scattering below the ground-state mass,
\cite{Miao-Xu-Zhao2,MMRZ}]\label{MMRZthm}
Let $u_0\in L^2(\R^d)$ satisfy $M(u_0)<M(Q)$. Then the solution of
\eqref{1.1} is global and scatters in both time directions.
\end{theorem}

The standing wave and its pseudo-conformal transform show that the
mass condition in Theorem~\ref{MMRZthm} is sharp. It is natural to ask whether these
are the only minimal-mass blowup solutions, up to the symmetries of
the equation.

For finite-time blowup, Miao, Xu, and Zhao \cite{Miao-Xu-Zhao3}
proved that every minimal-mass blowup solution in $H^1(\R^d)$ is a
pseudo-conformal transform of the ground state, up to symmetries.

Global minimal-mass non-scattering solutions with radial $H^1$ data
were classified by Li and Zhang \cite{LZ28} for $d\geq5$, and by Wang
and Gao \cite{Wang-Gao} for $d=3,4$.

For $d\ne4$, the preceding classification results
\cite{LZ28,Wang-Gao} were proved under the uniqueness assumption on
the ground state, and \cite{Wang-Gao} also assumes nondegeneracy for
$d=3$. The uniqueness assumption is also used in
\cite[Appendix A]{MMRZ}. We assume throughout this paper that $Q$ is
unique and nondegenerate. These properties hold in dimension four
\cite{Krieger,Chen}; the spectral properties used below are recalled
in Appendix~\ref{sec-mod}.

\subsection{Main result}\label{intro:main-result}

Our main result is the following.

\begin{theorem}[Classification at the ground-state mass]\label{main1}
Let $d\geq3$ and let $u:I\times\R^d\to\C$ be a maximal-lifespan
strong $L^2$ solution of
\eqref{1.1} satisfying
\[
M(u)=M(Q),\qquad S_I(u)=\infty.
\]
Then exactly one of the following alternatives occurs:
\begin{enumerate}
\item $I=\R$ and $u$ is a solitary wave. More precisely, there exist
$\lambda>0$, $x_0,\xi\in\R^d$, and $\gamma\in[0,2\pi]$ such that
\begin{align}\label{soliton}
u(t,x)=e^{i\gamma-it|\xi|^2}e^{i\lambda^2t}e^{ix\cdot\xi}
\lambda^{d/2}Q\bigl(\lambda(x-2\xi t)-x_0\bigr).
\end{align}
\item $I=(-\infty,T)$ or $I=(T,\infty)$ for some $T\in\R$, and $u$ is
a pseudo-conformal transform of a solitary wave. More precisely, there
exist $\lambda>0$, $x_0,\xi\in\R^d$, and $\gamma\in[0,2\pi]$ such that
\begin{align}\label{pseudosoliton}
u(t,x)=\frac{\lambda^{d/2}}{|T-t|^{d/2}}e^{i\gamma}
e^{\frac{i|x-\xi|^2}{4(t-T)}}e^{-i\frac{\lambda^2}{t-T}}
Q\left(\frac{\lambda(x-\xi)-(T-t)x_0}{T-t}\right).
\end{align}
\end{enumerate}
\end{theorem}

Theorem~\ref{main1} extends the classification results of
\cite{LZ28,Miao-Xu-Zhao3,Wang-Gao} from $H^1$ to $L^2$ and removes
the radial assumption in the global case. Together with
Theorem~\ref{MMRZthm}, it classifies all blowup solutions with mass
at most $M(Q)$.

This result contributes to the study of long-time and blowup dynamics
for dispersive equations with nonlocal nonlinearities, an active area
of research; see, for instance,
\cite{Gomez-Schmid-Wu,Li-Liu-CSS,Dodson-CS}.
We mention in particular the closely related rigidity results for the
equivariant Chern--Simons--Schr\"odinger equation, whose reduced form
contains nonlocal gauge terms. Li and Liu
\cite{Li-Liu-CSS} classified threshold solutions in $H^1$ in the
self-dual case and proved finite-time blowup rigidity in the
non-self-dual case. Dodson \cite{Dodson-CS} subsequently established
$L^2$ rigidity at the soliton threshold in the non-self-dual equivariant
setting.

Our proof follows the concentration-compactness and rigidity strategy
of Kenig and Merle \cite{kenig-merle1}, as adapted to the mass-critical
setting by Tao, Vi\c{s}an, and Zhang \cite{TVZ2}.
The concentration-compactness argument shows that a minimal-mass blowup
solution has a precompact $L^2$ orbit modulo symmetries in its blowup
direction. In our previous work \cite{MMRZ}, the rigidity argument
ruled out minimal counterexamples to scattering below $M(Q)$.
At the threshold, we instead identify the minimal-mass blowup solutions
with this compactness property as solitary waves or their
pseudo-conformal transforms.

For the threshold classification problem for \eqref{NLS}, Killip, Li,
Vi\c{s}an, and Zhang \cite{KLVZ} used this framework to classify global
non-scattering solutions at the ground-state mass for radial $H^1$
data in dimensions $d\geq4$. Together with Merle's earlier
classification of finite-time minimal-mass blowup solutions in $H^1$
\cite{M}, this gave the complete threshold classification in that
setting; see also Hmidi and Keraani \cite{MS} for a simpler proof of
the finite-time result. Li and Zhang subsequently extended the
rigidity theory in \cite{LZ0,LZ1,LZ2,LZ3}.
Dodson \cite{D2,D1} completed the classification for general $L^2$
data in dimensions $1\leq d\leq15$, removing both the additional
regularity and the radial symmetry assumptions. For \eqref{1.1},
carrying out this program requires controlling the nonlocal
interactions at the same critical regularity, as we explain below.

\subsection{Outline of the proof}\label{intro:outline}

The proof of Theorem~\ref{main1} is based on Dodson's work on the
mass-critical NLS \cite{D2,D1} and on our previous work with
Rosenzweig \cite{MMRZ}. The convergence result at mass $M(Q)$ in
\cite{MMRZ} is used to reduce the classification to a rigidity
theorem near the ground state.

The nonlocal nature of the Hartree nonlinearity creates several
difficulties in proving Theorem~\ref{main1} at $L^2$ regularity. In
particular, the spatial and frequency localizations required by the energy,
virial, and long-time estimates generate nonlocal interaction terms that do
not arise in the same form for the local NLS. To control these errors, we
exploit the structure of the Hartree equation near the ground state. More
precisely, we combine the decay of $Q$ and its induced potential with
local smoothing in the long-time estimates, and use the symmetry of
the interaction kernel in the localized virial argument.

To apply the localized energy and virial estimates at $L^2$
regularity, we work with a frequency truncation of $u$. Controlling
the resulting errors requires spacetime bounds for the
high-frequency part of $u$ over long intervals. We obtain these
bounds from the long-time Strichartz estimate in
Theorem~\ref{longtime}.

A key difficulty in establishing this estimate lies in the
interaction between the high-frequency component and the ground
state. More precisely, we encounter terms of the form
\begin{equation}\label{intro:nonlocal-interactions}
 (K*Q^2)h,
 \qquad
 \bigl(K*(Q\overline h)\bigr)Q,
 \qquad K(x)=|x|^{-2},
\end{equation}
where $h$ is a high-frequency function. A direct application of
H\"older's inequality and the Hardy--Littlewood--Sobolev inequality
gives bounds that grow with the length of the time interval.
The potential $K*Q^2$ in the first term has only inverse-square
decay, whereas $Q$ decays rapidly. In the second term, the
high-frequency factor lies inside the convolution. The key point is
to distinguish these two ways in which the high-frequency factor
enters the interaction. Keeping track of the ground-state factors
inside and outside the convolution, we combine their spatial decay
with local smoothing \cite{CS88,D2} to obtain the frequency gain
needed to close the induction. In the non-radial case, the estimates
must also keep track of the motion of the spatial center; see
Section~\ref{sec-long}.

A further difficulty comes from spatial localization in the virial
argument. Let $0\leq\chi_R\leq1$ be a smooth cutoff equal to one on
$B_R$ and vanishing outside $B_{2R}$. Unlike the corresponding error
for a local power nonlinearity, $\chi_RF(u)-F(\chi_Ru)$ is not supported
in the transition annulus. Indeed,
\[
 \chi_RF(u)-F(\chi_Ru)
 =\chi_Ru\bigl(K*((1-\chi_R^2)|u|^2)\bigr).
\]
The localized virial identity consequently contains terms coupling
the ground state in the interior to the remainder in the exterior.
These terms are quadratic in the remainder of the modulation
decomposition, as is the positive term in the virial inequality,
and hence must be estimated with a small coefficient. To obtain
this smallness, we use the symmetry of the interaction kernel to
write the nonlinear virial contribution as a symmetric double
integral. This form allows us to separate the interaction between
the interior ground state and the exterior remainder, and to combine
spatial separation with the decay of $Q$. Together with the
truncated energy estimates, these bounds control the quadratic
coupling with the small coefficient required for absorption; see
Section~\ref{Sec:viria}.

We now describe the main steps of the proof in more detail.

\subsubsection*{Reduction to solutions near the ground state.}

By time reversal, we may assume that $u$ blows up forward in time.
Theorem~\ref{MMRZthm}, linear profile decomposition, and stability
imply that a minimal-mass blowup solution is almost periodic modulo
symmetries in its blowup direction. Using
\cite[Appendix A]{MMRZ} and a compactness argument, we obtain times
$t_n\to\sup I$ and symmetries $g_n$ such that
\begin{equation}\label{equ:gnutnQ}
 g_nu(t_n)\longrightarrow Q
 \qquad\text{strongly in }L^2(\R^d).
\end{equation}
The argument of \cite{D2,D1} then reduces Theorem~\ref{main1} to
Theorem~\ref{t2.3}, which classifies solutions that stay in a small
$L^2$ neighborhood of the ground-state orbit in one time direction.
We give this reduction in Section~\ref{reductionsec}.

For such a solution, we write
\begin{equation}\label{intro:modulation}
 u(t,x)=\frac{e^{-i\gamma(t)}e^{-iy\cdot\xi(t)}}{\lambda(t)^{d/2}}
 \bigl(Q(y)+\epsilon(t,y)\bigr),
 \qquad y=\frac{x-x(t)}{\lambda(t)}.
\end{equation}
The modulation parameters are chosen by the $L^2$ orthogonality
conditions in Appendix~\ref{sec-mod}. We denote the two terms in
\eqref{intro:modulation} by $\widetilde Q$ and
$\widetilde\epsilon$, respectively.

\subsubsection*{Long-time Strichartz estimate.}

We next prove a long-time Strichartz estimate for solutions close to
the ground-state orbit. Let $J=[a,b]$ satisfy the hypotheses of
Theorem~\ref{longtime}, and set
\[
 T=\int_a^b\lambda(t)^{-2}\,dt=2^{3k_0}.
\]
Then
\begin{equation}\label{intro-LTS}
 \|P_{\geq k_0}u\|_{U_\Delta^2(J)}
 \lesssim
 \left(\frac1T\int_a^b\lambda(t)^{-2}
 \|\epsilon(t)\|_{L^2}^2\,dt\right)^{1/2}+T^{-10},
\end{equation}
with an implicit constant independent of $T$. Here $P_{\geq k_0}$
is the Littlewood--Paley projection to frequencies
$|\xi|\gtrsim2^{k_0}$, and $U_\Delta^2$ is the atomic space recalled
in Section~\ref{sec-long}; see also \cite{HHK}.

The proof is by induction on frequency. After isolating a
high-frequency factor, we decompose the remaining factors into
$\widetilde Q$ and $\widetilde\epsilon$. The terms involving
$\widetilde\epsilon$ are controlled by its smallness, while the terms
with two ground-state factors are estimated using local smoothing
and the decay of $Q$ and $|x|^{-2}*Q^2$.

\subsubsection*{Truncated energy and rigidity.}\label{intro:cutoffs}

Let $v=P_{\leq k_0+4}u$. The equation for $v$ differs from
\eqref{1.1} by the term
\[
 P_{\leq k_0+4}F(u)-F(v).
\]
Using \eqref{intro-LTS} and the Lorentz-space estimate in
Remark~\ref{lorentzlongtime}, we control the increment of $E(v)$.
Under the hypotheses of Theorem~\ref{t4.2}, this gives
\begin{equation}\label{intro:energy-increment}
 |E(v(t))-E(v(t_0))|
 \lesssim\frac{2^{2k_0}}{T}
 \int_a^b\lambda(\tau)^{-2}\|\epsilon(\tau)\|_{L^2}^2\,d\tau
 +2^{2k_0}T^{-10},\qquad t,t_0\in J.
\end{equation}

Combining \eqref{intro:energy-increment} with the energy expansion
around the modulated ground state gives the upper bound for $E(v)$
in Theorem~\ref{t4.2}. For the lower bound, the ground-state equation
and the exact mass constraint allow us to organize the expansion at
quadratic order. The modulation orthogonality conditions then give
access to the coercivity of the linearized operators, with the
frequency truncation errors controlled by the estimates above.
Comparing the two bounds gives the estimates for $\epsilon$ and its
low-frequency part in Corollary~\ref{c4.3}.

We then use these estimates in the modulated virial inequality of
Section~\ref{Sec:viria} to estimate
$\int_a^b\lambda(t)^{-2}\|\epsilon(t)\|_{L^2}^2\,dt$.
In the non-radial setting, the modulation equations control the
motion of the spatial and frequency centers. This allows the
localization estimates to be applied on the long intervals required
by the rigidity argument.
With $ds=\lambda(t)^{-2}dt$, the virial inequality implies
$\|\epsilon(s)\|_{L^2}\in L_s^p$ for every $p>1$. It also gives
almost monotonicity of $\lambda(s)$. For a global solution, we use
these estimates to find times at which $\epsilon$ decays sufficiently
rapidly that the energy estimates in Section~\ref{sec-en} force $\epsilon\equiv0$.
For finite-time blowup, we apply the pseudo-conformal transformation
and reduce to the global case.

\subsection{Organization of the paper}

In Section~\ref{sec-seq}, we  prove sequential convergence to the ground-state
orbit. Section~\ref{reductionsec} reduces Theorem~\ref{main1} to the
classification of solutions that remain in a small $L^2$ neighborhood
of this orbit. In Section~\ref{sec-long}, we  establish the long-time Strichartz estimate in $U_\Delta^2$, and we show  the
lower and upper bounds for the frequency-truncated energy in Section~\ref{sec-en}.
Section~\ref{Sec:viria} derives the modulated virial inequality.
In Section~\ref{Sec:Thm2.3}, we  utilize these estimates to prove rigidity for
global solutions and then treat finite-time blowup by the
pseudo-conformal transformation. Appendix~\ref{sec-mod} contains the
non-radial modulation analysis and the coercivity estimates for the
linearized operators.

\subsection*{Acknowledgements} Part of this work was done when Z. Ma was visiting Southeast University.  C. Miao was supported by the National Key R\&D program
of China: No.2021YFA1002500 and NSFC Grants 12371095 and 12531005. J. Zheng was supported by  National Key R\&D program of China: 2021YFA1002500 and  NSFC Grant 12671284.

\subsection*{Declaration of AI Use}
The authors used ChatGPT for language polishing, checking typographical errors, and reviewing some calculations and proof details. The authors reviewed all revisions and take full responsibility for the mathematical content and conclusions of the paper.

\section{Notation and useful lemmas}

\subsection{Notation} We begin by setting up some notation. For nonnegative quantities $X$ and $Y$, we write $X\lesssim Y$ to denote the estimate $X\leq C Y$ for some $C>0$. If $X\lesssim Y\lesssim X$, we will write $X\sim Y$. Dependence on various parameters will be indicated by subscripts, e.g. $X\lesssim_u Y$ indicates $X\leq CY$ for some $C=C(u)$. 

We will use the expression $O(X)$ to denote a finite linear combination of terms that resemble $X$ up to Littlewood--Paley projections, complex conjugation, and/or maximal functions. We will use the expression $X\pm$ to denote $X\pm\eps$ for some small $\eps>0$. We will make use of the $L^2$ and $\dot H^1$ inner products given by $(f,g)_{L^2}=\Re\int f\bar g$ and $(f,g)_{\dot H^1}=(\nabla f,\nabla g)_{L^2}$. We will use $A^\perp$ to denote the orthogonal complement of a set $A$. 

For a spacetime slab $I\times\R^d$, we write $L_t^q L_x^r(I\times\R^d)$ for the Banach space of functions $u:I\times\R^d\to\C$ equipped with the norm
\[
\|u\|_{L_t^{q}L_x^r(I\times\R^d)}:=\bigg(\int_I \|u(t)\|_{L_x^r(\R^d)}^{q}\,dt\bigg)^{1/q},
\]
with the usual adjustments if $q$ or $r$ is infinity. When $q=r$, we abbreviate $L_t^qL_x^q=L_{t,x}^q$. We also abbreviate $\|f\|_{L_x^r(\R^d)}$ to $\|f\|_{L_x^r}.$ For $1\leq r\leq\infty$, we use $r'$ to denote the dual exponent to $r$, i.e. the solution to $\frac{1}{r}+\frac{1}{r'}=1.$

The Fourier transform on $\mathbb{R}^d$ is given by
\begin{equation*}
\aligned \widehat{f}(\xi)= \mathcal{F} f(\xi)= \tfrac{1}{(2\pi)^{\frac{d}{2}}}\int_{\mathbb{R}^d}e^{- ix\cdot \xi}f(x)\,dx.
\endaligned
\end{equation*}
We then define the fractional differentiation operator $|\nabla|^s=\mathcal{F}^{-1}|\xi|^s\mathcal{F}$, with the corresponding homogeneous Sobolev norm 
\[
\|f\|_{\dot{H}_x^s(\R^d)}:=\||\nabla|^s f\|_{L_x^2(\R^d)}.
\]

Let $\varphi\in C_0^\infty(\R^d)$ be a real-valued, nonnegative, radial, and radially decreasing function such that
\begin{align*}
\varphi(\xi) =
\begin{cases}
1, \ |\xi| \leq 1, \\
0, \  |\xi| \ge 2,
\end{cases}
\end{align*}
and define $\chi(\xi)=\varphi(\xi)-\varphi(2\xi)$. For a real number $N$ (in particular, for integer dyadic indices), we define the Littlewood-Paley projectors $P_{\leq N}, P_{N}, P_{\geq N}$ via
\begin{align*}
\widehat{P_Nf}(\xi):=\chi\big(\tfrac{\xi}{2^N} \big)\widehat{f}(\xi),  \quad \widehat{P_{\leq N}f}(\xi):=\varphi\big(\tfrac{\xi}{2^N} \big)\widehat{f}(\xi),\quad \widehat{P_{\geq N}f}(\xi):=\Big(1-\varphi\big(\tfrac{\xi}{2^N} \big)\Big)\widehat{f}(\xi).
\end{align*}

\subsection{Some analysis tools} We denote the free Schr\"odinger propagator by $e^{it\Delta}$.  It is most naturally defined as the Fourier multiplier operator with symbol $e^{-it|\xi|^2}$, but also has a physical-space representation given by convolution with a complex Gaussian. This operator satisfies well-known estimates known as Strichartz estimates, which we state as follows.

First, for $d\geq3$, we call a pair of exponents $(q,r)$ \emph{Schr\"odinger admissible}  if $2\leq q,r\leq\infty$ and $\tfrac{2}{q}+\tfrac{d}{r}=\frac{d}{2}$, denoted by $(q,r)\in\Lambda_0$. For a spacetime slab $I\times\R^d$, we define the Strichartz norm 
\[
\|u\|_{S^0(I)}:=\sup\big\{\|u\|_{L_t^{q}L_x^r(I\times\R^d)}:(q,r)\in\Lambda_0\big\}.
\]
We denote $S^0(I)$ to be the closure of all test functions under this norm and write $N^0(I)$ for the dual of
$S^0(I)$. We then have the following:

\begin{proposition}[Strichartz estimates, \cite{GinibreVelo, KeelTao, Strichartz}]\label{prop1}
Let $d\geq3$, $s\geq 0$ and suppose $u:I\times\R^d\to\C$ is a solution to $(i\partial_t+\Delta)u=F$. For any $t_0\in I$, 
\[
\||\nabla|^s u\|_{S^0(I)}\lesssim\||\nabla|^s u(t_0)\|_{L_x^2(\R^d)}+\||\nabla|^s F\|_{N^0(I)}.
\]
\end{proposition}

We write $\phi_f:=|x|^{-2}*f$ and next record some basic properties of the ground state $Q$: 

\begin{proposition}[Existence of ground state,\cite{Miao-Xu-Zhao3,Moroz}]
     There exists a positive, radial function $Q \in H^1(\R^d)$  which solves the ground state equation \eqref{grounstate}. Moreover, $Q$ has the variational characterization
\[
J(Q) = \min_{u \in H^1(\R^d)\setminus\{0\}} J(u), \quad J(u) := \frac{\|u\|_{L^2(\R^d)}^2 \|\nabla u\|_{L^2(\R^d)}^2}{\int_{\R^d} |u(x)|^2 (|x|^{-2}\ast |u|^2)(x) dx}.
\]
Moreover, $J(Q) = \frac{\|Q\|_{L^2(\R^d)}^2}{2}$, so that
\[
\int_{\R^d} |u(x)|^2 (|x|^{-2} * |u|^2)(x) dx \leq \frac{2 \|u\|_{L^2(\R^d)}^2 \|\nabla u\|_{L^2(\R^d)}^2}{\|Q\|_{L^2(\R^d)}^2}, \quad \forall u \in H^1(\R^d).
\]
\end{proposition}
\begin{proposition}[\cite{Chen,Li,Moroz}]\label{grondstate1}
The following two statements hold:
\begin{enumerate}

\item (Decay of the ground state) There exists a constant $\delta>0$ such that
    \[
   |Q(x)|+|\nabla Q(x)|\lesssim e^{-\delta|x|},
    \]
\item (Decay of the non-local term)  We have
\[
\big[|x|^{-2}\ast |Q|^2\big](x)\lesssim \frac{1}{1+|x|^2}.
\]
\end{enumerate}
\end{proposition}
As a consequence, we have
\begin{proposition}[$W^{k,p}$ estimate of $Q$]\label{grondstate2}
    For any  $p\in(1,\infty)$ and integer $k\geq0$, we have
    \[
    \||\nabla|^k Q\|_{L^p}\lesssim_{k,p}1 
    \]
\end{proposition}
\begin{proof} From Proposition \ref{grondstate1} we see that
\[
Q\in L^p\qtq{for all } p\in[1,\infty]\qtq{ and} |x|^{-2}\ast |Q|^2\in L^q\qtq{for all} q\in(\frac{d}{2},\infty].
\]
Recalling $Q$ satisfies the elliptic equation \eqref{grounstate}, therefore, by the Calderón-Zygmund estimate, we deduce that
\begin{equation}\label{firstq}
  Q\in W^{2,p}\qtq{for all}p\in (1,\infty).  
\end{equation}
Noticing that for any $u,v\in W^{2k,p},\forall p>1$, we have
\[
\Delta^k \left[(|x|^{-2}\ast |u|^2)v\right]\in L^p\qtq{for all} p\in(1,\infty),
\]
by the Hardy-Littlewood-Sobolev inequality and H\"older's inequality, then the desired estimate comes from Calderón-Zygmund estimate by iteration, starting with \eqref{firstq} and chain rule.
\end{proof}

\subsection{Almost periodic solutions}
\begin{definition}[Almost periodic modulo symmetries (APMS)]
We say that $u \in C_{t, \text { loc }}^0 L_x^2\left(I \times \R^d\right)$ is almost periodic  if there exists a spatial center function $x: I \rightarrow \R^d$, a frequency center function $\xi: I \rightarrow \R^d$, a frequency scale function $N: I \rightarrow(0, \infty)$, and a compactness modulus function $C:(0, \infty) \rightarrow[0, \infty)$, such that for every $\eta>0$, we have the spatial and frequency localization estimates
\[\int_{|x-x(t)| \geq C(\eta) / N(t)}|u(t, x)|^2 d x+\int_{|\xi-\xi(t)|\geq C(\eta) N(t)}|\hat{u}(t, \xi)|^2 d \xi \leq \eta, \quad \forall t \in I.\]
\end{definition}
\begin{lemma}\label{alcompact}
Suppose $u:I\times\R^d\to\C$ is a maximal-lifespan solution to \eqref{1.1} satisfying $M(u)=M(Q)$.
If $u$ blows up forward in time, then $u$ is almost periodic on $[0,\sup I)$.  Similarly, if $u$ blows up backward in time, then $u$ is almost periodic on $(\inf I,0]$.
\end{lemma}

\begin{proof}[Sketch of the proof] The argument to obtain this lemma follows along fairly standard lines (see e.g. \cite{TVZ2,Visan}), so we only outline a brief sketch here. The idea is to take an arbitrary sequence of times $t_n$ and apply the linear $L^2$ profile decomposition (see e.g. \cite{MV}) to $u(t_n)$. Using Theorem~\ref{MMRZthm} and stability theory, we can use the construction of a \emph{nonlinear} profile decomposition to argue that there must be exactly one profile in the profile decomposition (since in all other cases we obtain scattering for the solution $u$). This yields the desired compactness.
\end{proof}

\subsection{Lorentz spaces and  nonlinear estimates}\label{S:2.2}
 In this subsection, we recall the well-known Lorentz space and some properties of this space for our purpose. 
Let $f$ be a measurable function on $\mathbb{R}^d$. The distribution function of $f$ is defined by
\begin{equation*}
d_f(\lambda):= |\{x\in \mathbb{R}^d : |f(x)|>\lambda\}|, \quad \lambda>0,
\end{equation*}
where $|A|$ is the Lebesgue measure of a set $A$ in $\mathbb{R}^d$. The decreasing rearrangement of $f$ is defined by
\begin{equation*}
f^*(s):= \inf \left\{ \lambda>0 : d_f(\lambda)\leq s\right\}, \quad s>0.
\end{equation*}

For $1\leq r<\infty$ and $1\leq \rho\leq \infty$, define the Lorentz quasi-norm 
\[\|f\|_{L^{r,\rho}}:= \left\{
\begin{array}{cl}
( \frac{\rho}{r} \int_0^\infty (s^{1/r} f^*(s))^\rho \frac{1}{s}ds)^{1/\rho} &\text{ if } \rho <\infty, \\
\sup_{s>0} s^{1/r} f^*(s) &\text{ if } \rho=\infty.
\end{array}\right.\]
The Lorentz space  $L^{r,\rho}$ denotes the space of complex-valued measurable functions $f$ on $\R^d$ such that its quasi-norm $\|f\|_{L^{r,\rho}(\R^d)}$ is finite.
From this characterization, $L^{r,\infty}(\R^d)$ is the usual weak $L^r$ space, $L^{r,r}=L^r$ and $L^{r,\rho}\subset L^{r,\tilde{\rho}}$ with $\rho<\tilde{\rho}$.

We collect the following basic properties of $L^{r,\rho}(\mathbb{R}^d)$ in the following lemmas.
\begin{lemma}[Properties of Lorentz spaces, \cite{ONeil}] \label{basicp} We have

\begin{itemize}
	
\item[$(i)$] For $1<r<\infty$, $0<\rho \leq \infty$, and $\theta>0$, $\||f|^\theta\|_{L^{r,\rho}} = \|f\|^\theta_{L^{\theta r, \theta \rho}}$.

\item[$(ii)$]  $|x|^{-2} \in L^{\frac{d}{2},\infty}(\mathbb{R}^d)$.
\end{itemize}
\end{lemma}

\begin{lemma}[H\"older's inequality, \cite{ONeil}]  There holds
\begin{itemize}
\item[$(i)$]  Let $1<r, r_1, r_2<\infty$ and $1\leq \rho, \rho_1, \rho_2 \leq \infty$ be such that
\[\frac{1}{r}=\frac{1}{r_1}+\frac{1}{r_2}, \quad \frac{1}{\rho} \leq \frac{1}{\rho_1}+\frac{1}{\rho_2}.
		\]
Then for any $f \in L^{r_1, \rho_1}(\mathbb{R}^d)$ and $g\in L^{r_2, \rho_2}(\mathbb{R}^d)$
\[\|fg\|_{L^{r,\rho}} \lesssim \|f\|_{L^{r_1, \rho_1}} \|g\|_{L^{r_2,\rho_2}}.\]
\item[$(ii)$] Let  $1<r_1,r_2<\infty $ and  $1\le \rho_1,\rho_2\le \infty $  be such that 
\begin{equation*}
1=\frac{1}{r_1}+\frac{1}{r_2},\quad 1\le \frac{1}{\rho_1}+\frac{1}{\rho_2}.
\end{equation*}
Then  for any $f \in L^{r_1, \rho_1}(\mathbb{R}^d)$ and $g\in L^{r_2, \rho_2}(\mathbb{R}^d)$
\begin{equation*}
\|fg\|_{L^1}\lesssim  \|f\|_{L^{r_1, \rho_1}} \|g\|_{L^{r_2,\rho_2}}.
\end{equation*}
\end{itemize}
\end{lemma}
\begin{lemma}[{Interpolation in Lorentz spaces}, \cite{Dao}]\label{CZ}
Let  $1<p_0<p<p_1<\infty $,  $r>0$ and $0<\theta<1$ satisfy  $\frac{1}{p}=\frac{1-\theta}{p_0}+\frac{\theta}{p_1}$. For every  $f\in L^{p_0,\infty }(\mathbb{R}^d)\cap L^{p_1,\infty }(\mathbb{R}^d)$, we have
	\begin{equation*}
		\|f\|_{L^{p,r}(\mathbb{R} ^d)}\lesssim  \|f\|_{L^{p_0,\infty }(\mathbb{R} ^d)}^{1-\theta} \|f\|_{L^{p_1,\infty }(\mathbb{R} ^d)}^{\theta}.
	\end{equation*}
\end{lemma}
\begin{lemma}[Young's inequality for weak type spaces, \cite{Grafakos}]\label{lem:weakYoung}
Let $1 \leq p < \infty$ and $1 < q, r < \infty$ satisfy
\[\frac{1}{q} + 1 = \frac{1}{p} + \frac{1}{r}.\]
Then there exists a constant $C_{p,q,r} > 0$  such that for all $f \in L^p(\R^d)$  and $g \in L^{r,\infty}(\R^d)$  we have
\[\|f*g\|_{L^{q,\infty}(\R^d)} \leq C_{p,q,r}\|g\|_{L^{r,\infty}(\R^d)}\|f\|_{L^p(\R^d)}.
\]
\end{lemma}

Finally, we record here the Strichartz estimates in Lorentz space.
\begin{lemma}[Strichartz estimates in Lorentz space, \cite{KeelTao}]\label{P:SZ} We have
\begin{itemize}
\item[$(i)$] Let $(q,r)$ be an admissible pair, then for any $f\in L^2(\mathbb{R}^d)$
\begin{align}  
\|e^{it\Delta }f\|_{L^q_t L^{r,2}_x(\mathbb{R}\times \mathbb{R}^d)} \lesssim \|f\|_{L^2_x(\mathbb{R}^d)}.\notag
\end{align}

\item[$(ii)$] Let $(q_1, r_1), (q_2,r_2)$ be two admissible pairs,  and $I\subset \mathbb{R}$ be an interval containing $t_0$. Then for any $F\in L_t^{q_2'}L^{r_2',2}_x(I\times \mathbb{R}^d)$ 
\begin{align} 
\left\|\int_{t_0}^t e^{i(t-\tau)\Delta } F(\tau) d\tau\right\|_{L^{q_1}_tL^{r_1,2}_x(I\times \mathbb{R}^d)} \lesssim \|F\|_{L^{q_2'}_tL^{r_2',2}_x(I\times \mathbb{R}^d)}.\notag
\end{align}
\end{itemize}
\end{lemma}

\section{A sequential convergence result}\label{sec-seq}
In this section, we prove a sequential convergence result, which plays an important role in the next section.

Let $u(t)$ be a  blowup solution to \eqref{1.1} with maximal lifespan $I$ and
$M(u) =  M(Q)$. For any time $t\in I$, we define the distance to the $2d+2$-dimensional manifold of symmetries acting on  $Q$ by
\begin{equation}\label{ppp}
   \delta\big(u(t)\big):=     \inf_{\substack{\lambda>0, \gamma\in [0,2\pi]\\\tilde{x},\xi\in \R^d}}
\left\|  e^{i \gamma}e^{ix\cdot\xi} \lambda^{d/2}  u \big(t, \lambda(x+\tilde{x})\big) - Q (x)  \right\|_{L_x^{2}}.
\end{equation}
In fact, we can show that  this infimum can be achieved:
\begin{lemma}\label{att}
There exist parameters $\lambda_{0} > 0, \gamma_0\in[0,2\pi], \xi_0\in\R^d$ and $x_0\in\R^d$ such that
    \begin{align*}
       \delta\big(u(t)\big)
&=\left\|  e^{i \gamma_0}e^{ix\cdot\xi_0} \lambda_0^{d/2}  u \big(t, \lambda_0 (x+x_0) \big) - Q (x)  \right\|_{L_x^{2}}.
    \end{align*}
\end{lemma}
\begin{proof}
    The proof of \cite{D2,D1} carries over mutatis mutandis. We therefore omit the details.
\end{proof}

\begin{lemma}[Upper semicontinuity of the distance function]\label{conti}
    The distance function \eqref{ppp}
   is upper semicontinuous  for any $t \in I$. Moreover, it is continuous  when it is small.
\end{lemma}
\begin{proof}
  The proof of \cite{D2,D1} carries over mutatis mutandis. We therefore omit the details.
\end{proof}
\begin{proposition}\label{sequ}
    Let $u$ be a solution to \eqref{1.1} with the mass $M(u)=M(Q)$ that blows up forward in time. Let $I$ be the maximal lifespan of $u$, then there exists a sequence $t_n\to \sup I$ such that 
    \begin{equation}\label{f2.31}
        \lim_{n\to\infty}\delta\big(u(t_n)\big)=0.
    \end{equation}
\end{proposition}
The proof of Proposition \ref{sequ} relies on the characterization of minimal mass blow up solutions with a bounded frequency function, which we have proved in our former work \cite{MMRZ}:
\begin{lemma}[\cite{MMRZ}]\label{sequencelemma}
    Let $u$ be an almost-periodic solution with mass $M(u)=M(Q)$. Suppose $u$ blows up in both time directions and the frequency function $N(t)\leq1$ on $[0,\infty)$, then there exists a sequence of times $t_n\geq0$  and parameters $N_n>0$, $\gamma_n\in[0,2\pi]$, $\xi_n\in\R^d$ and $x_n\in\R^d$ such that
    \begin{equation}\label{sequence}
       \lim_{n\to\infty} \|e^{i\gamma_n}e^{ix\cdot\xi_n}N_n^{\frac{d}{2}}u(t_n,N_n(x-x_n))-Q\|_{L^2}=0.
    \end{equation}
\end{lemma}
With the above lemma in hand, we are ready to prove Proposition \ref{sequ}:
\begin{proof}[Proof of Proposition \ref{sequ}]
 Our first step will be to argue for the existence of a solution that blows up in both time directions (and hence is almost periodic on the entire lifespan).

We let $\tau_n\to\sup I$ be an increasing sequence.  Applying Lemma~\ref{alcompact}, after passing to a subsequence, we may write 
\[
u(\tau_n)=e^{-ix\cdot\xi_n}\lambda_n^{-d/2}\phi\big(\tfrac{x-x_n}{\lambda_n}\big)+w_n=: g_n \phi + w_n
\]
for some group elements $g_n$ (corresponding to scales $\lambda_n$, Galilean parameters $\xi_n$ and translation parameters $x_n$) and functions $w_n$ satisfying $\|w_n\|_{L^2}\to 0$.

We let $v$ be the solution to \eqref{1.1} with $v|_{t=0}=\phi$.  By the $L^2$ convergence, we have $M(v)=M(Q)$.  Moreover, as $u$ blows up forward in time,
\[
\|u\|_{L_{t,x}^{\frac{2(d+2)}{d}}((\inf I,\tau_n]\times\R^d)} \to \infty \qtq{and} \|u\|_{L_{t,x}^{\frac{2(d+2)}{d}}([\tau_n,\sup I)\times\R^d)} \to \infty.
\]
Therefore, by the standard stability argument, we know that $v$ blows up in both time directions, and  thus by Lemma~\ref{alcompact}, $v$ is almost periodic on the entire lifespan $\tilde{I}$.  

We now assert the following: there exist $\{s_m\}\subset\tilde{I}$ and group elements $\{g_m\}$ such that 
\begin{align}\label{th5v26}
\delta\big(v(s_m)\big)=\delta\big(g_mv(s_m)\big)\to 0\quad\text{as}\quad m\to \infty.
\end{align}
 Assuming \eqref{th5v26} for the moment, let us complete the proof of Proposition~\ref{sequ}:

First, passing to a subsequence, we may assume that
\begin{equation}\label{equ:dis2-m}
\delta\big(g_mv(s_m)\big)\leq 2^{-m }.
\end{equation}
Set
\[
\begin{aligned}
\tilde u_{m,n}(x)
={}&\lambda_n^{d/2}
e^{i(\lambda_nx+x_n)\cdot\xi_n-i\lambda_n^2s_m|\xi_n|^2}\\
&\times u(\tau_n+\lambda_n^2s_m,
             \lambda_nx+x_n-2\lambda_n^2s_m\xi_n).
\end{aligned}
\]
By construction,
\[
\lambda_n^{d/2}e^{i(\lambda_nx+x_n)\cdot\xi_n}u(\tau_n, \lambda_n x + x_n) \to \phi(x) \qtq{in} L^2.
\]
Thus, for fixed $m$ and all sufficiently large $n$, the standard stability argument gives
\begin{align}\nonumber
\|\tilde u_{m,n}(x) - v(s_m,x)\|_{L^2} &\leq C(s_m) \|\lambda_n^{d/2}e^{i(\lambda_nx+x_n)\cdot\xi_n} u(\tau_n,\lambda_n x + x_n)-\phi(x)\|_{L^2} \\\label{equ:difuvest}
& \leq C(s_m) o_n(1) \qtq{as}n\to\infty. 
\end{align}
Moreover, as $\tau_n\to\sup I$, we can also choose  sufficiently large $n(m)$ such that
\[
\tau_{n(m)}+\lambda_{n(m)}^2 s_m\to\sup I \qtq{as} m\to\infty.
\]
By \eqref{ppp}, \eqref{equ:dis2-m}, and \eqref{equ:difuvest},
\begin{align*}
\delta\bigl(u(\tau_n+\lambda_n^2s_m)\bigr)
&=\delta(\widetilde u_{m,n})\\
&=\inf_g\|g\widetilde u_{m,n}-Q\|_{L_x^2}\\
&\leq\inf_g\|gv(s_m)-Q\|_{L_x^2}
       +\|\widetilde u_{m,n}-v(s_m)\|_{L_x^2}\\
&\lesssim2^{-m}+C(s_m)o_n(1),
\end{align*}
where the infimum is over the symmetries in \eqref{ppp}.
Thus for each $m$, there exists $n=n(m)$ such that 
\[
\delta\big(u(\tau_{n(m)}+\lambda_{n(m)}^2s_m)\big)
\lesssim 2^{-m},
\]
so that Proposition~\ref{sequ} holds with $t_m:=\tau_{n(m)}+\lambda_{n(m)}^2 s_m$.

It therefore remains to prove \eqref{th5v26}.  

To this end, we will first argue for the existence of another almost periodic solution, which has bounded frequency scale.  The construction, which utilizes the notion of the normalization of an almost periodic solution, is standard by now (see e.g. \cite{TVZ2,Visan}).  Thus we will omit the details and instead refer the reader to \cite{TVZ2,Visan} for a clear and thorough exposition.

In particular, we may construct a solution $w:\tilde{\tilde{I}}\times\R^d\to\C$, $[0,\infty)\subset \tilde{\tilde{I}}$ to \eqref{1.1} such that
\[
w(0) = \lim_{n\to\infty} \tilde g_n v(r_n)
\]
for some group elements $\tilde g_n$ and some sequence $\{r_n\}\subset\R$.  This solution satisfies $M(w)=M(Q)$ and is almost periodic on the entire lifespan, with some spatial center $x(t)$, Galilean boost $\xi(t)$ and frequency scale $N(t)$ satisfying $\sup_{t\geq0}N(t)\leq 1$.  
 
Arguing as we did above, we now observe that to prove \eqref{th5v26} it suffices to show that
\begin{equation}\label{th5v262}
\qtq{there exists}\{s_m^{\prime}\}\subset\R\qtq{such that}
\lim_{m\to\infty}\delta\big(w(s_m')\big)= 0
\end{equation}
In view of Lemma \ref{sequencelemma}, \eqref{th5v262} holds naturally and we can conclude the proof.
\end{proof}



\section{Reduction of the proof}\label{reductionsec}
In this section, we show that  Theorem \ref{main1} follows  by a rigidity result, which states as
\begin{theorem}\label{t2.3}
Let $0 < \eta_{\ast} \ll 1$ be a small, fixed constant to be defined later. Suppose  $u$ is a  solution to \eqref{1.1} on the maximal interval of existence $I \subset \mathbb{R}$, with $ M(u) =  M(Q)$. If $u$ blows up forward in time, and satisfies 
\begin{equation}\label{f2.32}
\sup_{t \in [0, \sup(I))}\delta\big(u(t)\big) \leq \eta_{\ast},
\end{equation}
then $u$ is either a soliton solution of the form \eqref{soliton} or a pseudo-conformal transformation of a soliton of the form \eqref{pseudosoliton}.
\end{theorem}

\begin{proof}[Theorem \ref{t2.3} $\Rightarrow $ Theorem \ref{main1}]
The proof mainly follows Dodson \cite{D2,D1}.  For the sake of completeness, we present a detailed proof here.  
Let $u(t)$ be the solution to \eqref{1.1} satisfying the assumption of Theorem \ref{main1}. Without loss of generality, we assume that $u$ blows up forward in time.  We divide into two cases:

\textbf{Case I}. There exists some $t_{0} > 0$ such that
\begin{equation*}
\sup_{t \in [t_{0}, \sup(I))} \delta\big(u(t)\big)  \leq \eta_{\ast},
\end{equation*}
where $\eta_{\ast}$ is the small parameter in Theorem \ref{t2.3}. In this case, after translating in time so that $t_{0} = 0$, \eqref{f2.32} holds, thus we can directly  use Theorem \ref{t2.3} to prove Theorem \ref{main1}.

\textbf{Case II}.  There exists a sequence $t_{n}^{-} \nearrow \sup(I)$ such that
\begin{equation}\label{f2.35}
 \delta\big(u(t_n^-)\big)   > \eta_{\ast},
\end{equation}
for every $n$. In this case, after passing to a subsequence, we can suppose that for every $n$, $t_{n}^{-} < t_{n} < t_{n + 1}^{-}$, where $t_{n}$ is the sequence in \eqref{f2.31} and $t_{n}^{-}$ is the sequence in \eqref{f2.35}. The fact that
$\delta\big(u(t)\big)$
is upper semicontinuous as a function of $t$, and  continuous when it is small, guarantees that there exists a small, fixed $0 < \eta_{\ast} \ll 1$ such that the sequence $t_{n}^{+}$, defined by
\begin{equation*} 
t_{n}^{+} = \inf \bigg\{ t_0 \in I : \sup_{t \in [t_{0}, t_n)} \delta\big(u(t)\big) 
< \eta_{\ast}  \bigg\},
\end{equation*}
satisfies $t_{n}^{+} \nearrow \sup I$ as $n\to \infty $ and
\begin{equation}\label{f2.38}
\delta\big(u(t_n^+)\big) 
= \eta_{\ast}.
\end{equation}
Indeed, the fact that $\delta\big(u(t)\big) $ is upper semicontinuous as a function of $t$ implies that
\begin{equation*}
\bigg\{ 0 \leq t < t_{n} : \delta\big(u(t)\big)  \geq \eta_{\ast}  \bigg\}
\end{equation*}
is a closed set. Since this set is also bounded, it has a maximal element $t_{n}^{+}$, and $t_{n}^{+} \geq t_{n}^{-}$. The fact that $\delta\big(u(t)\big) $ is upper semicontinuous in time also implies that
\begin{equation}\label{edc0}
 \delta\big(u(t_n^+)\big)\geq \eta_\ast.
\end{equation}
On the other hand, since
\begin{equation}\label{edc}
\delta\big(u(t)\big) 
 < \eta_{\ast} \quad \text{for all} \quad t_{n}^{+} < t < t_{n},
\end{equation}
Lemma~\ref{conti} gives continuity near $t_n^+$ for sufficiently
small $\eta_\ast$. Letting $t\searrow t_n^+$ in \eqref{edc}, we obtain
\begin{equation}\label{f2.41}
 \delta\big(u(t_n^+)\big)   \leq \eta_{\ast}.
\end{equation}
By \eqref{edc0} and \eqref{f2.41}, we have
\begin{equation}\label{zzzx}
	\delta\big(u(t_n^+)\big) 
 = \eta_{\ast}.
\end{equation}

We then claim that 
\begin{equation}\label{scatteringnorm}
	\lim_{n \rightarrow \infty}\|u\|_{L_{t,x}^{\frac{2(d+2)}{d}}( \left[t_{n}^{+}, t_{n} \right] \times \mathbb{R}^d)} = \infty.
\end{equation}
In fact,  suppose that \eqref{scatteringnorm} fails, then, after passing to a subsequence, there exists a constant $M$ such that 
    \[
    \|u\|_{L_{t,x}^{\frac{2(d+2)}{d}}([t_{n}^{+}, t_{n}]\times \R^d)}\leq M \qtq{for all} n\geq0.
    \]
    By Lemma \ref{att} and Proposition \ref{sequ}, there exist parameters
$\lambda_n>0$, $\gamma_n$, $\xi_n$, and $x_n$ such that
\[
\varepsilon_{n,0}(x)
=e^{i\gamma_n}e^{ix\cdot\xi_n}\lambda_n^{d/2}
u\bigl(t_n,\lambda_n(x+x_n)\bigr)-Q(x),
\qquad
\|\varepsilon_{n,0}\|_{L_x^2}
=\delta\bigl(u(t_n)\bigr)\longrightarrow0.
\]
Set $u_n(t,x)=u(t_n+t,x)$ and define the exact solitary wave
\begin{align*}
\widetilde Q_n(t,x)
={}&\lambda_n^{-d/2}e^{-i\gamma_n}
e^{-i(x/\lambda_n-x_n)\cdot\xi_n}
e^{i(1-|\xi_n|^2)t/\lambda_n^2}\\
&\times Q\left(\frac{x}{\lambda_n}-x_n
                       +\frac{2t\xi_n}{\lambda_n^2}\right).
\end{align*}
Let $\varepsilon_n(t,x)=u_n(t,x)-\widetilde Q_n(t,x)$. Then
\[
\varepsilon_n(0,x)
=\lambda_n^{-d/2}e^{-i\gamma_n}
e^{-i(x/\lambda_n-x_n)\cdot\xi_n}
\varepsilon_{n,0}(x/\lambda_n-x_n),
\qquad
\|\varepsilon_n(0)\|_{L_x^2}=\|\varepsilon_{n,0}\|_{L_x^2},
\]
and
\[
i\partial_t\varepsilon_n+\Delta\varepsilon_n
=-F(u_n)+F(\widetilde Q_n).
\]
Since $\widetilde Q_n=u_n-\varepsilon_n$, we have
\begin{align*}
|F(u_n)-F(\widetilde Q_n)|
\lesssim{}&|\varepsilon_n|
\bigl[|x|^{-2}*|u_n|^2+|x|^{-2}*|\varepsilon_n|^2\bigr]\\
&+\bigl(|x|^{-2}*|\varepsilon_nu_n|\bigr)
\bigl(|u_n|+|\varepsilon_n|\bigr).
\end{align*}
For $J\subset[t_n^+-t_n,0]$, Strichartz estimates and the
Hardy--Littlewood--Sobolev inequality give
\begin{align}\label{reduction:perturbation}
&\|\varepsilon_n\|_{L_t^\infty L_x^2(J\times\R^d)}
+\|\varepsilon_n\|_{L_t^2L_x^{\frac{2d}{d-2}}(J\times\R^d)}\notag\\
&\quad\lesssim\inf_{t\in J}\|\varepsilon_n(t)\|_{L_x^2}
+\|\varepsilon_n\|_{L_t^6L_x^{\frac{6d}{3d-2}}(J\times\R^d)}^3\notag\\
&\qquad+\|u_n\|_{L_t^6L_x^{\frac{6d}{3d-2}}(J\times\R^d)}
\|\varepsilon_n\|_{L_t^6L_x^{\frac{6d}{3d-2}}(J\times\R^d)}^2\notag\\
&\qquad+\|u_n\|_{L_t^6L_x^{\frac{6d}{3d-2}}(J\times\R^d)}^2
\|\varepsilon_n\|_{L_t^6L_x^{\frac{6d}{3d-2}}(J\times\R^d)}.
\end{align}
By local theory and the assumed spacetime bound,
$\|u_n\|_{S^0([t_n^+-t_n,0])}\lesssim_M1$. Partitioning
$[t_n^+-t_n,0]$ into finitely many intervals on which
$\|u_n\|_{L_t^6L_x^{6d/(3d-2)}}$ is small, and applying
\eqref{reduction:perturbation} successively backward from $t=0$, we obtain
\[
\|\varepsilon_n\|_{L_t^\infty L_x^2([t_n^+-t_n,0]\times\R^d)}
\lesssim_M\|\varepsilon_n(0)\|_{L_x^2}\longrightarrow0.
\]
Since $\widetilde Q_n(t_n^+-t_n)$ lies on the ground-state orbit,
\[
\delta\bigl(u(t_n^+)\bigr)
\leq\|\varepsilon_n(t_n^+-t_n)\|_{L_x^2}\longrightarrow0,
\]
contradicting \eqref{zzzx}.

Using the fact that $u$ is a minimal-mass blowup solution, there exist $w_0\in L_x^2$ and parameters $\widehat{\lambda}_n > 0$, $\widehat{\gamma}_n\in[0,2\pi]$, $\widehat{x}_n\in\mathbb{R}^d$ and $\widehat{\xi}_n\in\mathbb{R}^d$, 
 such that
\begin{equation}\label{f2.43}
\left\|e^{i\widehat{\gamma}_{n }}e^{ix\cdot\widehat{ \xi}_n}\widehat{\lambda}_n^{d/2}  u \left(t_{n}^{+}, \widehat{\lambda}_n
\left(x+ \widehat{x}_n \right) \right)-w_{0} (x) \right\|^2_{L^2_x} \rightarrow 0, \text{ as } n \to \infty .
\end{equation}
Moreover,   $ \| w_0 \|_{L_x^{2}} = \| Q\|_{L_x^{2}}$, and by \eqref{f2.38},
\begin{equation}\label{f2.44}
\inf_{\substack{\lambda>0, \gamma\in [0,2\pi]\\\tilde{x},\xi\in \R^d}}
\left\|  e^{i \gamma}e^{ix\cdot\xi} \lambda^{d/2}  w_0 \left(\lambda  \left(x+\tilde{x} \right) \right) - Q (x)  \right\|_{L_x^{2}} =\eta_{\ast}.
\end{equation}
Let $w$ be the solution to \eqref{1.1} with initial data $w_{0}$. Then the fact that $u$ blows up forward in time and $t_n^+\nearrow \sup I$ as $n \to \infty$ implies that $w$ blows up both forward and backward in time. Denote the maximal interval of existence of $w$ as $\tilde{I}$. We claim that
\begin{equation}\label{f2.45}
\delta\big(w(t)\big)\leq \eta_{\ast} ,\quad \forall \,  t\in [0, \sup\tilde{I}).
\end{equation}
Indeed, for any compact interval $[0, M]\subseteq  \tilde{I}$, \eqref{f2.43} and the stability lemma guarantee that as $n \to \infty$, 
\begin{align}\label{ijz}
\|u'_n(t)-w(t) \|_{L_t^{\infty} L_x^2\cap L_{t, x}^{\frac{2(d+2)}{d}}  ([0,M]\times\mathbb{R}^d)}\to 0,
\end{align}
where
\[
\begin{aligned}
u'_n(t,x)={}&e^{i\widehat\gamma_n}
e^{ix\cdot\widehat\xi_n-it|\widehat\xi_n|^2}
\widehat\lambda_n^{d/2}\\
&\times u\left(t_n^++\widehat\lambda_n^2t,
\widehat\lambda_n(x+\widehat x_n-2t\widehat\xi_n)\right).
\end{aligned}
\]
Clearly, \eqref{scatteringnorm}, \eqref{f2.43}, \eqref{ijz}, and the arbitrariness in the choice of $M$ 
imply that \eqref{f2.45} holds.

However, Theorem \ref{t2.3}, \eqref{f2.44}, and \eqref{f2.45} imply that $w$ must be of the form \eqref{pseudosoliton}. Such a solution scatters backward in time, which contradicts the fact that $w$ blows up both forward and backward in time. Thus \eqref{f2.35} cannot happen, which precludes possibility of Case II. So we finish the proof.
\end{proof}

\section{Long-time Strichartz estimates}\label{sec-long}
This section aims to demonstrate that the high-frequency truncation of the solution exhibits excellent decay properties, meaning that, in a certain sense, the high-frequency part of the solution closely approximates the free solution. Such estimates, known as long-time Strichartz estimates, were first introduced by Dodson \cite{D7,D6,D66,D5}. In our previous work \cite{MMRZ}, we developed bilinear estimates that account for \emph{non-local} effect of the Hartree nonlinearity and employed Lorentz spaces to obtain long-time Strichartz estimates. In this paper, we improve the long-time Strichartz estimates for solutions near the ground state orbit by utilizing local smoothing estimates.
First, we introduce some useful norms.
\begin{definition}[$U_\Delta^p(I)$ space,\cite{HHK}] 
Let $1\leqslant p<\infty$. We define $U_\Delta^p(\R)$ as an atomic space with atoms $v^\lambda(t,x)$ given by 
\begin{align*}
v^\lambda(t,x)=\sum_{k=0}^{M}\chi_{[t_k,t_{k+1})}(t)e^{it\Delta} v_k^\lambda(x),\quad
\sum_{k=0}^M\left\| v_k^\lambda \right\|_{L^2(\R^d)}^p=1.
\end{align*}
In the summation above, $M$ can be finite or infinite. If $M$ is finite, we further assume $t_0=-\infty$ and $t_{M+1}=+\infty$. The norm of $U_\Delta^p (\R)$ is defined as 
\begin{align*}
\| v \|_{U_\Delta^p(\R)}
= \inf\left\{\sum_{\lambda\in\mathbb{Z}} |c_\lambda|: v =\sum_{\lambda\in\mathbb{Z}}c_\lambda v^\lambda, \hspace{1ex} v^\lambda\mbox{ is an atom}\right\},
\end{align*}
where the infimum is taken over all atom decompositions. 
For any interval $I\subseteq\R$, we can define the local version as
\begin{align*}
\| v\|_{U_\Delta^p \left(I \right)}=\| v \chi_I(t)\|_{U_\Delta^p \left(\R\right)},
\end{align*}
where $\chi_I$ denotes the characteristic function of the set $I$.
\end{definition}

Next, we define the space $V_\Delta^p \left(I \right)$, which is also useful for our analysis.

\begin{definition}[$V_\Delta^p \left(I \right)$ space,\cite{HHK}]
Let $1\leqslant p<\infty$. We define $V_\Delta^p(\R )$ as the space of right-continuous functions $v \in L_t^{\infty} L_x^2(\R\times\R^d)$ endowed with the norm
\begin{align*}
\| v \|_{V_\Delta^p \left(\R\right)}^p
= \| v \|_{L_t^\infty L_x^2(\R\times\R^d)}^p + \sup_{\{t_k\}\nearrow\infty}
\sum_{k} \big\|e^{-it_{k+1}\Delta} v(t_{k+1})-e^{- it_k\Delta} v(t_k)\big\|_{L^2(\R^d)}^p.
\end{align*}
For any interval $I\subseteq\R$, we can define the local version as
\begin{align*}
\| v \|_{V_\Delta^p \left(I \right)}=\| v \chi_I(t)\|_{V_\Delta^p \left(\R\right)}.
\end{align*}
\end{definition}
By definition, it's easy to check that for $0<\theta<1$ and $p, p_2>1$  satisfying $\frac{1}{p}=\theta+\frac{1-\theta}{p_2}$, we have
\begin{align}\label{interpo}
   \Big\| \int_{t_0}^t e^{i(t-\tau)\Delta} F(\tau) \, {d} \tau \Big\|_{V_{\Delta}^{p}}\leq \|F\|_{L_t^1L_x^2}^{\theta}\Big\|\int_{t_0}^t e^{i(t-\tau)\Delta} F(\tau) \, {d} \tau \Big\|^{1-\theta}_{V_{\Delta}^{p_2}}.
\end{align}

Next, we present some useful properties of the $U_\Delta^p \left(I \right)$ and $V_\Delta^p \left(I\right)$ spaces.

\begin{lemma}[\cite{HHK}]\label{property}
Suppose $1<p<q<\infty$ and $t_0\leqslant t_1\leqslant t_2$. For $I\subset \R$, we have
\begin{align}
(i)&\; U_\Delta^p \left(I \right)\subseteq V_\Delta^p \left(I \right)\subseteq U_\Delta^q \left(I \right),\label{embeddingvp} \\
(ii)&\;\| v \|_{U_\Delta^p \left([t_0,t_1] \right)}\leqslant\| v \|_{U_\Delta^p \left([t_0,t_2] \right)},  \notag \\
(iii)&\;\| v \|_{U_\Delta^2 \left([t_0,t_2] \right)}^2\leqslant\| v \|_{U_\Delta^2 \left([t_0,t_1] \right)}^2
+\| v \|_{U_\Delta^2 \left([t_1,t_2]\right)}^2,  \notag\\
(iv)&\;\| v \|_{L_t^p L_x^q \left(I\times\R^d \right)}\lesssim_p \| v \|_{U_\Delta^p \left(I \right)},\;(p,q)\in\Lambda_0 \label{u-pstrichartz}\\
(v)&\; \| v \|_{U_\Delta^2 \left(I \right)}\lesssim_p \inf_{\tilde{t}\in I} \left\| v(\tilde{t}) \right\|_{L^2(\R^d)}
+\|i\partial_t v +\Delta v \|_{L_t^{p'} L_x^{q'} \left(I\times\R^d \right)},\;(p,q)\in\Lambda_0. \label{dupstrichartz}
\end{align}
\end{lemma}
\begin{remark}\label{lorentzup}
    Furthermore, by using Lemma \ref{P:SZ}, one can easily verify that the stronger embedding
    \[
    \| v \|_{L_t^p L_x^{q,2} \left(I\times\R^d \right)}\lesssim_p \| v \|_{U_\Delta^p \left(I \right)},
    \]
still holds.
\end{remark}
Furthermore, we have the following estimates:

\begin{lemma}[\cite{HHK}]\label{up-dual}
Let $p>1$ and let $I$ be an interval with $t_0\in I$, we have
\begin{align*}
\left\|\int_{t_0}^t e^{i(t-\tau)\Delta} F(\tau) \, {d} \tau\right\|_{U_\Delta^p \left(I\right)}
\lesssim_p \sup_{\| G \|_{V_\Delta^{p'} \left(I \right)}\leq1} \Big|\int_I\int_{\R^d}F\bar{G}\;dx\;dt\Big|.	
\end{align*}
\end{lemma}

We next prove the following lemma, which plays a key role in the proof of long-time Strichartz estimate.
\begin{lemma}\label{atominterpolation}
For any $0<\delta<1$,  we have
 \begin{align}\nonumber
   & \Big \|\int_{a}^te^{i(t-s)\Delta}F(s)\; ds\Big\|_{U_{\Delta}^2([a,b])}\\ \label{rrrt}  \lesssim_{\delta}&\|F\|^{\delta}_{L_t^{1}L_{x}^{2}}\bigg\{\sup_{\|v\|_{U_{\Delta}^2([a,b])}\leq 1}\Big|\int_a^b\int_{\R^d}F\bar{v}\;dx\;dt\Big|\bigg\}^{1-\delta}.
 \end{align}
\end{lemma}
\begin{proof}
Fix $\delta\in(0,1)$, we can find $p_1\in(1,2)$ and $p_2>2$ such that
$$\frac{1}{p_1}=\delta+\frac{1-\delta}{p_2}.$$
Using \eqref{embeddingvp}, \eqref{interpo}, Lemma \ref{up-dual}, we obtain
\begin{align*}
\Big \|\int_{a}^te^{i(t-s)\Delta}F(s)\,ds\Big\|_{U_{\Delta}^2([a,b])}
\lesssim&\Big \|\int_{a}^te^{i(t-s)\Delta}F(s)\,ds\Big\|_{V_{\Delta}^{p_1}([a,b])}\\
\lesssim&\|F\|_{L_t^1L_x^2}^{\delta}\Big\|\int_a^t e^{i(t-\tau)\Delta} F(\tau) \, {d} \tau \Big\|^{1-\delta}_{V_{\Delta}^{p_2}}\\
\lesssim&\|F\|_{L_t^1L_x^2}^{\delta}\Big\|\int_a^t e^{i(t-\tau)\Delta} F(\tau) \, {d} \tau \Big\|^{1-\delta}_{U_{\Delta}^{p_2}}\\
\lesssim&\|F\|_{L_t^1L_x^2}^{\delta}\bigg\{\sup_{\| G \|_{V_\Delta^{p_2'} \left([a,b] \right)}=1} \Big|\int_a^b\int_{\R^d}F\bar{G}\;dx\;dt\Big|\bigg\}^{1-\delta}\\
\lesssim&\|F\|_{L_t^1L_x^2}^{\delta}\bigg\{\sup_{\|v\|_{U_{\Delta}^2([a,b])}\leq 1}\Big|\int_a^b\int_{\R^d}F\bar{v}\;dx\;dt\Big|\bigg\}^{1-\delta},
\end{align*}
and so \eqref{rrrt}   follows.

\end{proof}

With the above preparations, we now introduce the following long-time Strichartz estimate:
\begin{theorem}[Long-time Strichartz estimate]\label{longtime}
Let $u$ be the  solution as in Theorem \ref{t2.3}. Let $J = [a,b]$ be an interval for which
$\lambda(t)\geq1$ and
\begin{equation}\label{f5.7}
    \begin{cases}
     \int_{a}^b \lambda(t)^{-2} \, {d} t = T=2^{3k_0},\;   \int_{a}^b\frac{1}{\lambda(t)^2}\|\epsilon(t)\|_{L_x^2}^2\, {d} t\leq \eta_1,\\
     \sup\limits_{t \in J}\frac{|\xi(t)|}{\lambda(t)}\leq \eta_0,\ \sup\limits_{t\in J}|x(t)|\leq T^{\frac{1}{25}},
    \end{cases}
\end{equation}
for some $\eta_0,\eta_1\ll1$. Then the long-time Strichartz estimate,
\begin{align}\label{f5.66}
\|P_{\geq k_0}u\|_{U_{\Delta}^2([a,b])}\lesssim \Big(\frac{1}{T}\int_{a}^{b}\frac{1}{\lambda(t)^2}\|\epsilon(t)\|^2_{L_x^2}\, dt\Big)^{1/2}+\frac{1}{T^{10}},
\end{align}
holds (universally) for sufficiently small $\eta_{\ast}\ll1$ and large $T\gg1$ and the implicit constant is independent of $T$.
\end{theorem}

This theorem plays a key role in our whole argument to prove  Theorem \ref{t2.3}. We will  prove it by  induction on the frequency, which is based on the following lemma: 
\begin{lemma}[Iterative procedure]\label{integ1} Under the assumptions of Theorem \ref{longtime}, for any $i\geq [\frac{\log_2(T)}{4}]+5$ and $0<\delta<1$, we have
\begin{align}\label{finalzz}
 &\|P_{\geq i}u\|_{U_{\Delta}^2([a,b])}\notag\\
 \lesssim_{\delta}& \inf_{t\in[a,b]}\|\epsilon(t)\|_{L_x^2(\R^d)}
+\frac{1}{T^{20}} \\\nonumber&+\big[{\eta_1^{\frac{1}{2}-}}+\eta_1^{1-}+(T^{\frac{d}{d+2}\delta}+T^{\frac{d+4}{2(d+2)}\delta})(2^{-i}T^{1/25})^{1-\delta}\big]\|P_{\geq i-4}u\|_{U_{\Delta}^2\left([a,b]\right)}
\end{align}
for sufficiently large $T.$
\end{lemma}
The proof of the above lemma is comparatively complicated and technical. Therefore, let us first take it for granted and then proceed to prove the long-time Strichartz estimates.
\begin{proof}[Proof of Theorem \ref{longtime}]
    We observe that by the standard local theory,
    \begin{equation}\label{equ:uU2est}
        \|u\|_{U_{\Delta}^2([a,b])}\lesssim T^{\frac12}.
    \end{equation}
By \eqref{f5.7} and integral mean value theorem, we get
$$\inf_{t\in[a,b]}\|\epsilon(t)\|_{L_x^2(\R^d)}
\lesssim \Big(\frac{1}{T}\int_{a}^{b}\frac{1}{\lambda(t)^2}\|\epsilon(t)\|^2_{L_x^2}\, dt\Big)^{1/2}.$$
Hence, using  Lemma \ref{integ1}, we obtain
\begin{align*}
   \|P_{\geq[\frac{\log_2(T)}{4}]+5}u\|_{U_{\Delta}^2([a,b])}\lesssim_{\delta}& \Big(\frac{1}{T}\int_{a}^{b}\frac{1}{\lambda(t)^2}\|\epsilon(t)\|^2_{L_x^2}\, dt\Big)^{1/2}+\frac{1}{T^{20}}
   \\&+\big({\eta_1^{1-}}+\eta_1^{\frac{1}{2}-}+T^{\frac{d}{d+2}\delta-\frac{21(1-\delta)}{100}}+T^{\frac{d+4}{2(d+2)}\delta-\frac{21(1-\delta)}{100}}\big)T^{\frac12}.
   \end{align*}
Letting $A:=C_\delta({\eta_1^{1-}}+\eta_1^{\frac{1}{2}-}+T^{\frac{d}{d+2}\delta-\frac{21(1-\delta)}{100}}+T^{\frac{d+4}{2(d+2)}\delta-\frac{21(1-\delta)}{100}})$, with $C_\delta$ large enough to include the constant in Lemma~\ref{integ1}, we can iterate $j-1$ times to show that
\[
   \|P_{\geq[\frac{\log_2(T)}{4}]+1+4j}u\|_{U_{\Delta}^2([a,b])}\lesssim_{\delta} 
   A^{j}T^{\frac12}+
 \Big[\Big(\frac{1}{T}\int_{a}^{b}\frac{1}{\lambda(t)^2}\|\epsilon(t)\|^2_{L_x^2}\, dt\Big)^{1/2}
 +\frac1{T^{20}}\Big]  \sum_{k=0}^jA^k.
\]
Choose $\delta>0$ sufficiently small, then $\eta_1$ sufficiently
small and $T$ sufficiently large so that $A\ll1$. Taking
$j\sim\log_2(T)/48$ with $\lfloor\log_2(T)/4\rfloor+1+4j\leq k_0$,
we have $A^j\leq T^{-22}$, provided $A$ is sufficiently small.
Thus
\begin{align}\label{zzbvz}
\|P_{\geq k_0}u\|_{U_\Delta^2([a,b])}
\lesssim{}&
\left(\frac1T\int_a^b\lambda(t)^{-2}
\|\epsilon(t)\|_{L_x^2}^2\,dt\right)^{1/2}
+T^{-20}+A^jT^{1/2}.
\end{align}
Since $A^j\leq T^{-22}$, the last term is bounded by $T^{-20}$,
which together with \eqref{zzbvz} yields
\[
\|P_{\geq\frac{\log_2(T)}{3}}u\|_{U_{\Delta}^2([a,b])}\lesssim \Big(\frac{1}{T}\int_{a}^{b}\frac{1}{\lambda(t)^2}\|\epsilon(t)\|^2_{L_x^2}\, dt\Big)^{1/2}+\frac{1}{T^{10}},
\]
as desired.
\end{proof}
\begin{remark}\label{lorentzlongtime}
  From Remark \ref{lorentzup}, we know that Theorem \ref{longtime} also implies that
   \[
   \|P_{\geq k_0}u\|_{L_t^2L_x^{\frac{2d}{d-2},2}([a,b]\times\R^d)}\lesssim \Big(\frac{1}{T}\int_{a}^{b}\frac{1}{\lambda(t)^2}\|\epsilon(t)\|^2_{L_x^2}\, dt\Big)^{1/2}+\frac{1}{T^{10}}.
   \]
\end{remark}
The rest of this section is devoted to the proof of Lemma \ref{integ1}. The proof mainly relies on the rapid decay of $Q$ and the square decay of $|x|^{-2}\ast|Q|^2$, which allows us to use local smoothing estimates to complete the frequency iteration. We emphasize that this method heavily relies on modulation analysis. In contrast, \cite{MMRZ} requires establishing refined bilinear estimates to overcome the non-local effect.

\begin{proof}[Proof of Lemma \ref{integ1}]
For the sake of brevity, we will omit the integration domain in each formula from this point onward, except for the first line.
By Duhamel's formula and Strichartz estimate, we have for any $t_0\in [a,b]$
\[
\|P_{\geq i}u\|_{U^2_{\Delta}([a,b])}\lesssim \|P_{\geq i} u(t_0)\|_{L_x^2}+\Big\|\int_{t_0}^te^{i(t-s)\Delta} P_{\geq i}[(|x|^{-2}*|u|^2)u]\ ds\Big\|_{U^2_{\Delta}([a,b])}.
\]
From now on we stipulate that if $(p,q)\in\Lambda_0$, then the chosen $(p_{\pm}, q_{\mp})$ also belongs to $\Lambda_0$. In view of  \eqref{dupstrichartz} and \eqref{rrrt}, for any $t_0$, we obtain
\begin{align}\notag
&\|P_{\geq i}u\|_{U^2_{\Delta}([a,b])}\\\label{equ:lintermpu}
\lesssim& \|P_{\geq i} u(t_0)\|_{L_x^2}\\\label{equ:inhterm1}
&+\Big\|\int_{t_0}^te^{i(t-s)\Delta}P_{\geq i}\big[(|x|^{-2}*|u|^2)P_{\geq i-4}u\big]\ ds\Big\|_{U^2_{\Delta}([a,b])}\\\label{equ:inhterm2}
&+\Big\|\int_{t_0}^te^{i(t-s)\Delta}P_{\geq i}\Big(\big[|x|^{-2}*\text{\O}(\bar{u}P_{\geq i-4}u)\big]P_{\leq i-4}u\Big)\ ds\Big\|_{U^2_{\Delta}([a,b])}
\end{align}
where the expression $\text{\O}(X)$ denotes a finite linear combination of terms
 that resemble $X$ up to  complex conjugation.

{\bf The estimate of \eqref{equ:lintermpu}:} Using the fact that $\lambda(t)\geq 1$, Proposition \ref{grondstate2} and Bernstein's inequality, we see that
 \begin{align}\label{vvv1}
\|P_{\geq i} u(t_0)\|_{L_x^2}&\leq \|\epsilon(t_0)\|_{L_x^2}+\|P_{\geq i}\tilde{Q}\|_{L^2}\notag\\
     &\leq \|\epsilon(t_0)\|_{L_x^2}+\|P_{\geq i-1} Q\|_{L^2}\notag\\
     &\leq  \|\epsilon(t_0)\|_{L_x^2}+\frac{1}{2^{1000i}}\||\nabla|^{1000}Q\|_{L^{2}}\notag\\
&\leq\|\epsilon(t_0)\|_{L_x^2}+\frac{1}{T^{20}},
 \end{align}
 which is acceptable.

{\bf The estimate of \eqref{equ:inhterm1}:} By \eqref{decomp}, we have
\begin{align}\notag
    u(t,x)=&e^{-i\gamma(t)}e^{-i\frac{x-x(t)}{\lambda(t)}\cdot\xi(t)}\lambda(t)^{-\frac{d}{2}}
    \Big(\epsilon\big(t,\tfrac{x-x(t)}{\lambda(t)}\big)+Q\big( \tfrac{x-x(t)}{\lambda(t)}\big)\Big)\\\label{equ:uepsQ}
    \triangleq&\tilde{\epsilon}(t,x)+\tilde{Q}(t,x).
\end{align}
This together with Strichartz estimate and Lemma \ref{atominterpolation} implies
\begin{align}\notag
\eqref{equ:inhterm1}\lesssim&\Big\|\int_{t_0}^te^{i(t-s)\Delta}P_{\geq i}\big[(|x|^{-2}*|\tilde{\epsilon}|^2)P_{\geq i-4}u\big]\ ds\Big\|_{U^2_{\Delta}}\\\notag
&+\Big\|\int_{t_0}^te^{i(t-s)\Delta}P_{\geq i}\big[(|x|^{-2}*|\tilde{Q}|^2)P_{\geq i-4}u\big]\ ds\Big\|_{U^2_{\Delta}}\\\label{equ:x-2epsest}
\lesssim&\|(|x|^{-2}*|\tilde{\epsilon}|^2)P_{\geq i-4}u\|_{L_t^{1}L_x^{2}}\\\notag
&+\|(|x|^{-2}*|\tilde{Q}|^2)P_{\geq i-4}u\|^{\delta}_{L_t^{1}L_x^{2}}\notag\\
&\qquad\times\sup_{\|v\|_{U_{\Delta}^{2}}\leq 1}\|(|x|^{-2}*|\tilde{Q}|^2)P_{\geq i-4}uP_{\geq i}v\|^{1-\delta}_{L_{t,x}^{1}}.
\end{align}

Next, using \eqref{reduction:perturbation} and the assumption
\eqref{f5.7}, we prove that
\begin{equation}\label{equ:tildeespest}
    \|\tilde{\epsilon}\|_{L_t^2L_x^{\frac{2d}{d-2}}([a,b]\times\R^d)}\lesssim \Big(\int_{a}^b\frac{1}{\lambda(t)^2}\|\epsilon(t)\|_{L_x^2}^2\, dt\Big)^\frac12\lesssim \eta_1^\frac12.
\end{equation}
In fact, we can divide the interval $[a,b]$ into the subintervals $J_l$ satisfying
\[
\int_{J_l}\frac{1}{\lambda(t)^2}\,dt=1\qtq{for all} 0\leq l\leq 2^{3k_0}-1.
\]
By \eqref{f4.45} and \eqref{f4.56},
\[
\sup_{t\in J_i}\lambda(t)\sim\inf_{t\in J_i}\lambda(t),\qquad
\sup_{t\in J_i}\|\epsilon(t)\|_{L_x^2}
\sim\inf_{t\in J_i}\|\epsilon(t)\|_{L_x^2}.
\]
Moreover,
 $\tilde{Q}$ satisfies the equation 
\[
i\partial_t \tilde{Q}+\Delta\tilde{Q}=-(|x|^{-2}\ast|\tilde{Q}|^2)\tilde{Q}+\mathcal{G}
\]
where the term $\mathcal{G}$ is given by
\[
\begin{aligned}
\mathcal{G}={}&\lambda^{-d/2}e^{-i\gamma-i\varphi}\\
&\quad\times\Bigg\{\left(\gamma_t-\frac{x_t\cdot\xi}{\lambda}
+y\cdot\Big(\xi_t-\xi\frac{\lambda_t}{\lambda}\Big)
+\frac{1-|\xi|^2}{\lambda^2}\right)Q(y)\\
&\qquad+i\left(-\frac d2\frac{\lambda_t}{\lambda}Q(y)
-\frac{x_t}{\lambda}\cdot\nabla Q(y)
-\frac{\lambda_t}{\lambda}y\cdot\nabla Q(y)
-2\frac{\xi\cdot\nabla Q(y)}{\lambda^2}\right)\Bigg\}.
\end{aligned}
\]
with
\[
y = \frac{x - x(t)}{\lambda(t)}, \qquad \varphi = y\cdot\xi(t).
\]
Applying Lemma~\ref{P:SZ} to the equation for $\widetilde\epsilon$
and estimating the nonlinear terms as in \eqref{reduction:perturbation},
we use \eqref{f4.45}, \eqref{f4.53}--\eqref{f4.55} to obtain
   \[
   \Big\|\frac{1}{\lambda(t)^{d/2}}\epsilon \Big(t,\frac{x-x(t)}{\lambda(t)}\Big)\Big\|^2_{L_t^2L_x^{\frac{2d}{d-2},2}(J_i\times\R^d)}\lesssim \inf_{t\in J_i} \|\epsilon(t)\|_{L_x^{2}}^2\lesssim \int_{J_i}\frac{1}{\lambda(t)^2}\|\epsilon(t)\|^2_{L_x^2}\, dt.
   \]
Summing over $i$, we readily obtain
\begin{equation}\label{estepsilonnorm}
   \Big\|\frac{1}{\lambda(t)^{d/2}}\epsilon\Big(t,\frac{x-x(t)}{\lambda(t)}\Big)\Big\|^2_{L_t^2L_x^{\frac{2d}{d-2},2}([a,b]\times\R^d)}\lesssim  \int_{a}^b\frac{1}{\lambda(t)^2}\|\epsilon(t)\|^2_{L_x^2}\, dt\leq\eta_1,
\end{equation}
as desired. Hence
\begin{align*}
    \|(|x|^{-2}*|\tilde{\epsilon}|^2)P_{\geq i-4}u\|_{L_t^{1}L_x^{2}}\lesssim \|\tilde{\epsilon}\|^2_{L_t^{2+}L_x^{\frac{2d}{d-2}-}} \|P_{\geq i-4}u\|_{L_t^{\infty-}L_x^{2+}}
    \lesssim \eta_1^{1-}\|P_{\geq i-4}u\|_{U^2_\Delta}.
\end{align*}
Plugging it into \eqref{equ:x-2epsest}  yields
\begin{align}\notag
\eqref{equ:inhterm1}\lesssim&\eta_1^{1-}\|P_{\geq i-4}u\|_{U^2_\Delta}\\\notag
&+\|(|x|^{-2}*|\tilde{Q}|^2)P_{\geq i-4}u\|^{\delta}_{L_t^{1}L_x^{2}}\notag\\
&\qquad\times\sup_{\|v\|_{U_{\Delta}^{2}}\leq 1}\big\|(|x|^{-2}*|\tilde{Q}|^2)P_{\geq i-4}uP_{\geq i}v\big\|^{1-\delta}_{L_{t,x}^{1}}.
\end{align}

{\bf The estimate of \eqref{equ:inhterm2}:} By the decomposition \eqref{equ:uepsQ}, Strichartz estimate and Lemma \ref{atominterpolation}, we have
\begin{align}\notag
\eqref{equ:inhterm2}\lesssim&
\Big\|\int_{t_0}^te^{i(t-s)\Delta}P_{\geq i}\Big(\big[|x|^{-2}*\text{\O}(\tilde{\epsilon}\overline{P_{\geq i-4}u})\big]P_{\leq i-4}\tilde{\epsilon}\Big)\ ds\Big\|_{U^2_{\Delta}}\\\notag
&+\Big\|\int_{t_0}^te^{i(t-s)\Delta}P_{\geq i}\Big(\big[|x|^{-2}*\text{\O}(\tilde{Q}\overline{P_{\geq i-4}u})\big]P_{\leq i-4}\tilde{\epsilon}\Big)\ ds\Big\|_{U^2_{\Delta}}\\\notag
&+\Big\|\int_{t_0}^te^{i(t-s)\Delta}P_{\geq i}\Big(\big[|x|^{-2}*\text{\O}(\tilde{\epsilon}\overline{P_{\geq i-4}{u}})\big]P_{\leq i-4}\tilde{Q}\Big)\ ds\Big\|_{U^2_{\Delta}}\\\notag
&+\Big\|\int_{t_0}^te^{i(t-s)\Delta}P_{\geq i}\Big(\big[|x|^{-2}*\text{\O}(\tilde{Q}\overline{P_{\geq i-4}u})\big]P_{\leq i-4}\tilde{Q}\Big)\ ds\Big\|_{U^2_{\Delta}}\\\notag
\lesssim&\big\|(|x|^{-2}*|\tilde{\epsilon}P_{\geq i-4}u|)\tilde{\epsilon}\big\|_{L_t^{1}L_x^{2}}\notag\\
&+\big\|(|x|^{-2}*|\tilde{Q}P_{\geq i-4}u|)\tilde{\epsilon}\big\|_{L_t^{1}L_x^{2}}+\|(|x|^{-2}*|\tilde{\epsilon}P_{\geq i-4}u|)\tilde{Q}\|_{L_t^{1}L_x^{2}}\notag\\
&+\big\|(|x|^{-2}*|P_{\geq i-4}u\tilde Q|)\tilde Q\big\|^{\delta}_{L_t^{1}L_x^{2}}\notag\\
&\qquad\times\sup_{\|v\|_{U_{\Delta}^{2}}\leq 1}\big\|\big(|x|^{-2}*\text{\O}(\tilde{Q}\overline{P_{\geq i-4}u})\big)P_{\leq i-4}\tilde{Q}P_{\geq i}v\big\|^{1-\delta}_{L_{t,x}^1}\notag\\\notag
\lesssim&\big\|\tilde{\epsilon}(t)\big\|_{L_t^{2+}L_x^{\frac{2d}{d-2}-}}\|\tilde{\epsilon}(t)\|_{L_t^{\infty-}L_x^{2+}}\|P_{\geq i-4}u\|_{L_t^{2+}L_x^{\frac{2d}{d-2}-}}\\\notag
&+\|\tilde{\epsilon}\|_{L_t^{2}L_x^{\frac{2d}{d-2},2}}\|\tilde{Q}\|_{L_t^{\infty}L_x^{2}}\|P_{\geq i-4}u\|_{L_t^{2}L_x^{\frac{2d}{d-2},2}}\\\notag
&+\|(|x|^{-2}*|P_{\geq i-4}u\tilde{Q}|)\tilde{Q}\|^{\delta}_{L_t^{1}L_x^{2}}\notag\\
&\qquad\times\sup_{\|v\|_{U_{\Delta}^{2}}\leq 1}\big\|\big(|x|^{-2}*\text{\O}(\tilde{Q}\overline{P_{\geq i-4}u})\big)P_{\leq i-4}\tilde{Q}P_{\geq i}v\big\|^{1-\delta}_{L_{t,x}^1}\\\notag
\lesssim&\eta_1^{\frac12-}\|P_{\geq i-4}u\|_{U^2_\Delta}\\\notag
&+\|(|x|^{-2}*|P_{\geq i-4}u\tilde{Q}|)\tilde{Q}\|^{\delta}_{L_t^{1}L_x^{2}}\notag\\
&\qquad\times\sup_{\|v\|_{U_{\Delta}^{2}}\leq 1}\big\|\big(|x|^{-2}*\text{\O}(\tilde{Q}\overline{P_{\geq i-4}u})\big)P_{\leq i-4}\tilde{Q}P_{\geq i}v\big\|^{1-\delta}_{L_{t,x}^1}.
\end{align}

Plugging the above estimates into \eqref{equ:lintermpu}-\eqref{equ:inhterm2}, and
observing that the implicit constant is independent of $t_0\in[a,b]$, we obtain that
\begin{align}
    &\|P_{\geq i}u\|_{U^2_{\Delta}([a,b])}\lesssim \inf_{t\in [a,b]}\|\epsilon(t)\|_{L_x^2}+\frac{1}{T^{20}}+\big[\eta_1^{\frac{1}{2}-}+\eta_1^{1-}\big]\|P_{\geq i-4}u\|_{U^2_{\Delta}}\notag\\
&+\big\|(|x|^{-2}*|\tilde{Q}|^2)P_{\geq i-4}u\big\|^{\delta}_{L_t^{1}L_x^{2}}\notag\\
&\qquad\times\sup_{\|v\|_{U_{\Delta}^{2}}\leq 1}\big\|(|x|^{-2}*|\tilde{Q}|^2)P_{\geq i-4}uP_{\geq i}v\big\|^{1-\delta}_{L_{t,x}^{1}}\notag\\
&+\big\|(|x|^{-2}*|P_{\geq i-4}u\tilde{Q}|)\tilde{Q}\big\|^{\delta}_{L_t^{1}L_x^{2}}\notag\\
&\qquad\times\sup_{\|v\|_{U_{\Delta}^{2}}\leq 1}\big\|\big(|x|^{-2}*\text{\O}(\tilde{Q}\overline{P_{\geq i-4}u})\big)P_{\leq i-4}\tilde{Q}P_{\geq i}v\big\|^{1-\delta}_{L_{t,x}^1}.\label{cbz}
\end{align}
 
By H\"older's inequality and Hardy-Littlewood-Sobolev's inequality and the fact that $\lambda(t)\geq1$,
\begin{align}\label{mmty}
 &\|(|x|^{-2}*|\tilde{Q}|^2)P_{\geq i-4}u\|_{L_t^{1}L_x^{2}}+\|(|x|^{-2}*|P_{\geq i-4}u\tilde{Q}|)\tilde{Q}\|_{L_t^{1}L_x^{2}}\notag\\
\lesssim& \|\tilde{Q}\|^2_{L_{t,x}^{\frac{2(d+2)}{d}}}\|P_{\geq i-4}u\|_{L_t^{\frac{d+2}{2}}L_x^{\frac{2d(d+2)}{(d+4)(d-2)}}}\notag\\
&+\|\tilde{Q}\|_{L_{t,x}^{\frac{2(d+2)}{d}}}\|P_{\geq i-4}u\|_{L_{t,x}^{\frac{2(d+2)}{d}}}\|\tilde{Q}\|_{L_t^{\frac{d+2}{2}}L_x^{\frac{2d(d+2)}{(d+4)(d-2)}}}\notag\\
\lesssim& \Big(\int_{a}^b \frac{1}{\lambda(t)^2}\, dt\Big)^{\frac{d}{(d+2)}}\|P_{\geq i-4}u\|_{U_{\Delta}^2}+\Big(\int_{a}^b \frac{1}{\lambda(t)^2}\, dt\Big)^{\frac{d}{2(d+2)}+\frac{2}{d+2}}\|P_{\geq i-4}u\|_{U_{\Delta}^2}\notag\\
\lesssim& \big[T^{\frac{d}{(d+2)}}+T^{\frac{d+4}{2(d+2)}}\big]\|P_{\geq i-4}u\|_{U_{\Delta}^2}.
\end{align}

Thus, it remains to estimate 
\begin{equation}\label{equ:pi-4up-2v}
    \sup_{\|v\|_{U_{\Delta}^{2}}\leq 1}\big\|(|x|^{-2}*|\tilde{Q}|^2)P_{\geq i-4}uP_{\geq i}v\big\|_{L_{t,x}^{1}},
\end{equation}
 and
\begin{equation}\label{equ:pi-4Qpi-2v}
\sup_{\|v\|_{U_{\Delta}^{2}}\leq 1}\big\|\big(|x|^{-2}*\text{\O}(\tilde{Q}\overline{P_{\geq i-4}u})\big)P_{\leq i-4}\tilde{Q}P_{\geq i}v\big\|_{L_{t,x}^1}.
\end{equation}
For this, we will use \emph{local smoothing estimate}, which, however, relies on the smoothness of $Q$. We observe that by local smoothing estimate (see e.g. \cite{CS88}), for any $R>0$, $i\geq 1$ and $v_0\in L^2$,
\[
\|e^{it\Delta}P_{\geq i}v_0\|_{L_{t,x}^2(\R\times B_{R}(0))}\lesssim R^{1/2}2^{-i/2}\|v_0\|_{L_x^2}.
\]
Therefore, for any atomic decomposition $v=\sum_{\lambda\in\mathbb{Z}}c_\lambda  v^\lambda$, 
\begin{align}
\left\|P_{\geq i} v\right\|_{L_{t,x}^2 \left([a,b]\times B_{R}(0) \right)}
&\lesssim \sup_{\lambda\in\mathbb{Z}} \left\|P_{\geq i}v^{\lambda} \right\|_{L_{t,x}^2(\R\times B_{R}(0))}\cdot \sum_{\lambda\in\mathbb{Z}}|c_{\lambda}|\notag\\
&\lesssim R^{1/2}2^{-i/2}\sum_{\lambda\in\mathbb{Z}}|c_{\lambda}|,
\end{align}
and so
\begin{equation}\label{f5.20}
    \left\|P_{\geq i} v\right\|_{L_{t,x}^2 \left([a,b]\times B_{R}(0) \right)}\lesssim R^{1/2}2^{-i/2}\|P_{\geq i}v\|_{U_{\Delta}^2([a,b])}.
\end{equation}

{\bf The estimate of \eqref{equ:pi-4up-2v}:}
Now we set   $\Omega_j=B_{2^{j}}(0)\setminus B_{2^{j-1}}(0), \ j\in \mathbb{Z}$, then by H\"older's inequality, assumption \eqref{f5.7}, and the local smoothing estimate \eqref{f5.20}, 
we obtain
\begin{align}\label{f5.21}
&\sup_{\left\|v\right\|_{U_{\Delta}^{2}\left([a,b]\right)}\leq 1}
\left\|[(|x|^{-2}*|\tilde{Q}|^2)P_{\geq i-4}u]P_{\geq i}v \right\|_{L_{t,x}^1 \left([a,b]\times \R^d \right)}\notag\\
&\leq\sup_{\left\|v\right\|_{U_{\Delta}^{2}\left([a,b]\right)}\leq 1}
\sum_{j=-\infty}^{\infty}\left\|(|x|^{-2}*|\tilde{Q}|^2)P_{\geq i-4}uP_{\geq i}v \right\|_{L_{t,x}^1 \left([a,b]\times \Omega_j\right)}\notag\\
&\lesssim \sum_{j=-\infty}^\infty2^{j/2}2^{-i/2}\left\|(|x|^{-2}*|\tilde{Q}|^2)P_{\geq i-4}u\right\|_{L_{t,x}^2 \left([a,b]\times \Omega_j\right)}\notag\\
&\lesssim \sum_{j=-\infty}^\infty2^{j/2}2^{-i/2}\|(|x|^{-2}*|\tilde{Q}|^2)\|_{L_{t,x}^{\infty}\left([a,b]\times \Omega_j\right)}\|P_{\geq i-4}u\|_{L_{t,x}^{2}\left([a,b]\times \Omega_j\right)}\notag\\
&\lesssim \sum_{j=-\infty}^\infty2^{j}2^{-i}\|(|x|^{-2}*|\tilde{Q}|^2)\|_{L_{t,x}^{\infty}\left([a,b]\times \Omega_j\right)}\|P_{\geq i-4}u\|_{U_{\Delta}^2\left([a,b]\right)}\notag\\
&\lesssim 2^{-i}\|P_{\geq i-4}u\|_{U_{\Delta}^2\left([a,b]\right)}\sum_{j=-\infty}^\infty2^{j}\|(|x|^{-2}*|\tilde{Q}|^2)\|_{L_{t,x}^{\infty}\left([a,b]\times \Omega_j\right)}.
\end{align}

On the other hand, by  changing of variables, 
\[
|x|^{-2}*|\tilde{Q}|^2=\frac{1}{\lambda(t)^2}\phi_{Q^2}\Big(\frac{x}{\lambda(t)}-\frac{x(t)}{\lambda(t)}\Big).
\]
Let $x_\ast=\sup_{t\in[a,b]}|x(t)|$. By
Proposition~\ref{grondstate1} and $\lambda(t)\geq1$,
\[
\frac{1}{\lambda(t)^2}
\phi_{Q^2}\left(\frac{x-x(t)}{\lambda(t)}\right)
\lesssim\frac{1}{\lambda(t)^2+|x-x(t)|^2}.
\]
Hence
\begin{align}\label{th1}
&\sum_{j\in\mathbb Z}2^j
\big\||x|^{-2}*|\widetilde Q|^2\big\|_{L_{t,x}^\infty([a,b]\times\Omega_j)}
\notag\\
&\qquad\lesssim
\sum_{2^j\leq4(1+x_\ast)}2^j
+\sum_{2^j>4(1+x_\ast)}2^{-j}
\lesssim1+x_\ast\lesssim T^{1/25}.
\end{align}
Inserting the above estimate into \eqref{f5.21}, then
\begin{align} \label{final1}\eqref{equ:pi-4up-2v}
&\lesssim T^{1/25}2^{-i}\|P_{\geq i-4}u\|_{U_{\Delta}^2\left([a,b]\right)}.
\end{align}

{\bf The estimate of \eqref{equ:pi-4Qpi-2v}:}
Set
\[
W(t,x)=\lambda(t)^{-d/2}
\left(1+\left|\frac{x-x(t)}{\lambda(t)}\right|^2\right)^{-50d}.
\]
By the decay of $Q$, the kernel estimate for $P_{\leq i-4}$, and
$2^{i-4}\lambda(t)\geq1$,
\begin{equation}\label{local:projected-Q}
|\widetilde Q(t,x)|+|P_{\leq i-4}\widetilde Q(t,x)|\lesssim W(t,x).
\end{equation}
By H\"older's inequality, the Hardy--Littlewood--Sobolev inequality,
and \eqref{f5.20},
\begin{align}\notag
\eqref{equ:pi-4Qpi-2v}
&\lesssim 2^{-i/2}\sum_{j\in\mathbb Z}2^{j/2}
\|W\|_{L_t^\infty L_x^{\frac{2d}{d-2}}([a,b]\times\Omega_j)}
\big\||x|^{-2}*|\widetilde QP_{\geq i-4}u|\big\|_{L_t^2L_x^d}\notag\\
&\lesssim 2^{-i/2}\sum_{j\in\mathbb Z}2^{j/2}
\|W^{2/d}\|_{L_{t,x}^\infty([a,b]\times\Omega_j)}
\|WP_{\geq i-4}u\|_{L_t^2L_x^{\frac{d}{d-1}}}\notag\\
&\lesssim 2^{-i/2}\sum_{j\in\mathbb Z}2^{j/2}
\|W^{2/d}\|_{L_{t,x}^\infty([a,b]\times\Omega_j)}
\|W^{2/d}P_{\geq i-4}u\|_{L_{t,x}^2}.\label{cccv}
\end{align}
Here $\|W^{(d-2)/d}\|_{L_t^\infty L_x^{2d/(d-2)}}$
is bounded uniformly. Since
\[
W(t,x)^{2/d}\lesssim\frac{1}{\lambda(t)+|x-x(t)|},
\]
we have
\begin{align}\label{zznn1}
&\sum_{j\in\mathbb Z}2^{j/2}
\|W^{2/d}\|_{L_{t,x}^\infty([a,b]\times\Omega_j)}\notag\\
&\qquad\lesssim
\sum_{2^j\leq 4(1+x_\ast)}2^{j/2}
+\sum_{2^j>4(1+x_\ast)}2^{-j/2}
\lesssim(1+x_\ast)^{1/2}\lesssim T^{1/50}.
\end{align}
Applying local smoothing once more, we obtain
\begin{align}\notag
\|W^{2/d}P_{\geq i-4}u\|_{L_{t,x}^2}
&\leq\sum_{j\in\mathbb Z}
\|W^{2/d}\|_{L_{t,x}^\infty([a,b]\times\Omega_j)}
\|P_{\geq i-4}u\|_{L_{t,x}^2([a,b]\times\Omega_j)}\notag\\
&\lesssim 2^{-i/2}T^{1/50}
\|P_{\geq i-4}u\|_{U_\Delta^2([a,b])}.\label{znn2}
\end{align}
Inserting \eqref{zznn1} and \eqref{znn2} into \eqref{cccv} yields
\begin{align}\label{final2}
\eqref{equ:pi-4Qpi-2v}
\lesssim 2^{-i}T^{1/25}
\|P_{\geq i-4}u\|_{U_\Delta^2([a,b])}.
\end{align}
Now we collect estimates \eqref{vvv1}, \eqref{mmty}, \eqref{final1} and \eqref{final2}. In view of \eqref{cbz},  we can finally obtain 
\begin{align*}
&\|P_{\geq i}u\|_{U_{\Delta}^2([a,b])}\\
\lesssim_{\delta}& \inf_{t\in[a,b]}\|\epsilon(t)\|_{L_x^2(\R^d)}+\frac{1}{T^{20}}\\
&+\big({\eta_1^{\frac{1}{2}-}}+\eta_1^{1-}+(T^{\frac{d}{d+2}\delta}+T^{\frac{d+4}{2(d+2)}\delta})(2^{-i}T^{1/25})^{1-\delta}\big)\|P_{\geq i-4}u\|_{U_{\Delta}^2\left([a,b]\right)},
\end{align*}
for sufficiently large $T\gg1$, as desired.
\end{proof}

\section{Estimate of truncated energy \texorpdfstring{$E(P_{\leq k_0+4}u)$}{E(P low u)}}\label{sec-en}
This section is devoted to establishing a universal control of the remainder term $\epsilon$ in the $L^2_x$ norm. The proof mainly relies on examining the upper and lower bounds of the energy of the low-frequency truncated solution, combined with modulation analysis. In particular, we use long-time Strichartz estimates to safely control the error terms arising from the low-frequency truncation. We start with proving the lower bound on the truncated energy:
\begin{theorem}[Lower bound for $E(P_{\leq k_0+4}u)$]\label{t4.1}
Let $u$ be the  solution as in Theorem \ref{t2.3}. Suppose $0< \eta_0,\eta_1\ll1$ are two sufficiently small constants to be determined later,  and $J = [a,b]$ is an interval for which
\begin{equation}\label{assump1}
    \begin{cases}
1\leq \lambda(t) \leq  T^{1/100},
\; \forall \,  t \in J,\;    \int_a^b\frac{1}{\lambda(t)^2}\, dt=2^{3k_0}=T,\\
\sup\limits_{t \in J}|x(t)| \leq  T^{1/25},\;
\sup\limits_{t \in J}\frac{|\xi(t)|}{\lambda(t)}\leq \eta_0\ll\eta_{\ast},\\
\int_{a}^b\frac{1}{\lambda(t)^2}\|\epsilon(t)\|_{L_x^2}^2 \, dt\leq \eta_1.
    \end{cases}
\end{equation}
Then we have
\begin{align}\notag
 E \left(P_{\leq k _0+ 4} u(t) \right)
 \gtrsim&  \Big\|\nabla \Big( e^{-i\gamma(t)}e^{-i\frac{x-x(t)}{\lambda(t)}\cdot\xi(t)}P_{\leq k_0+4}\Big[\frac{1}{\lambda(t)^{d/2}}\epsilon\big(\frac{x-x(t)}{\lambda(t)}\big)\Big]\Big)\Big\|_{L_x^{2}}^{2}\\
 &+\frac{1}{2\lambda(t)^2}\|\epsilon\|^2_{L_x^2}+O(\frac{1}{T^{10}}).
\end{align}
uniformly for all $t\in J$.
\end{theorem}
\begin{proof}
Using the decomposition \eqref{equ:uepsQ}, we easily obtain
\begin{align}\notag
&E(P_{\leq k_0+4}u)\\\notag
=&E(P_{\leq k_0+4}\tilde{Q})\\\label{equ:tranener1}
&+( P_{\leq k_0+4}\tilde{Q},P_{\leq k_0+4}\tilde{\epsilon})_{\dot{H}_x^1}-\big( (|x|^{-2}*(|P_{\leq k_0+4}\tilde{Q}|^2))P_{\leq k_0+4}\tilde{Q},P_{\leq k_0+4}\tilde{\epsilon} \big)_{L_x^2}\\\label{equ:tranener2}
&+\frac12\|\nabla P_{\leq k_0+4}\tilde{\epsilon}\|_{L^2}^2-\frac{1}{2}\int_{\R^d}(|x|^{-2}\ast(|P_{\leq k_0+4}\tilde{Q}|^2))|P_{\leq k_0+4}\tilde{\epsilon}|^2\ dx\\\label{equ:tranener3}
     &-\int_{\R^d}(|x|^{-2}\ast(\Re[P_{\leq k_0+4}\tilde{Q}\overline{P_{\leq k_0+4}\tilde{\epsilon}}]))\Re[P_{\leq k_0+4}\tilde{Q}\overline{P_{\leq k_0+4}\tilde{\epsilon}}]\ dx\\\label{expan1}
     & +R(P_{\leq k_0+4}\tilde{\epsilon},P_{\leq k_0+4}\tilde{Q}),        
 \end{align}
where $R(\tilde{\epsilon},\tilde{Q})$ represents the terms that contain at least three $\tilde{\epsilon}$. Since $E(Q)=0$ and $Q$ has rapid decay, by the assumptions \eqref{assump1}, Hardy's inequality and Bernstein's inequality,
\begin{align}
  E(P_{\leq k_0+4}\tilde{Q})&=\frac{1}{2}\frac{|\xi(t)|^2}{\lambda(t)^2}\|Q\|^2_{L_x^2}+\frac{1}{2}\|P_{\geq k_0+4}\tilde{Q}\|_{\dot{H}_x^1}^2+O\big(\|P_{\geq k_0+4}\tilde{Q}\|_{\dot{H}^1}\|\tilde{Q}\|_{\dot{H}_x^1}\big) \notag\\
  &=\frac{1}{2}\frac{|\xi(t)|^2}{\lambda(t)^2}\|Q\|^2_{L_x^2}+O(2^{-100k_0}).\label{ccc0}
\end{align}

{\bf The estimate of the term \eqref{equ:tranener1}:}
Integrating by parts we obtain
\begin{align}
   &( P_{\leq k_0+4}\tilde{Q},P_{\leq k_0+4}\tilde{\epsilon})_{\dot{H}_x^1}-( (|x|^{-2}*(|P_{\leq k_0+4}\tilde{Q}|^2))P_{\leq k_0+4}\tilde{Q},P_{\leq k_0+4}\tilde{\epsilon} )_{L_x^2}\notag\\
   =&\big(-\Delta P_{\leq k_0+4}\tilde{Q}- (|x|^{-2}*(|P_{\leq k_0+4}\tilde{Q}|^2))P_{\leq k_0+4}\tilde{Q},P_{\leq k_0+4}\tilde{\epsilon} \big)_{L_x^2}.\label{kkj}
\end{align}
Recall that $\phi_{uv}=|x|^{-2}\ast(uv)$. Noticing that
\begin{align*}
&\Delta\tilde{Q}+\phi_{|\tilde Q|^2}\tilde Q\\
=&\lambda(t)^{-2}\tilde{Q}-|\xi(t)|^2\lambda(t)^{-2}\tilde{Q}\\
&-2i\lambda(t)^{-2}\xi(t)\cdot\lambda(t)^{-\frac{d}2}e^{-i\gamma(t)}e^{-i\xi(t)\cdot\frac{x-x(t)}{\lambda(t)}}(\nabla Q)\Big(\frac{x-x(t)}{\lambda(t)}\Big),
\end{align*}
 then by \eqref{assump1}, Proposition \ref{grondstate2} and Bernstein's inequality, 
\begin{equation}\label{large1}
  \|P_{\geq k_0+4}\Delta\tilde{Q}\|_{L_x^2}+\big\|P_{\geq k_0+4}[\phi_{|P_{\leq k_0+4 }\tilde{Q}|^2}P_{\leq k_0+4}\tilde{Q}]\big\|_{L_x^2}\lesssim 2^{-100k_0}.  
\end{equation}
In view of  \eqref{assump1}, \eqref{kkj}, we can  use the identity
\[
\|\tilde\epsilon+\tilde{Q}\|^2_{L^2}=\|\tilde Q\|^2_{L^2}\Rightarrow \|\tilde\epsilon\|_{L^2}^2=-2(\tilde\epsilon,\tilde Q)_{L^2},
\]
to further show that
\begin{align}
   &( P_{\leq k_0+4}\tilde{Q},P_{\leq k_0+4}\tilde{\epsilon})_{\dot{H}_x^1}-\big( (|x|^{-2}*(|P_{\leq k_0+4}\tilde{Q}|^2))P_{\leq k_0+4}\tilde{Q},P_{\leq k_0+4}\tilde{\epsilon} \big)_{L_x^2}\notag\\
   =&\frac{1}{2\lambda(t)^2}\|\tilde{\epsilon}\|^2_{L_x^2}-\frac{|\xi(t)|^2}{2\lambda(t)^2}\|\tilde\epsilon\|^2_{L_x^2}+O(2^{-100k_0}).\label{ccc1}
\end{align}

{\bf The estimate of the terms $\eqref{equ:tranener2}+\eqref{equ:tranener3}$:}
 We further decompose the term $P_{\leq k_0+4}\tilde{\epsilon}$ as
\begin{align}
    P_{\leq k_0+4}\tilde{\epsilon}&=P_{\leq k_0+4}e^{-i\gamma}e^{-i\frac{x-x(t)}{\lambda}\cdot\xi}\frac{1}{\lambda(t)^{d/2}}\epsilon\big(\frac{x-x(t)}{\lambda}\big)\notag\\
    &=e^{-i\gamma}e^{-i\frac{x-x(t)}{\lambda(t)}\cdot\xi}P_{\leq k_0+4}\Big[\frac{1}{\lambda(t)^{d/2}}\epsilon\big(\frac{x-x(t)}{\lambda}\big)\Big]+\mathcal{R}\notag\\
    &:=\tilde{\tilde{\epsilon}}+\mathcal{R}.
\end{align}
As observed in \cite{D2,D1}, the assumption \eqref{assump1} guarantees that
\begin{equation}\label{mmmb}
    \|\mathcal{R}\|_{L_x^2}\lesssim 2^{-k_0}\frac{|\xi(t)|}{\lambda(t)}\|\epsilon\|_{L_x^2}\qtq{and}\|\nabla \mathcal{R}\|_{L_x^2}\lesssim \frac{|\xi(t)|}{\lambda(t)}\|\epsilon\|_{L_x^2}
\end{equation}
Set $f_t(x)=P_{\leq k_0+4}
\big[\epsilon(t,(\cdot-x(t))/\lambda(t))\big](x)$.
By \eqref{mmmb} and \eqref{large1},
\begin{align}\label{midest1}
&\eqref{equ:tranener2}+\eqref{equ:tranener3}\\\notag
=&\frac{1}{2}\|\nabla \tilde{\tilde{\epsilon}}\|^2_{L^2}-\frac{1}{2}\int_{\R^d}(|x|^{-2}\ast(|\tilde{Q}|^2))|P_{\leq k_0+4}\tilde{\epsilon}|^2\ dx\\\notag
&-\int_{\R^d}(|x|^{-2}\ast(\Re[\tilde{Q}\overline{P_{\leq k_0+4}\tilde{\epsilon}}]))\Re[\tilde{Q}\overline{P_{\leq k_0+4}\tilde{\epsilon}}]\ dx\\\notag
&+O\big(\tfrac{|\xi(t)|^2}{\lambda(t)^2}\|\epsilon\|^2_{L^2}\big)+O\big(\tfrac{|\xi(t)|}{\lambda(t)}\|\epsilon\|_{L^2}\|\nabla \tilde{\tilde{\epsilon}}\|_{L^2}\big)+O(2^{-100k_0})\\\notag
=&\frac12\|\nabla\widetilde{\widetilde\epsilon}\|_{L_x^2}^2
-\frac{1}{2\lambda(t)^{2d}}\int\phi_{Q((\cdot-x(t))/\lambda(t))^2}|f_t|^2\,dx\notag\\
&-\frac{1}{\lambda(t)^{2d}}\int\phi_{Q((\cdot-x(t))/\lambda(t))\Re f_t}\,Q\left(\frac{x-x(t)}{\lambda(t)}\right)\Re f_t\,dx\notag\\
&+O\left(\frac{|\xi(t)|^2}{\lambda(t)^2}\|\epsilon\|_{L_x^2}^2\right)
+O\left(\frac{|\xi(t)|}{\lambda(t)}\|\epsilon\|_{L_x^2}
\|\nabla\widetilde{\widetilde\epsilon}\|_{L_x^2}\right)
+O(2^{-100k_0}).
\end{align}
By the product rule,
\begin{align}\label{midest0}
\frac12\|\nabla\widetilde{\widetilde\epsilon}\|_{L_x^2}^2
={}&\frac{|\xi(t)|^2}{2\lambda(t)^{d+2}}\|f_t\|_{L_x^2}^2
+\frac{\xi(t)}{\lambda(t)^{d+1}}\cdot(f_t,i\nabla_xf_t)_{L_x^2}
\notag\\
&+\frac{1}{2\lambda(t)^d}\|f_t\|_{\dot H_x^1}^2.
\end{align}
Let $2^{n(t)}=\lambda(t)$. Changing variables, we obtain
\begin{align}\label{4.27.0}
&\frac{\|f_t\|_{\dot H_x^1}^2}{2\lambda(t)^d}
-\frac{1}{2\lambda(t)^{2d}}\int\phi_{Q((\cdot-x(t))/\lambda(t))^2}|f_t|^2\,dx\notag\\
&-\frac{1}{\lambda(t)^{2d}}\int\phi_{Q((\cdot-x(t))/\lambda(t))\Re f_t}\,Q\left(\frac{x-x(t)}{\lambda(t)}\right)\Re f_t\,dx
\notag\\
=&\frac{\|P_{\leq k_0+4+n(t)}\epsilon\|_{\dot H_x^1}^2}{2\lambda(t)^2}
-\frac{1}{2\lambda(t)^2}\int\phi_{Q^2}|P_{\leq k_0+4+n(t)}\epsilon|^2\,dx\notag\\
&-\frac{1}{\lambda(t)^2}\int\phi_{Q\Re P_{\leq k_0+4+n(t)}\epsilon}\,Q\Re P_{\leq k_0+4+n(t)}\epsilon\,dx.
\end{align}
Since $\epsilon$ satisfies \eqref{orthod}, and $Q$, $\chi$, and
the derivatives $\partial_jQ$ decay rapidly, the coefficients
$p(f):=(P_{\leq k_0+4+n(t)}\epsilon,f)_{L_x^2}$ satisfy
\[
|p(f)|\lesssim2^{-100k_0},\qquad
f\in\{\chi,i\chi,\partial_jQ,i\partial_jQ:1\leq j\leq d\}.
\]
 Let
\begin{align*}
e(t)
= & P_{\leq k_0+4+n(t)} \epsilon-\frac{p \left(\chi \right)}{ \left\|\chi\right\|^2_{L_x^2}}\chi
-\frac{p \left(i\chi \right)}{ \left\|i\chi\right\|^2_{L_x^2}}i\chi\notag\\
&-\sum_{j=1}^d
\frac{p\left(\partial_j Q \right)}{ \left\|\partial_j Q\right\|^2_{L_x^2}}\partial_j Q-\sum_{j=1}^d
\frac{p\left(i\partial_j Q \right)}{ \left\|i\partial_j Q\right\|^2_{L_x^2}}i\partial_j Q
\end{align*}
then $e(t)$ satisfies the orthogonality conditions \eqref{orthod}.
The ground-state equation gives $\ker L_-=\operatorname{span}\{Q\}$.
Since $(Q,\partial_{x_j}Q)_{L_x^2}=0$ and $(L_+Q,Q)_{L_x^2}<0$,
Theorem \ref{posit} implies $(Q,\chi)_{L_x^2}\ne0$.
The decay of $\phi_{Q^2}$ and the usual compactness argument therefore
give coercivity of $L_-$ on the orthogonal complement of $\chi$.
Combining this with the fourth assertion of Theorem \ref{posit},
\eqref{assump1} and Bernstein's inequality, we obtain

\begin{align}
    & \frac{1}{2 \lambda(t)^{2}} \| e(t, x) \|_{\dot{H}^{1}}^{2} -\frac{1}{2\lambda(t)^{2}}\int_{\R^d}(|x|^{-2}\ast(|Q(x)|^2))|e(t,x)|^2\ dx\notag\\
     &-\frac{1}{\lambda(t)^{2}}\int_{\R^d}(|x|^{-2}\ast(\Re[Q(x)e(t,x)]))\Re[Q(x)e(t,x)]\ dx\notag\\
     \geq& \frac{\lambda_1}{\lambda(t)^2}\|e(t,x)\|^2_{H_x^1}-\frac{1}{2\lambda(t)^2}\|e(t,x)\|^2_{L_x^2}\notag\\
     \geq& \frac{\lambda_1}{\lambda(t)^2}\|P_{\leq k_0+4+n(t)}\epsilon(t,x)\|^2_{H_x^1}-\frac{1}{2\lambda(t)^2}\|P_{\leq k_0+4+n(t)}\epsilon(t,x)\|^2_{L_x^2}- O(2^{-50k_0})\label{midest2} 
\end{align}
for some $\lambda_1\ll1$. Inserting \eqref{midest2} into \eqref{4.27.0}, and then inserting \eqref{4.27.0} into \eqref{midest0} and \eqref{midest1} as well as using \eqref{mmmb}, we readily obtain that
\begin{align}\notag
&\eqref{equ:tranener2}+\eqref{equ:tranener3}\\
 \geq&\frac{|\xi(t)|^2}{2\lambda(t)^{d+2}}\|f_t\|_{L_x^2}^2
+\frac{\xi(t)}{\lambda(t)^{d+1}}\cdot(f_t,i\nabla_xf_t)_{L_x^2}\notag\\
&+ \frac{\lambda_1}{2\lambda(t)^2} \|P_{\leq k_0+4+n(t)}\epsilon(t,x)\|_{H_x^1}^{2}-\frac{1}{2\lambda(t)^2}\|P_{\leq k_0+4+n(t)}\epsilon\|^2_{L_x^2}\\\notag
&\notag\\
&+O\big(\frac{|\xi(t)|^2}{\lambda(t)^2}\|\epsilon\|^2_{L_x^2}\big)+O\big(\frac{|\xi(t)|}{\lambda(t)}\|\epsilon\|_{L_x^2}\|\nabla \tilde{\tilde{\epsilon}}\|_{L_x^2}\big)+O(2^{-50k_0})\\\notag
\geq& \frac{\lambda_1}{2\lambda(t)^2} \|P_{\leq k_0+4+n(t)}\epsilon(t,x)\|_{H_x^1}^{2}-\frac{1}{2\lambda(t)^2}\|P_{\leq k_0+4+n(t)}\epsilon\|^2_{L_x^2}\\\notag
&\notag\\
&+O\big(\frac{|\xi(t)|^2}{\lambda(t)^2}\|\epsilon\|^2_{L_x^2}\big)+O\big(\frac{|\xi(t)|}{\lambda(t)}\|\epsilon\|_{L_x^2}\|\nabla \tilde{\tilde{\epsilon}}\|_{L_x^2}\big)+O(2^{-50k_0})\notag\\
\geq&\lambda_2 \|\nabla \tilde{\tilde{\epsilon}}\|_{L_x^{2}}^{2}-\frac{1-\lambda_2}{2\lambda(t)^2}\|P_{\leq k_0+4+n(t)}\epsilon\|^2_{L_x^2}\notag\\
&+O\big(\frac{|\xi(t)|^2}{\lambda(t)^2}\|\epsilon\|^2_{L_x^2}\big)+O\big(\frac{|\xi(t)|}{\lambda(t)}\|\epsilon\|_{L_x^2}\|\nabla \tilde{\tilde{\epsilon}}\|_{L_x^2}\big)+O(2^{-50k_0})\notag\\
\geq &\lambda_3 \|\nabla \tilde{\tilde{\epsilon}}\|_{L_x^{2}}^{2}-\frac{1-\lambda_2}{2\lambda(t)^2}\|P_{\leq k_0+4+n(t)}\epsilon\|^2_{L_x^2}\notag\\
&+O\big(\frac{|\xi(t)|^2}{\lambda(t)^2}\|\epsilon\|^2_{L_x^2}\big)+O(2^{-50k_0})\label{cccc0}
\end{align}
for some constants $\lambda_2, \lambda_3\ll1$.

{\bf The estimate of the term \eqref{expan1}:}
We observe that any term in $R(\tilde{\epsilon},\tilde{Q})$ is of form
\[
c(|x|^{-2}\ast|P_{\leq k_0+4}\tilde{\epsilon}|^2)\overline{P_{\leq k_0+4}\tilde{\epsilon}}P_{\leq k_0+4}\tilde{Q},\quad c(|x|^{-2}\ast|P_{\leq k_0+4}\tilde{\epsilon}|^2)\overline{P_{\leq k_0+4}\tilde{Q}}P_{\leq k_0+4}\tilde{\epsilon}
\]
or
\[
 c(|x|^{-2}\ast|P_{\leq k_0+4}\tilde{\epsilon}|^2)|P_{\leq k_0+4}\tilde{\epsilon}|^2
\]
for some constant $c$. Therefore, by Hardy's inequality and H\"older's inequality,
\begin{equation*}
    |R(\tilde{\epsilon},\tilde{Q})|\lesssim \|P_{\leq k_0+4}\tilde{\epsilon}\|^2_{\dot H_x^1}\|\tilde{\epsilon}\|_{L_x^2}\|\tilde{Q}\|_{L_x^2}+\|P_{\leq k_0+4}\tilde{\epsilon}\|^2_{\dot H_x^1}\|\tilde{\epsilon}\|_{L_x^2}^2
\end{equation*}
which together with \eqref{mmmb} yields
\begin{equation}\label{cccc1}
    |R(\tilde{\epsilon},\tilde{Q})|\lesssim [\eta_{\ast}+\eta_{\ast}^2]\|\tilde{\tilde{\epsilon}}\|^2_{\dot H_x^1}+\eta_{\ast}^3\frac{|\xi(t)|^2}{\lambda(t)^2}
\end{equation}
Inserting estimates \eqref{ccc0}, \eqref{ccc1}, \eqref{cccc0} and \eqref{cccc1} into \eqref{equ:tranener1}-\eqref{expan1}, we obtain the lower bound in Theorem \ref{t4.1}
if $\eta_0\ll1$.
\end{proof}

We next prove the upper bound on the truncated energy:

    \begin{theorem}[Upper bound for $E(P_{\leq k_0+4}u)$]\label{t4.2}
Let $u$ be the  solution as in Theorem \ref{t2.3}. Suppose $0<\eta_0,\eta_1\ll1$ are two  sufficiently small constants to be determined later,  and $J = [a,b]$ is an interval for which the assumptions \eqref{assump1} hold.
Then we have
\begin{equation}\label{f6.1}
\sup_{t \in J} E \left(P_{\leq k _0+ 4} u(t) \right)
 \lesssim \frac{2^{2k_0}}{T} \int_{J}  \left\| \epsilon(t)  \right\|_{L^2_{x}}^{2} \lambda(t)^{-2}  \, {d} t
 + \sup_{t \in J} \frac{|\xi(t)|^{2}}{\lambda(t)^{2}} + 2^{2k_0} T^{-10}.
\end{equation}
\end{theorem}
\begin{proof}
Since $ \| \epsilon(t) \|_{L^2_{x}}$ is continuous as a function of $t$, so
the integral mean value theorem implies that  there exists $t_{0}\in [a, b]=J$ such that
\begin{equation*} 
\left\|{\epsilon}(t_{0}) \right\|_{L^2_{x}}^{2}
= \frac{1}{T} \int_{a}^{b}  \left\| {\epsilon}(t)  \right\|_{L^2_{x} }^{2} \lambda(t)^{-2}  \, {d} t.
\end{equation*}
Using the fact that $Q$ is real-valued, smooth, and has rapidly decaying derivatives and following the same argument as in the proof of Theorem \ref{t4.1}, we have for any $t\in J$
\begin{align}\notag
 E \left(P_{\leq k _0+ 4} u(t) \right)
  =& \frac{1}{2} \frac{|\xi(t)|^{2}}{\lambda(t)^{2}} \| Q \|_{L^2_{x}}^{2} + \frac{1}{2 \lambda(t)^{2}} \left\| {\epsilon}(t) \right\|_{L^2_{x} }^{2}
  - \frac{|\xi(t)|^{2}}{2 \lambda(t)^{2}} \left\| {\epsilon}(t)  \right\|_{L^2_{x} }^{2}
  \\\label{f6.3}
  &+ O \left(2^{2k_0}  \left\| {\epsilon}(t) \right\|_{L^2_{x}}^{2}  \right)
   + O \left(2^{2k_0} T^{-10} \right).
\end{align}
Thus \eqref{f6.1} holds at time $t_{0}$.

Since $E(P_{\leq k_0+4}u)$ is no longer conserved, we need to estimate the increment of this energy. By direct calculation,
\begin{align*} 
&\frac{d}{dt}E(P_{\leq k_0+4}u(t))\\
&\quad=\bigl(i\Delta P_{\leq k_0+4}u(t)+iP_{\leq k_0+4}F(u(t)),
\;P_{\leq k_0+4}F(u(t))-F(P_{\leq k_0+4}u(t))\bigr)_{L_x^2}.
\end{align*}
First, we  estimate
\begin{equation}\label{f6.5}
\int_{t_{0}}^{t'} \big( i \Delta P_{\leq k _0+ 4} u(t),
P_{\leq k_0+4}  {F} \big(u(t) \big)-F \big(P_{\leq k_0+4}u(t)\big) \big)_{L^2_x} \, {d} t.
\end{equation}
Since $P_{\leq k_0+4}F \left(P_{\leq k_0}u \right) - F \left(P_{\leq k_0}u\right)=0$, we have
\begin{align*}
&P_{\leq k_0+4}  {F}(u)-F \big(P_{\leq k_0+4}u\big)\\
=&P_{\leq k_0+4}\big(F(u)-F(P_{\leq k_0}u)\big)
+F\big(P_{\leq k_0}u\big)-F\big(P_{\leq k_0+4}u\big)\\
=&\sum_{(u_1,u_2,u_3)\in{\rm Sym}(A)}P_{\leq k_0+4}\big[(|x|^{-2}\ast(u_1\overline{u_2}))u_3\big]\\
&+\sum_{(u_1,u_2,u_3)\in{\rm Sym}(B)}P_{k_0+1\leq\cdot \leq k_0+4}\big[(|x|^{-2}\ast(u_1\overline{u_2}))u_3\big]\\
&+\sum_{(u_1,u_2,u_3)\in{\rm Sym}(C)}P_{\leq k_0+4}\big[|x|^{-2}\ast(u_1\overline{u_2})u_3\big]
\\
&+\sum_{(u_1,u_2,u_3)\in{\rm Sym}(D)}P_{\geq k_0+1}\big[|x|^{-2}\ast(u_1\overline{u_2})u_3\big]
,
\end{align*}
where ${\rm Sym}(\{a,b,c\})=\big\{(x,y,z):\;\{x,y,z\}=\{a,b,c\}\big\},$ and
\begin{align*}
    A=\{P_{\geq k_0}u,P_{\geq k_0-4}u,u\},\\ B=\{P_{\geq k_0}u,P_{\leq k_0-4}u,P_{\leq k_0-4}u\},\\
    C=\{P_{k_0\leq\cdot\leq k_0+4}u,P_{k_0-4\leq \cdot \leq k_0+4}u,P_{\leq k_0+4}u\},\\ D=\{P_{k_0\leq\cdot\leq k_0+4}u,P_{\leq k_0-4}u,P_{\leq k_0-4}u\}.
\end{align*}
Hence, we get
\begin{align*}
|\eqref{f6.5}|
\lesssim&\sum_{(u_1,u_2,u_3)\in{\rm Sym}(A)}\big\|\Delta P_{\leq k_0+4}u
\big[(|x|^{-2}\ast(u_1\overline{u_2}))u_3\big]\big\|_{L_{t,x}^1}\\
&+\sum_{(u_1,u_2,u_3)\in{\rm Sym}(B)}\big\|\Delta P_{k_0+1\leq\cdot \leq k_0+4}u \big[(|x|^{-2}\ast(u_1\overline{u_2}))u_3\big]\big\|_{L_{t,x}^1}\\
&+\sum_{(u_1,u_2,u_3)\in{\rm Sym}(C)}\big\|\Delta P_{\leq k_0+4}u\big[|x|^{-2}\ast(u_1\overline{u_2})u_3\big]\big\|_{L_{t,x}^1}
\\
&+\sum_{(u_1,u_2,u_3)\in{\rm Sym}(D)}\big\|\Delta P_{k_0+1\leq \cdot\leq k_0+4}u\big[|x|^{-2}\ast(u_1\overline{u_2})u_3\big]\big\|_{L_{t,x}^1}.
\end{align*}
Returning to the pairing in \eqref{f6.5}, expand using
$u=P_{\leq k_0}u+P_{\geq k_0}u$.
For a term with one high-frequency factor, write the commutator as
the outer projection error plus the term in which
$1-P_{\leq k_0+4}$ acts on that factor. Their output frequencies are
at least $16\cdot2^{k_0}$ and $12\cdot2^{k_0}$, respectively.
Both terms are therefore orthogonal to $P_{\leq k_0}u$ and
$F(P_{\leq k_0}u)$.
Here convolution with $|x|^{-2}$ preserves Fourier support.
Thus each nonzero pairing contains two high-frequency factors.
By H\"older's inequality, the Lorentz-space Hardy--Littlewood--Sobolev
inequality, Bernstein's inequality and Remark \ref{lorentzlongtime},
\begin{align}\label{bbv}
 |\eqref{f6.5}|
 &\lesssim 2^{2k_0}\|u\|_{L_t^\infty L_x^2([a,b]\times\R^d)}^2
 \|P_{\geq k_0}u\|_{L_t^2L_x^{\frac{2d}{d-2},2}([a,b]\times\R^d)}^2
 \notag\\
 &\lesssim \frac{2^{2k_0}}{T}\int_a^b
 \lambda(t)^{-2}\|\epsilon(t)\|_{L_x^2}^2\,dt
 +\frac{2^{2k_0}}{T^{10}}.
\end{align}

The term
\begin{equation*}
\int_{t_{0}}^{t'}  \left( i P_{\leq k _0+ 4} F(u(t)),
P_{\leq k_0+4}  {F} \left(u(t) \right)-F \left(P_{\leq k_0+4}u(t)\right) \right)_{L^2_x} \, {d} t
\end{equation*}
is estimated in the same way. The frequency support observation
above again supplies two high-frequency factors. Applying the
Lorentz-space Hardy--Littlewood--Sobolev inequality and Bernstein's
inequality, we obtain
\begin{align}
 &\left|\int_{t_0}^{t'}
 \big(iP_{\leq k_0+4}F(u(t)),
 P_{\leq k_0+4}F(u(t))-F(P_{\leq k_0+4}u(t))\big)_{L_x^2}\,dt\right|
 \notag\\
 &\qquad\lesssim 2^{2k_0}\|u\|_{L_t^\infty L_x^2([a,b]\times\R^d)}^4
 \|P_{\geq k_0}u\|_{U_\Delta^2([a,b])}^2
 \notag\\
 &\qquad\lesssim \frac{2^{2k_0}}{T}\int_a^b
 \lambda(t)^{-2}\|\epsilon(t)\|_{L_x^2}^2\,dt
 +\frac{2^{2k_0}}{T^{10}}.\label{bbv2}
\end{align}
Now \eqref{f6.3}, \eqref{bbv} and \eqref{bbv2} yield \eqref{f6.1} and we can conclude the proof of Theorem \ref{t4.2}.
\end{proof}
Combining Theorem \ref{t4.1} and Theorem \ref{t4.2}, we obtain that
\begin{corollary}\label{c4.3}
    Let $u$ be the  solution in Theorem \ref{t2.3}. Suppose $0<\eta_0,\eta_1\ll1$ are two sufficiently small constants to be determined later,  and $J = [a,b]$ is an interval for which
the assumptions \eqref{assump1} hold.
Then we have
\begin{align}\nonumber
    &\Big\|\nabla \Big( e^{-i\gamma(t)}e^{-i\frac{x-x(t)}{\lambda(t)}\cdot\xi(t)}P_{\leq k_0+4}\Big[\frac{1}{\lambda(t)^{d/2}}\epsilon\big(\frac{x-x(t)}{\lambda(t)}\big)\Big]\Big\|_{L_x^{2}}^{2}\\
 \lesssim& \frac{2^{2k_0}}{T} \int_{J}  \left\| \epsilon(t)  \right\|_{L^2_{x}}^{2} \lambda(t)^{-2}  \, {d} t
 + \sup_{t \in J} \frac{|\xi(t)|^{2}}{\lambda(t)^{2}} + 2^{2k_0} T^{-10}.
\end{align}
and
\begin{align}
    &\|\epsilon(t)\|_{L_x^{2}}^{2}
 \lesssim \frac{2^{2k_0}T^{\frac{1}{50}}}{T} \int_{J}  \left\| \epsilon(t)  \right\|_{L^2_{x}}^{2} \lambda(t)^{-2}  \, {d} t
 + \sup_{t \in J} T^{\frac{1}{50}}\frac{|\xi(t)|^{2}}{\lambda(t)^{2}} + 2^{2k_0}T^{\frac{1}{50}-10}.
\end{align}
\end{corollary}

\section{Modulated virial inequality}\label{Sec:viria}
In this section, we derive a  virial-type inequality, which plays a key role in the proof of our main Theorem \ref{t2.3}. Inspired by \cite{D2,D1}, we will prove this estimate by making use of the frequency-localized Morawetz estimate and the long-time Strichartz estimates 
 derived in the previous section. 
\begin{theorem}\label{t7.13}
Let $u$ be the  solution in Theorem \ref{t2.3}. Suppose $0<\eta_0,\eta_1\ll1$ are two sufficiently small constants to be determined later,  and $J = [a,b]$ is an interval for which
\begin{equation}\label{equ:assumpxixcond}
   \begin{cases}
\xi(b)=x(a)=0,\\
   1 \leq \lambda(t) \leq  T^{1/100},\;  \int_a^b\frac{1}{\lambda(t)^2}\, dt=T=2^{3k_0},\\
   \sup\limits_{t \in J}|x(t)|\leq R^{1/2}=T^{1/25},\;
   \sup\limits_{t \in J}\frac{|\xi(t)|}{\lambda(t)}\leq \eta_0\ll\eta_{\ast},\\
   \int_{a}^b\frac{1}{\lambda(t)^2}\|\epsilon(t)\|_{L_x^2}^2\, dt\leq \eta_1.
   \end{cases} 
\end{equation}

Let $\epsilon =\epsilon_1+ i\epsilon_2$, then for $T$ sufficiently large, there exists a universal constant $C_d$ such that
\begin{align}\label{f7.2}
 \int_{a}^{b} \| \epsilon(t) \|_{L^2_{x}}^{2} \lambda(t)^{-2} \, {d} t
\notag 
 \leq& C_d\big( \epsilon_2(a), x \cdot \nabla Q+\tfrac{d}{2}Q \big)_{L^2_{x}} - C_d\big( \epsilon_2(b), x \cdot \nabla Q+\tfrac{d}{2}Q \big)_{L^2_{x}}\\
&+ C_dT^{1/10}\sup_{t \in J} \frac{|\xi(t)|^{2}}{\lambda(t)^{2}} + O \left(T^{-8} \right).
\end{align}
\end{theorem}

\begin{proof}

We start by recalling the standard Morawetz action. Let $a$ be a spatial function and $u$ satisfy
\begin{equation}\label{equ:usateq}
    i\partial_tu+\Delta u=\mathcal{N}_1,
\end{equation}
on $I\times\R^d.$ For $t\in I$, we define the Morawetz action
$$M_a(t):=\int_{\R^d}\nabla a(x)\cdot \Im \big(\bar{u}\nabla u\big)(t,x)\, {d}x.$$
A simple computation shows
\begin{lemma}[Morawetz identity]\label{lem:moriden}
We have
\begin{align*}
\frac{d}{dt}M_a(t)=\frac12\int_{\R^d}(-\Delta\Delta
a)|u|^2dx+2\int_{\R^d}a_{jk}\Re(\pa_j\bar{u}\pa_ku)dx+\int_{\R^d}a_j(x)\{\mathcal{N}_1,u\}_P^jdx
\end{align*}
where the momentum bracket  $\{f,g\}_P=\Re(f\nabla\bar{g}-g\nabla\bar{f}),$  and repeated indices are implicitly summed.
\end{lemma}

Let $\vartheta(r) \in C^{\infty}([0, \infty))$ be a nonnegative,
nonincreasing cutoff satisfying $\vartheta(r) = 1$ for $0 \leq r \leq 1$
and $\vartheta(r)=0$ for $r\geq2$. Then let
\begin{equation*} 
\phi(r) = \int_{0}^{r} \psi \left(\frac{s}{2R} \right)  \, {d} s, \quad\psi(r):=\vartheta^2(r),
\end{equation*}
and we define the Morawetz action 
\begin{equation}\label{f7.4}
M(t) = \int_{\R^d} \phi(|x|) \frac{x}{|x|} \cdot \Im \left[\overline{P_{\leq k_0+ 4} u} \nabla P_{\leq k_0+4} u \right](t,x) \,   {d} x.
\end{equation}
Letting $v=P_{\leq k_0+4} u$, it solves 
\begin{equation}\label{equ:vsolequ}
i \partial_{t} v + \Delta v +  {F}(v)
 =  F(v) - P_{\leq k_0+4}  {F} (u) := -{\mathcal{N}},
\end{equation}
with $F(v)=(|x|^{-2}\ast|v|^2)v.$

\begin{remark}\label{rem:Moraweig}
 Taking $\nabla a(x)=\phi(|x|) \frac{x}{|x|}$, we obtain
 \begin{align*}
  a_{jk}(x)=&
  \phi'(|x|)\tfrac{x_jx_k}{|x|^2}+\frac{\phi(|x|)}{|x|}\big(\delta_{jk}-\tfrac{x_jx_k}{|x|^2}\big)\\
  =&\vartheta^2\big(\tfrac{|x|}{2R}\big)\delta_{jk}
  +\Big[\frac{\phi(|x|)}{|x|}-\vartheta^2\big(\tfrac{|x|}{2R}\big)\Big]\big(\delta_{jk}-\tfrac{x_jx_k}{|x|^2}\big),\\
 \Delta a=&\phi'(|x|)+\frac{d-1}{|x|}\phi(|x|),\\
 \Delta\Delta a=&\phi^{(3)}+(d-1)\big(\tfrac{\phi''}{|x|}-\tfrac{2\phi'}{|x|^2}+\tfrac{2\phi}{|x|^3}\big)+\tfrac{d-1}{|x|}\Big[\phi''+(d-1)\big(\tfrac{\phi'}{|x|}-\tfrac{\phi}{|x|^2}\big)\Big]\\
 =&\psi''\big(\tfrac{r}{2R}\big)\tfrac{1}{4R^2}+O\big(\tfrac{1}{R^2}\mathbf{1}_{|x|\geq 2R}\big)=O\big(\tfrac{1}{R^2}\mathbf{1}_{|x|\geq 2R}\big),
 \end{align*}
 and due to the symmetry of convolution,
 \begin{align*}
&\int_{\R^d}\nabla a(x)\cdot \{F(v),v\}_P\; dx\\
=& -\frac{1}{2}\iint \big(\nabla a(x)-\nabla a(y)\big)\cdot \nabla_x \big(|x-y|^{-2}\big)|v(x)|^2|v(y)|^2\ dxdy.
 \end{align*}
\end{remark}

By  Lemma \ref{lem:moriden} and Remark \ref{rem:Moraweig}, we get
\begin{align}\label{estq3}
 \frac{d}{dt}M(t)=&O\Big(\tfrac{1}{R^2}\int_{|x|\geq 2R}|v|^2\;dx\Big)\\  \label{estq2}
 &+2\Re \int_{\R^d} \Big[\frac{\phi(|x|)}{|x|}-\vartheta^2\big(\tfrac{|x|}{2R}\big)\Big]\big(\delta_{jk}-\tfrac{x_jx_k}{|x|^2}\big)\big(\pa_j\bar{v}\pa_kv\big)\;dx\\\label{equ:Marterm3}
 &+2\int_{\R^d} \vartheta^2\big(\tfrac{|x|}{2R}\big)|\nabla v|^2\;dx
 \\\label{equ:Marterm4}
 &+\frac{1}{2}\iint \big(\nabla a(x)-\nabla a(y)\big)\cdot \nabla_x \big(|x-y|^{-2}\big)|v(x)|^2|v(y)|^2\ dxdy\\\label{estq4}
 &-\int_{\R^d} \phi(|x|)\frac{x}{|x|}\cdot \big\{\mathcal{N},v\big\}_P\;dx.
\end{align}

{\bf Step 1: Estimation of $M(b) - M(a)$.}
Recalling the decomposition
\begin{align}\notag
    u(t,x)=&e^{-i\gamma(t)}e^{-i\frac{x-x(t)}{\lambda(t)}\cdot\xi(t)}\lambda(t)^{-\frac{d}{2}}
    \Big(\epsilon\big(t,\tfrac{x-x(t)}{\lambda(t)}\big)+Q\big( \tfrac{x-x(t)}{\lambda(t)}\big)\Big)\\\label{equ:uepsQ1}
    \triangleq&\tilde{\epsilon}(t,x)+\tilde{Q}(t,x),
\end{align}
we have
\begin{align*}
 M(t)=&\int_{\R^d}\phi(|x|)\frac{x}{|x|}\cdot \Im\Big[\overline{P_{\leq k_0+4}(\tilde{Q}+\tilde{\epsilon})}\nabla P_{\leq k_0+4} (\tilde{Q}+\tilde{\epsilon})\Big]\;dx\\
 =&\int_{\R^d}\phi(|x|)\frac{x}{|x|}\cdot \Im\Big[\overline{P_{\leq k_0+4}\tilde{Q}}\nabla P_{\leq k_0+4} \tilde{Q}\Big]\;dx\\
 &+\int_{\R^d}\phi(|x|)\frac{x}{|x|}\cdot \Im\Big[\overline{P_{\leq k_0+4}\tilde{\epsilon}}\nabla P_{\leq k_0+4}\tilde{Q}\Big]\;dx\\
 &+\int_{\R^d}\phi(|x|)\frac{x}{|x|}\cdot \Im\Big[\overline{P_{\leq k_0+4}\tilde{Q}}\nabla P_{\leq k_0+4} \tilde{\epsilon}\Big]\;dx\\
 &+\int_{\R^d}\phi(|x|)\frac{x}{|x|}\cdot \Im\Big[\overline{P_{\leq k_0+4}\tilde{\epsilon}}\nabla P_{\leq k_0+4} \tilde{\epsilon}\Big]\;dx\\
 \triangleq&M_1(t)+M_2(t)+M_3(t)+M_4(t).
\end{align*}

{\bf The estimate of the term $M_1(b)-M_1(a)$:}
By \eqref{equ:uepsQ}, $Q$ is real-valued and
\[
\Im\bigl(\overline{\widetilde Q}\nabla\widetilde Q\bigr)
=-\frac{\xi(t)}{\lambda(t)}|\widetilde Q|^2.
\]
Using \eqref{large1}, $x(a)=\xi(b)=0$, and the radiality of $Q$,
we obtain
\begin{align}\label{qqw1}
M_1(b)-M_1(a)
={}&-\left.\frac{1}{\lambda(t)^d}
\int\phi(|x|)\frac{x}{|x|}\cdot\frac{\xi(t)}{\lambda(t)}
Q\left(\frac{x-x(t)}{\lambda(t)}\right)^2\,dx\right|_a^b
\notag\\
&+O(R2^{-100k_0})=O(T^{-10}).
\end{align}

{\bf The estimate of the term $M_2(b)-M_2(a)$:} 
Let $\theta(x)=\phi(|x|)x/|x|$. Since
$\|\partial^\alpha\theta\|_{L^\infty}\lesssim_\alpha R^{1-|\alpha|}$,
Proposition~\ref{grondstate2} and Bernstein's inequality give
\begin{align}\label{subqa}
&\left|\int\overline{P_{\leq k_0+4}\widetilde\epsilon},
\theta\cdot\nabla P_{\leq k_0+4}\widetilde Q\,dx
-\int\overline{\widetilde\epsilon},
\theta\cdot\nabla\widetilde Q\,dx\right|\notag\\
&\quad\leq\|\widetilde\epsilon\|_{L_x^2}
\left(\|P_{\geq k_0+4}(\theta\cdot\nabla P_{\leq k_0+4}\widetilde Q)\|_{L_x^2}
+\|\theta\cdot\nabla P_{\geq k_0+4}\widetilde Q\|_{L_x^2}\right)\notag\\
&\quad\lesssim R2^{-100k_0}\|\epsilon(t)\|_{L_x^2}
\lesssim T^{-10}.
\end{align}
Writing $\epsilon=\epsilon_1+i\epsilon_2$ and using \eqref{equ:uepsQ},
we obtain
\begin{align}\label{f7.15}
M_2(b)-M_2(a)
={}&-\left.\int\theta(x)\cdot\frac{\xi(t)}{\lambda(t)}
\frac{\epsilon_1(t,y)Q(y)}{\lambda(t)^d}\,dx\right|_a^b
\notag\\
&-\left.\frac{1}{\lambda(t)}\int\theta(x)\cdot
\frac{\epsilon_2(t,y)\nabla Q(y)}{\lambda(t)^d}\,dx\right|_a^b
+O(T^{-10}),
\end{align}
where $y=(x-x(t))/\lambda(t)$.
On the other hand, since $\phi(|x|)=|x|$ for $|x|\leq 2R$ and by conditions \eqref{equ:assumpxixcond}, we have
$$ \left\|\frac{1}{\lambda(t)^{d/2}}|x|Q \left(\frac{x-x(t)}{\lambda(t)} \right) \right\|_{L_x^2(|x|\geq 2R)} \lesssim T^{-100},$$ 
then by H\"older's inequality and $x(a)=\xi(b)=0$,
\begin{align}\notag
&\left.\int\theta(x)\cdot\frac{\xi(t)}{\lambda(t)}
\frac{\epsilon_1(t,y)Q(y)}{\lambda(t)^d}\,dx\right|_a^b\\
&\quad=\left.\xi(t)\cdot\int x\epsilon_1(t,x)Q(x)\,dx\right|_a^b\notag\\
&\qquad+\left.\frac{x(t)\cdot\xi(t)}{\lambda(t)}
\int\epsilon_1(t,x)Q(x)\,dx\right|_a^b+O(T^{-10})\notag\\
&\quad=-\xi(a)\cdot\int x\epsilon_1(a,x)Q(x)\,dx
+O(T^{-10}).\label{subqa1}
\end{align}
Recalling that $ \left({\epsilon},i\partial_{x_j}Q\right)_{L^2_x}= 0$, $j=1\cdots,d$,
we also have
\begin{align}\notag
	& \frac{1}{\lambda(t)}\int \phi(|x|) \frac{x}{|x|}
\cdot \left[ \frac{1}{\lambda(t)^{d/2}} \epsilon_2 \left(t,\frac{x - x(t)}{\lambda(t)} \right)  \frac{1}{\lambda(t)^{d/2}} \nabla Q \left(\frac{x - x(t)}{\lambda(t)} \right)\right] \,{d}x \Bigg|_{a}^b \\\notag
	=&\int \epsilon_2(t,x) x\cdot\nabla Q(x) \, {d} x \Big|_a^b+ \frac{1}{\lambda(t)}\int \epsilon_2(t,x)x(t)\cdot \nabla Q(x) \, {d} x \Big|_a^b 
+O \left(T^{-10} \right) \\\label{subqa2}
	=& \int \epsilon_2(t,x) x\cdot\nabla Q(x) \, {d} x \Big|_a^b+O \left(T^{-10} \right).
\end{align}
Substituting \eqref{subqa}, \eqref{subqa1} and \eqref{subqa2} into \eqref{f7.15} and then using H\"older's inequality along with 
 \eqref{equ:assumpxixcond}, we deduce that
\begin{align}\label{f7.19}
	&M_2(b)-M_2(a)\\\notag
	=&\xi(a)\cdot \int \epsilon_1(a,x) xQ(x)\, {d}x- \int \epsilon_2(t,x) x\cdot\nabla Q(x) \, {d} x \Bigg|_a^b+O \left(T^{-10} \right) \\\notag
	=& -\left( \epsilon_2(t,x),x\cdot\nabla Q  \right)_{L_x^2} \Big|_a^b+O \left( \|\epsilon\|^2_{L_x^2}
+|\xi(a)|^2 \right)+O(T^{-10}) \\\notag
	=&-\left( \epsilon_2(t,x),x\cdot\nabla Q(x)  \right)_{L_x^2} \Big|_a^b
\\\notag
& \quad +O \left(\frac{2^{2k_0} T^{1/50}}{ T} \int_{J}  \left\| \epsilon(t)  \right\|_{L^2_{x}}^{2} \lambda(t)^{-2} \, {d} t
 +T^{1/50}\sup_{t \in J} \frac{|\xi(t)|^{2}}{\lambda(t)^{2}} + 2^{2k_0} T^{1/50} T^{-10} \right).
\end{align}

{\bf The estimate of the term $M_3(b)-M_3(a)$:}
Similarly, one can verify that
\begin{align}\label{f7.20}
	&M_3(b)-M_3(a)\\\notag
	=&-\left( \epsilon_2(t,x),x\cdot\nabla Q(x)+dQ  \right)_{L_x^2} \Big|_a^b\\ \notag 
    & \quad 
+O \left(\frac{2^{2k_0} T^{1/50}}{ T} \int_{J}
\|\epsilon(t)\|_{L^2_{x}}^{2} \lambda(t)^{-2} \, {d} t + T^{1/50} \sup_{t \in J} \frac{|\xi(t)|^{2}}{\lambda(t)^{2}} + 2^{2k_0} T^{1/50}T^{-10} \right). 
\end{align}

{\bf The estimate of the term $M_4(b)-M_4(a)$:} by Corollary \ref{c4.3} and \eqref{equ:assumpxixcond}, 
\begin{align}\label{f7.21}
	&|M_4(b)-M_4(a)|\\ \notag 
	\lesssim& \|\phi\|_{L^{\infty}}\sup_{t\in J}
\|\epsilon(t)\|_{L^2_x} \cdot\sup_{t\in J}
\left\|\nabla P_{\leq k_0+4}
\left( e^{-i(x-x(t)) \cdot \frac{\xi(t)}{\lambda(t)}} \frac{1}{\lambda(t)^{d/2}} \epsilon \left(t,\frac{x - x(t)}{\lambda(t)} \right) \right) \right\|_{L_x^2}\\\notag
	\lesssim& R \left(\frac{2^{2k_0} T^{1/50}}{ T} \int_{J}
\|\epsilon(t)\|_{L^2_{x}}^{2} \lambda(t)^{-2}  \, {d} t + T^{1/50} \sup_{t \in J} \frac{|\xi(t)|^{2}}{\lambda(t)^{2}} + 2^{2k_0}T^{1/50} T^{-10} \right)\\\notag
	\lesssim& \frac{2^{2k_0} T^{1/10}}{T} \int_{J} \| \epsilon(t)  \|_{L^2_{x}}^{2} \lambda(t)^{-2} \, {d} t
+ T^{1/10} \sup_{t \in J} \frac{|\xi(t)|^{2}}{\lambda(t)^{2}} + 2^{2k_0} T^{1/10}T^{-10}. 
\end{align}

Thus, by \eqref{qqw1}, \eqref{f7.19}, \eqref{f7.20} and \eqref{f7.21}, we obtain
\begin{align}\label{f7.22}
&	M(b) - M(a) \\\nonumber
= &  2 \big( \epsilon_2(a), x \cdot \nabla Q+\tfrac{d}{2}Q\big)_{L^2_{x}}
- 2 \big( \epsilon_2(b),  x \cdot \nabla Q+\tfrac{d}{2}Q \big)_{L^2_{x}} \\\nonumber
& + O \left(\frac{2^{2k_0} T^{1/10}}{T} \int_{J} \|\epsilon(t) \|_{L^2_{x} }^{2} \lambda(t)^{-2}  \, {d} t
+ T^{1/10} \sup_{t \in J} \frac{|\xi(t)|^{2}}{\lambda(t)^{2}} + 2^{2k_0} T^{1/10} T^{-10} \right).
\end{align}

{\bf Step 2: Estimation of  $\frac{ {d} }{ {d} t}M(t)$ and conclusion of the proof.}
 We will estimate the terms \eqref{estq3}-\eqref{estq4} separately. 
 
{\bf Estimate of  \eqref{estq4}.}
Noticing that, we can always further decompose  nonlinear term $\mathcal{N}$ as $\mathcal{N}_1+\mathcal{N}_2$,  where $\mathcal{N}_1$ contains at least two high-frequency factors (i.e. Fourier supported on $|\xi|\gtrsim 2^{k_0}$) while $\mathcal{N}_1$ itself has Fourier support on $|\xi|\lesssim 2^{k_0}$, and $\mathcal{N}_2$ contains one high-frequency factor  while $\mathcal{N}_2$ itself has Fourier support on $|\xi|\sim 2^{k_0}$.

For the first case, by applying the Hölder inequality and the Hardy-Littlewood-Sobolev inequality, and utilizing the long-time Strichartz estimates for the two high-frequency factors (Theorem \ref{longtime}), then similar to \eqref{bbv}, one can easily obtain that
\begin{align}\nonumber
\int_a^b\Big| \int_{\R^d} \phi(|x|)\frac{x}{|x|}\cdot \big\{\mathcal{N}_1,v\big\}_P\;dx\Big|\,dt
 \lesssim&  2^{k_0} R  \|P_{\gtrsim k_0}u\|^2_{U_{\Delta}^{2}(J)} \\ \label{f7.24}
 \lesssim& 2^{k_0} R  \left(\frac{1}{T} \int_{a}^{b}  \| \epsilon(t) \|_{L^2_{x}}^{2} \lambda(t)^{-2}  \, {d} t +  T^{-10} \right).
\end{align}

For the second case, set $\theta(x)=\phi(|x|)x/|x|$, with
$\theta(0)=0$. This vector field is smooth and satisfies
$\|\partial^\alpha\theta\|_{L^\infty}\lesssim_\alpha
R^{1-|\alpha|}$ for every multi-index $\alpha$.
Applying Bernstein's inequality to $\theta$ and using
$\|u\|^q_{L_t^qL_x^r([a,b]\times\mathbb R^d)}\lesssim T$
for admissible $(q,r)$ with $q<\infty$, we obtain

\begin{align}\label{10.10}
 & \int_a^b\Big| \int_{\R^d} \phi(|x|)\frac{x}{|x|}\cdot \big\{\mathcal{N}_2,v\big\}_P\;dx\Big|\,dt
  \lesssim  2^{-99k_0} \left\| P_{\geq k_0 + 3} |\nabla|^{100}\theta(x)  \right\|_{L^{\infty}_x }T
  \lesssim  {T^{-20}}.
\end{align}
Combining \eqref{f7.24} and \eqref{10.10}, we obtain 
\begin{align}\label{estq41}
	\int_a^b|\eqref{estq4}| \, {d} t
\lesssim 2^{k_0} R \left(\frac{1}{T} \int_{J}  \| \epsilon(t)  \|_{L^2_{x} }^{2} \lambda(t)^{-2}  \, {d} t+T^{-10} \right).
\end{align}

{\bf Estimate of \eqref{estq2}}:
Since $\vartheta$ is nonincreasing, the matrix
\[
\left(\frac{\phi(r)}{r}-\vartheta^2\left(\frac{r}{2R}\right)\right)
\left(\delta_{jk}-\frac{x_jx_k}{|x|^2}\right),\qquad r=|x|,
\]
is positive semidefinite. Therefore,
\begin{align}\label{estq21}
	\eqref{estq2}\geq 0.
\end{align}

{\bf Estimate of \eqref{estq3}}:
Noticing that 
$ \left\|\chi_{|x|\geq 2R}\frac{1}{\lambda(t)^{d/2}}Q \left(\frac{x-x(t)}{\lambda(t)} \right) \right\|_{L_x^2}\lesssim T^{-100}$, by \eqref{equ:assumpxixcond}, we have 
\begin{align}\label{f7.32}
	& \frac{1}{R^{2}} \int_{|x|\geq 2R} 
\left|P_{\leq k_0+4}  \left(e^{-i \gamma(t)} e^{-i(x-x(t)) \cdot \frac{\xi(t)}{\lambda(t)}} \frac{1}{\lambda(t)^{d/2}}Q \left(\frac{x - x(t)}{\lambda(t)} \right) \right) \right|^{2} \, {d} x \notag \\
  \lesssim&	\frac{1}{R^2T^{100}}\lesssim\frac{1}{\lambda(t)^2T^{10}}.
\end{align}
On the other hand, by \eqref{equ:assumpxixcond}, for  sufficiently large $T$, we have
\begin{align}\label{f7.33}
	& \frac{1}{R^{2}}\int  \left|P_{\leq k_0+4} \left(e^{-i \gamma(t)} e^{-i(x-x(t)) \cdot \frac{\xi(t)}{\lambda(t)}}\frac{1}{\lambda(t)^{d/2}} \epsilon \left(t, \frac{x - x(t)}{\lambda(t)} \right) \right)  \right|^{2} \, {d} x \notag \\
 \lesssim&  \frac{1}{R^2} \| \epsilon(t) \|_{L^2_{x} }^{2}
\lesssim \frac{1}{T^{1/25}\lambda(t)^2}  \|\epsilon(t)\|_{L^2_{x}}^{2}.
\end{align}
Therefore,  \eqref{f7.32} and \eqref{f7.33} imply that
\begin{align}\label{estq31}
	|\eqref{estq3}|\lesssim \frac{1}{T^{1/25}\lambda(t)^2} \|\epsilon(t) \|^2_{L^2_x}+O(\frac{1}{\lambda(t)^2T^{10}}).
\end{align}

{\bf Estimate of $\eqref{equ:Marterm3}+\eqref{equ:Marterm4}$:}
We claim that
\begin{align}\label{finalz}
&\int_a^b\eqref{equ:Marterm3}+\eqref{equ:Marterm4}\ dt
 \gtrsim \int_a^b\frac{\|\epsilon(t)\|^2_{L_x^2}}{\lambda(t)^2}\,dt
 -O\Big(\int_a^b\Big|\frac{\xi(t)}{\lambda(t)}\Big|^2
       \|\epsilon(t)\|^2_{L_x^2}\,dt\Big)-O(T^{-8}).
\end{align}
Next, we prove Theorem \ref{t7.13} by assuming that the above claim holds.
Indeed, if \eqref{finalz} holds, then by estimates \eqref{estq41}, \eqref{estq21} and \eqref{estq31} and the fundamental theorem of calculus, we see that for all large $T$, then
\begin{align} \int_{a}^b\frac{\|\epsilon(t)\|^2_{L_x^2}}{\lambda(t)^{2}}\, {d}t  \leq&   2 \big( \epsilon_2(a), x \cdot \nabla Q+\tfrac{d}{2}Q\big)_{L^2_{x}}
- 2 \big( \epsilon_2(b),  x \cdot \nabla Q+\tfrac{d}{2}Q \big)_{L^2_{x}}\notag \\\label{equ:2k0RT}
&  +O\Big(\frac{2^{k_0} R}{T} \int_a^b  \frac{\| \epsilon(t)  \|_{L^2_{x} }^{2}}{\lambda(t)^2}  \, {d} t+\frac{2^{2k_0} T^{1/10}}{T} \int_a^b \frac{\|\epsilon(t) \|_{L^2_{x} }^{2}}{ \lambda(t)^2}  \, {d} t\Big)\\\notag
&+O\Big(\frac{2^{k_0} R}{T^{10}}+ T^{1/10} \sup_{t \in J} \frac{|\xi(t)|^{2}}{\lambda(t)^{2}} + 2^{2k_0} T^{1/10} T^{-10}+T^{-8} \Big).
\end{align}
Since $R=T^{2/25}=2^{6k_0/25}$, we have
\[
\frac{2^{k_0}R}{T}=T^{-44/75},\qquad
\frac{2^{2k_0}T^{1/10}}{T}=T^{-7/30}.
\]
Thus the term \eqref{equ:2k0RT} can be absorbed in the left-hand side
for $T$ sufficiently large, which implies \eqref{f7.2}. Therefore, we complete the proof of Theorem \ref{t7.13} provided that the claim \eqref{finalz} holds.

{\bf Step 3: proof of the claim \eqref{finalz}.}     
Set $\vartheta_R(x)=\vartheta(|x|/R)$ and define the symmetric remainder kernel
\[
 B_R(x,y)=-\frac12\bigl(\nabla a(x)-\nabla a(y)\bigr)
             \cdot\nabla_x|x-y|^{-2}
          -\vartheta_R(x)^2\vartheta_R(y)^2|x-y|^{-2}.
\]
Since $\nabla a(x)=x$ for $|x|\leq2R$ and $\nabla a$ is
uniformly Lipschitz, we have
\[
 |B_R(x,y)|\lesssim |x-y|^{-2}
      \mathbf{1}_{\{\max(|x|,|y|)\geq R\}}.
\]
The nonlinear virial contribution therefore admits the exact decomposition
\begin{align}
&\int_{\R^d}\nabla a(x)\cdot\{F(v),v\}_P\,dx\notag\\
=&-\frac12\iint_{\R^{2d}}
 \bigl(\nabla a(x)-\nabla a(y)\bigr)\cdot\nabla_x|x-y|^{-2}
 |v(x)|^2|v(y)|^2\,dxdy\notag\\
=&\iint_{\R^{2d}}|x-y|^{-2}
 |\vartheta_R(x)v(x)|^2|\vartheta_R(y)v(y)|^2\,dxdy\label{vvc1}\\
&+\iint_{\R^{2d}}B_R(x,y)|v(x)|^2|v(y)|^2\,dxdy.
 \label{vvc2}
\end{align}
We can further decompose \eqref{vvc2} as
\[
\eqref{vvc2}=\mathcal{A}_{0,1}+\mathcal{A}_{0,2}+\mathcal{A}_{0,3},
\]
where the term $\mathcal{A}_{0,1}$ contains the terms in which at least one ground-state factor is evaluated in $\{|x|\geq R\}$, $\mathcal{A}_{0,3}$ contains the terms that have at least three $\epsilon$ terms, and $\mathcal{A}_{0,2}$ contains the other terms in this decomposition. Noticing that by the Pigeonhole Principle, the term $\mathcal{A}_{0,2}$ must have the following form 
\[
c\iint_{|x|\leq R, |y|\geq R} B_R(x,y)|P_{\leq k_0+4}\tilde{Q}(x)|^2|P_{\leq k_0+4}\tilde{\epsilon}(y)|^2\,dxdy
\]
for some constant $c$. Using the bound on $B_R$ and Hardy's inequality, we obtain
\begin{align}
    |\mathcal{A}_{0,2}|\lesssim& \iint_{\frac{R}{4}\leq|x|\leq R, \ |y|\geq R} |x-y|^{-2}|P_{\leq k_0+4}\tilde{Q}(x)|^2|P_{\leq k_0+4}\tilde{\epsilon}(y)|^2\ dxdy\notag\\
    &+\iint_{|x|\leq \frac{R}{4}, \ |y|\geq R} |x-y|^{-2}|P_{\leq k_0+4}\tilde{Q}(x)|^2|P_{\leq k_0+4}\tilde{\epsilon}(y)|^2\ dxdy\notag\\
    \lesssim& \left(\frac{1}{\lambda(t)^2T^{10}}+\frac{1}{R^2}\right)
\|\epsilon(t)\|_{L_x^2}^2\notag\\
    \lesssim& \frac{\|\epsilon(t)\|^2_{L_x^2}}{\lambda(t)^2T^{10}}
+\frac{1}{\lambda(t)^2R}\|\epsilon(t)\|^2_{L_x^2}.
\end{align}
By \eqref{equ:assumpxixcond}, $R/\lambda(t)\geq T^{7/100}$
and $|x(t)|\leq R^{1/2}$. Proposition~\ref{grondstate2} and
Bernstein's inequality give
\begin{equation}\label{virial:Q-tail}
\begin{aligned}
&\|\nabla P_{\leq k_0+4}\widetilde Q\|_{L_x^2(|x|\geq R/8)}
+R^{-1}\|P_{\leq k_0+4}\widetilde Q\|_{L_x^2(|x|\geq R/8)}\\
&\qquad\lesssim\frac{1}{\lambda(t)T^{100}}.
\end{aligned}
\end{equation}
The preceding estimate follows by applying Hardy's inequality to
a smooth spatial cutoff and using \eqref{virial:Q-tail}.
Similarly, \eqref{virial:Q-tail} yields
\begin{align}
    |\mathcal{A}_{0,1}|
&\lesssim \Bigl(\|\nabla P_{\leq k_0+4}\widetilde Q\|_{L_x^2(|x|\geq R/2)}
+R^{-1}\|P_{\leq k_0+4}\widetilde Q\|_{L_x^2(|x|\geq R/2)}\Bigr)\notag\\
&\qquad\times\Bigl(\|\nabla P_{\leq k_0+4}\widetilde Q\|_{L_x^2}
+\|\nabla P_{\leq k_0+4}\widetilde\epsilon\|_{L_x^2}\Bigr)\notag\\
&\lesssim \frac{2^{k_0}}{\lambda(t)T^{100}}
\lesssim\frac{1}{\lambda(t)^2T^{20}}.
\end{align}
Finally, for the term $\mathcal{A}_{0,3}$, we just use Hardy-Littlewood-Sobolev's inequality to deduce that
\begin{align}\label{qrr}
   \int_a^b |\mathcal{A}_{0,3}|\, dt&\lesssim \int_a^b\Big(\Big\|\frac{1}{\lambda(t)^{d/2}}\epsilon\big(\frac{x-x(t)}{\lambda(t)}\big)\Big\|^2_{L_x^{\frac{2d}{d-2}}}\|\epsilon\|_{L_x^2}\big(\|\epsilon\|_{L_x^2}+\|Q\|_{L^2}\big)\Big)\, dt\notag\\
    &\lesssim \eta_{\ast}\int_a^b\frac{1}{\lambda(t)^2}\|\epsilon(t)\|^2_{L_x^2}\, dt,
\end{align}
where the last inequality comes from \eqref{estepsilonnorm}. In view of these estimates, we can use the assumptions \eqref{equ:assumpxixcond}(which means that $\lambda(t)^2\ll R$) to deduce that
\begin{align}\notag
\int_a^b\eqref{equ:Marterm3}+\eqref{equ:Marterm4}\, dt&=2\int_a^b\int_{\R^d} \vartheta^{2} \left(\frac{|x|}{2R} \right)
\left|\nabla v\right|^{2} \, dxdt\\\notag
&-\int_a^b\int_{\R^{2d}} |x-y|^{-2}\big|\vartheta(\frac{|x|}{R})v(x)\big|^2\big|\vartheta(\frac{|y|}{R})v(y)\big|^2 \ dxdydt\\
&+O\Big((\eta_{\ast}+\frac{1}{R})\int_a^b\frac{1}{\lambda(t)^2}\|\epsilon(t)\|^2_{L_x^2}\, dt\Big)+O(\frac{1}{T^{8}}).
\end{align}
Therefore, it only remains to estimate the localized energy part
\begin{align}\label{equ:locenerg}
    &2\int_a^b\int_{\R^d} \vartheta^{2} \big(\frac{|x|}{2R} \big)
|\nabla v|^{2} \, dxdt
-\int_a^b\int_{\R^{2d}} |x-y|^{-2}\big|\vartheta(\frac{|x|}{R})v(x)\big|^2\big|\vartheta(\frac{|y|}{R})v(y)\big|^2 \, dxdydt.
\end{align}
Obviously,
\begin{align}\notag
\eqref{equ:locenerg}
    \geq&2\int_a^b\int_{\R^d} \vartheta^{2} \left(\frac{|x|}{R} \right)
\left|\nabla v\right|^{2} \, dxdt\notag\\
&-\int_a^b\int_{\R^{2d}} |x-y|^{-2}\big|\vartheta(\frac{|x|}{R})v(x)\big|^2\big|\vartheta(\frac{|y|}{R})v(y)\big|^2 \ dxdydt\label{pppo}
\end{align}
As before, we can rewrite  \eqref{pppo} as a sum of terms involving $Q$ and $\epsilon$, and then estimate them separately.
For the terms that contain at least three $Q$ in the potential energy or at least one $Q$ in the kinetic energy, we denote them by $\mathcal{A}_1$. Then there must be one $Q$ outside the convolution structure. Hence, by using the smoothness of $Q$, one can easily check that
\begin{align}\label{zzm0}
    \mathcal{A}_1=  \int_a^b\Big(\frac{2|\xi(t)|^{2}}{\lambda(t)^{2}} \| Q \|_{L^2_{x}}^{2} + \frac{2}{ \lambda(t)^{2}}
\|\epsilon(t)\|_{L^2_{x}}^{2} - \frac{2|\xi(t)|^{2}}{\lambda(t)^{2}}  \| \epsilon(t)  \|_{L^2_{x}}^{2} + O \left(2^{2k_0} T^{-10} \right)\Big)\, dt. 
\end{align}
Next, we consider the terms involving at least three $\epsilon$'s, which are denoted as $\mathcal{A}_3$. In this case, similar to the proof of \eqref{qrr}, we can only use the Hardy-Littlewood-Sobolev inequality and H\"older's inequality to obtain 
\begin{align}
    |\mathcal{A}_3|&\lesssim \int_a^b\Big(\Big\|\frac{1}{\lambda(t)^{d/2}}\epsilon\big(\frac{x-x(t)}{\lambda(t)}\big)\Big\|^2_{L_x^{\frac{2d}{d-2}}}\|\epsilon\|_{L_x^2}\big(\|\epsilon\|_{L_x^2}+\|Q\|_{L^2}\big)\Big)\, dt\notag\\
    &\lesssim \eta_{\ast}\int_a^b\frac{1}{\lambda(t)^2}\|\epsilon(t)\|^2_{L_x^2}\, dt.
\end{align}

Finally, we consider the terms involving two $\epsilon$'s, which are denoted as $\mathcal{A}_2$. On one side, by the product rule and Young's inequality, 
we have
\begin{align*}
&\int\vartheta\left(\frac{|x|}{R}\right)^2
|\nabla P_{\leq k_0+4}\widetilde\epsilon|^2\,dx\\
&\quad\geq(1-\delta_0)
\left\|\vartheta\left(\frac{|x|}{R}\right)
P_{\leq k_0+4}\widetilde\epsilon\right\|_{\dot H_x^1}^2
-C\delta_0^{-1}R^{-2}\|\epsilon(t)\|_{L_x^2}^2.
\end{align*}
On the other hand, \eqref{equ:assumpxixcond} and the decay of the
orthogonality functions give
\[
\left|\left(\vartheta\left(\frac{|\lambda(t)x+x(t)|}{R}\right)
\epsilon(t),f\right)_{L_x^2}\right|\lesssim T^{-20},
\]
for every $f\in\{\chi,i\chi,\partial_jQ,i\partial_jQ:1\leq j\leq d\}$.
Applying the coercivity estimate \eqref{midest2} and the estimates leading to \eqref{cccc0} to the localized remainder, we obtain 
\begin{align}\notag
\mathcal{A}_2\geq&
\int_a^b\Big[\left(\frac{2\lambda_{1}}{\lambda(t)^{2}} -\frac{2}{\lambda(t)^{2}} \right)
\| \epsilon(t)\|_{L^2_{x}}^{2}\\
&+ \lambda_{1} \Big\|\vartheta(\frac{|x|}{R})\nabla P_{\leq k_0+4} \left( e^{-i \gamma(t)} e^{-i(x-x(t))\cdot \frac{\xi(t)}{\lambda(t)}} \frac{1}{\lambda(t)^{d/2}} \epsilon\left(t, \frac{x - x(t)}{\lambda(t)} \right) \right)\Big\|^2_{L_x^2}\notag \\\notag
&-\delta_0\left\| \vartheta \left(\frac{|x|}{R} \right) P_{\leq k_0+4} \left(e^{-i \gamma(t)} e^{-i(x-x(t))\cdot \frac{\xi(t)}{\lambda(t)}} \frac{1}{\lambda(t)^{d/2}} \epsilon
\left(t, \frac{x - x(t)}{\lambda(t)} \right) \right)  \right\|_{\dot{H}_x^{1}}^2\\\label{zzm2}
&-\delta_0^{-1}O \left(\frac{1}{R ^2}  \|\epsilon(t)\|_{L^2_{x}}^{2} \right) -O(|\frac{\xi(t)}{\lambda(t)}|^2\|\epsilon(t)\|^2_{L_x^2})\Big]\, dt- O \left(2^{2k_0} T^{-10} \right)
\end{align}
with a constant $0<\lambda_1\ll \frac{1}{2}$. 

Collecting estimates \eqref{zzm0}-\eqref{zzm2}, if we take $\delta_0$ and $\eta_{\ast}$ sufficiently small, then the claim \eqref{finalz} holds and we can conclude the proof of this theorem.
\end{proof}
Using H\"older's inequality and \eqref{f7.2}, we get
\begin{equation*}
\int_{a}^{b} \frac{\| \epsilon(t)\|_{L^2_{x}}^2}{ \lambda(t)^2} \, {d} t
		\lesssim \|\epsilon(a)\|_{L^2_x}+\|\epsilon(b)\|_{L^2_x}
		+ T^{1/10} \sup_{t \in J} \frac{|\xi(t)|^{2}}{\lambda(t)^{2}} + O \left(T^{-8} \right).
\end{equation*}
Notice that this inequality is scaling-invariant, using the change of variables \eqref{f4.23}  and then using the modulation theory, we have the following theorem:
\begin{theorem}\label{t7.14}
	Let $u$ be the solution as in Theorem \ref{t2.3}. Suppose $J = [a, b]$ is an interval on which
	\begin{equation}\label{equ:assumlamda}
		\frac{\sup_{t\in J}\lambda(t)}{\inf_{t\in J}\lambda(t)}\leq T^{1/100} \text{ and }  \int_{J} \lambda(t)^{-2}  \, {d} t = T.
	\end{equation}
	Also, suppose $\epsilon = \epsilon_1+ i\epsilon_2$, and let $[s_1,s_2]=s([a,b])$, with $s(t)$  defined in \eqref{f4.23}. Then there exist sufficiently small $\tilde{\eta}$ and large $\tilde{T}$, 
so that for any $\eta<\tilde{\eta}$ and $T\geq\tilde{T}$ satisfying 
\begin{equation}\label{equ:assepsl2nor}
    \sup_{t\in[a,b]} \|\epsilon(t) \|_{L^2_x}:=\eta\leq \tilde{\eta},\quad \eta T\ll T^{1/100}, \quad \eta^2 T\ll 1,
\end{equation} 
we have
\begin{equation}\label{equ:intepsilonest}
		\int_{s_1}^{s_2}  \| \epsilon(s)\|_{L^2_{x}}^{2}  \, {d} s
 \lesssim  \| \epsilon(s_1)  \|_{L^2_{x}} +  \|\epsilon(s_2) \|_{L^2_{x} }
 + \eta^{4} T^{2+1/10}  + O \left(
 {T^{-8}} \right).
	\end{equation}
\end{theorem}
\begin{proof}
Let $\eta_0,\eta_1$ be the constants in Theorem~\ref{t7.13}, and set
\[
\widetilde\lambda=\min_{t\in J}\lambda(t),\qquad
\alpha=\widetilde\lambda^{-2}a,\qquad
\beta=\widetilde\lambda^{-2}b,\qquad
\lambda_1(t)=\widetilde\lambda^{-1}
                 \lambda(\widetilde\lambda^2t).
\]
The rescaled function
\[
v_{\widetilde\lambda}(t,x)
  =\widetilde\lambda^{d/2}u(\widetilde\lambda^2t,
                            \widetilde\lambda x)
\]
is a solution of \eqref{1.1}. Apply a Galilean transformation and a
spatial translation to this solution, and continue to denote the
result by $v_{\widetilde\lambda}$. We choose these transformations so
that its modulation parameters satisfy
$\xi_1(\beta)=0$ and $x_1(\alpha)=0$. The modulation decomposition is
\[
e^{i\gamma_1(t)}e^{ix\cdot\xi_1(t)}\lambda_1(t)^{d/2}
v_{\widetilde\lambda}\bigl(t,\lambda_1(t)x+x_1(t)\bigr)
 =Q(x)+\widetilde\epsilon(t,x),\qquad t\in[\alpha,\beta],
\]
where $\widetilde\epsilon(t,x)=\epsilon(\widetilde\lambda^2t,x)$.
In particular,
\[
\|\widetilde\epsilon(t)\|_{L_x^2}
 =\|\epsilon(\widetilde\lambda^2t)\|_{L_x^2},
\]
and $\widetilde\epsilon$ satisfies \eqref{orthod}. By
\eqref{equ:assepsl2nor}, we have
	\begin{align*} 
	\int_{\tilde{\lambda}^{-2}a}^{\tilde{\lambda}^{-2}b}  \left\| \tilde{\epsilon}(t)  \right\|_{L^2_{x}}^{2}\lambda_1(t)^{-2}\, {d} t\leq\eta^2T\ll1,\\
1\leq\lambda_1(t)\leq T^{1/100},\quad \forall  \, t\in \left[\tilde{\lambda}^{-2}a,\tilde{\lambda}^{-2}b \right].
	\end{align*}
 Let $\tilde{s}(t):=\int_0^t \frac{\tilde{\lambda}^2}{\lambda \left(\tilde{\lambda}^2\tau \right)^2} \, {d} \tau, \
 \left[s_1',s_2' \right]:=\tilde{s} \left( \left[\tilde{\lambda}^{-2}a,\tilde{\lambda}^{-2}b \right] \right)=s([a,b])$. 
 By \eqref{f4.54} and \eqref{f4.55}, we have 
	\begin{equation*}
		\sup_{t \in  \left[\tilde{\lambda}^{-2}a,\tilde{\lambda}^{-2}b \right]}  \left|\frac{\xi_1(t)}{\lambda_1(t)} \right|
=\sup_{\tilde{s} \in \left[s_1', s_2' \right]}
\left|\frac{\xi_1 \left(\tilde{s} \right)}{\lambda_1 \left(\tilde{s} \right)} \right|
\lesssim  \eta^2T  \ll \eta_{0},
	\end{equation*}
    and 
	\begin{equation*}
		\sup_{t \in \left[\tilde{\lambda}^{-2}a,\tilde{\lambda}^{-2}b \right]}|x_1(t)|
= \sup_{\tilde{s} \in  \left[s_1', s_2' \right]} \left|x_1 \left(\tilde{s} \right) \right|
\lesssim  T^{\frac{1}{50}} \eta^2T^2 + T^{\frac{1}{100}} \eta T \ll T^{\frac{1}{25}}.
	\end{equation*}
Therefore, Theorem \ref{t7.13} can be applied  on $ \left[\tilde{\lambda}^{-2}a,\tilde{\lambda}^{-2}b \right]$, yielding 
	\begin{equation}\label{11.12}
		\int_{\tilde{\lambda}^{-2}a}^{\tilde{\lambda}^{-2}b}  \left\| \tilde{\epsilon}(t)  \right\|_{L^2_{x}}^{2}
        {\lambda_1(t)^{-2}  } \, {d} t
\lesssim \| \tilde{\epsilon} (\tilde{\lambda}^{-2}a )  \|_{L^2_{x} }
+  \| \tilde{\epsilon} (\tilde{\lambda}^{-2}b )  \|_{L^2_{x} }+  \eta^{4} T^{2+1/10}+ O \left({T^{-8}} \right).
	\end{equation}
Changing variables, \eqref{11.12} implies 
	\begin{equation*}
		\int_{s_1}^{s_2}  \left\| \epsilon(s)  \right\|_{L^2_{x}}^{2}  \, {d} s
\lesssim  \| \epsilon(s_1)  \|_{L^2_{x}} +  \|\epsilon(s_2) \|_{L^2_{x} } + \eta^{4} T^{2+1/10}
 + O \left( {T^{-8}} \right).
	\end{equation*}
Therefore, we conclude the proof of Theorem \ref{t7.14}. 
\end{proof}

\section{Proof of Theorem \ref{t2.3}}\label{Sec:Thm2.3}

As stated in Section \ref{reductionsec}, we have reduced the proof of main Theorem \ref{main1} to a rigidity result, i.e.  Theorem  \ref{t2.3}, which will be shown in this section. We easily see that the scale parameter $\lambda(t)$ is constant for solitons and monotonically decreasing to zero for forward finite-time blowup solutions obtained by pseudo-conformal transformation of solitons. In the first subsection, we will show that this scale parameter is \emph{almost stable} around the orbit of the ground state $Q$, that is, under the assumptions of Theorem \ref{t2.3},  the parameter $\lambda(t)$ is \emph{almost monotonically decreasing}, i.e.  $\lambda(t)\sim \inf_{t'\in[0,t]}\lambda(t')$. In the second subsection, we will utilize this property to show Theorem \ref{t2.3}. For the infinite time blowup case, we use the almost monotonicity decreasing property of $\lambda(t)$ and the modulated virial inequality  derived in Section \ref{Sec:viria} to show that there exists a time sequence $t_n$ such that $\|\epsilon(t_n)\|_{L^2_x }$ decays rapidly in some sense. Combining this decay with the energy estimates in Section \ref{sec-en}, we show that the remainder $\epsilon$ vanishes. For the finite time blowup case, we apply the  pseudo-conformal transformation to $u$ to construct an infinite time blowup  solution $v$ and then prove that such $v$ must be the soliton.
\subsection{Almost monotonicity of the scaling parameter \texorpdfstring{$\lambda(t)$}{lambda(t)}}
First, we utilize Theorem \ref{t7.13} and  Theorem \ref{t7.14} to show a regularity property of the remainder term $\epsilon$. Especially, we prove that $ \left\| \epsilon(s) \right\|_{L^2_{x} }$ lies in $L_{s}^{p}$ for any $p>1$. We first show the case that $p=2$ as follows:
\begin{theorem}\label{t8.15}
Let $u$ be the solution to \eqref{1.1} satisfying $M(u_0)=M(Q)$, and suppose that
\begin{equation}\label{f8.1}
 \left\| \epsilon(0)  \right\|_{L^2_{x} } = \eta_{\ast},\quad \sup_{s\in [0,\infty)} \left\| \epsilon(s)  \right\|_{L^2_{x} }\leq \eta_{\ast}.
\end{equation}
Then for sufficiently small $\eta_{\ast}$,
\begin{equation}\label{f8.2}
\int_{0}^{\infty}  \left\| \epsilon(s)  \right\|_{L^2_{x} }^{2}  \, {d} s \lesssim \eta_{\ast},
\end{equation}
where the implicit constant is independent of $\eta_{\ast}$.

Moreover, for any $j \in \mathbb{N}$, we define
\begin{equation}\label{f8.3}
s_{j} = \inf \left\{ s \in [0, \infty) :  \left\| \epsilon(s)  \right\|_{L^2_{x} } = 2^{-j} \eta_{\ast}  \right\}.
\end{equation}
By definition, $s_{0} = 0$, and by Proposition \ref{sequ}, $s_{j}$ exists for any $j > 0$. Then for each $j \in \mathbb{N}$, we have
\begin{equation}\label{f8.4}
\int_{s_{j}}^{\infty}  \left\| \epsilon(s)  \right\|_{L^2_{x} }^{2}  \, {d} s \lesssim 2^{-j} \eta_{\ast},
\end{equation}
where the  implicit constant is independent of $j$ and $\eta_{\ast}$.
\end{theorem}

\begin{proof}
Set $T_{\ast} = \frac{1}{\eta_{\ast}}$. By \eqref{f4.45} and \eqref{f8.1}, for any $s' \geq 0$,
\begin{equation}\label{f8.5}
 \bigg|\sup_{s \in  \left[s', s' + T_{\ast} \right]} \ln (\lambda(s)) - \inf_{s \in  \left[s', s' + T_{\ast} \right]} \ln(\lambda(s)) \bigg|
  \lesssim \int_{s'}^{s'+T_{\ast}}  \left|\frac{\lambda'(s)}{\lambda(s)} \right| \, {d} s \lesssim \eta_{\ast} T_{\ast}\lesssim 1,
\end{equation}
where the implicit constant is independent of $s' \geq 0$. Denote $J$ to be the largest dyadic integer satisfying 
\begin{equation}\label{f8.6}
J = 2^{j_{\ast}} \leq (-\ln (\eta_{\ast}))^{\frac14}.
\end{equation}
By \eqref{f8.5} and the triangle inequality,
\begin{equation}\label{f8.7}
\aligned
\bigg|\sup_{s \in \left[s', s' + J T_{\ast} \right]} \ln(\lambda(s)) - \inf_{s \in  \left[s', s' + J T_{\ast} \right]} \ln(\lambda(s)) \bigg| \lesssim J. 
\endaligned
\end{equation}
Hence,
\begin{equation}\label{f8.8}
\frac{\sup_{s \in  \left[s', s' + 3 J T_{\ast} \right]} \lambda(s)}{\inf_{s \in  \left[s', s' + 3JT_{\ast} \right]} \lambda(s)}
 \lesssim e^{CJ}\ll T_{\ast}^{\frac1{100}}.
\end{equation}
Since $\eta_{\ast}JT_{\ast}\lesssim J, $ we can choose $T_{\ast}$ sufficiently large   such that Theorem \ref{t7.14} holds. 
By Theorem \ref{t7.14}, we get for any $1\leq\gamma\leq3$, 
\begin{equation}\label{f8.9}
\begin{aligned}
\int_{s'}^{s'+\gamma JT_\ast}\|\epsilon(s)\|_{L_x^2}^2\,ds
\lesssim{}&\|\epsilon(s')\|_{L_x^2}
+\|\epsilon(s'+\gamma JT_\ast)\|_{L_x^2}\\
&+\eta_\ast^{2-1/10}J^{2+1/10}
+O(J^{-8}T_\ast^{-8}).
\end{aligned}
\end{equation}
which yields for any $s' > J T_{\ast}$, 
\begin{align*}
\int_{s'}^{s' + J T_{\ast}}  \left\| \epsilon(s)  \right\|_{L^2_{x} }^{2}  \, {d} s
\lesssim& \inf_{s \in  \left[s' - J T_{\ast}, s' \right]}  \left\| \epsilon(s)  \right\|_{L^2_{x} } + \inf_{s \in  \left[s' + J T_{\ast}, s' + 2J T_{\ast} \right]}
\left\| \epsilon(s)  \right\|_{L^2_{x} }\\
&+  \eta_{\ast}^{2-1/10}J^{2+1/10}  + O \left( {J^{-8} T_{\ast}^{-8}} \right).
\end{align*}
Thus, by the integral mean value theorem, $\forall \, s' \geq 0$,
\begin{align*}
& 	\int_{s'}^{s' + J T_{\ast}}  \left\| \epsilon(s) \right\|_{L^2_{x} }^{2} \,{d}s
\\
 \lesssim&  \left\| \epsilon  \left(s' \right) \right \|_{L^2_{x}}
 + \frac{1}{J^{1/2} T_{\ast}^{1/2}} \left(\sup_{a\in\mathbb{N}} \int_{s' + a J T_{\ast}}^{s' + (a + 1) J T_{\ast}} \left\| \epsilon(s) \right\|_{L^2_{x} }^{2}
  \, {d} s \right)^{1/2}\\
  & +  \eta_{\ast}^{2-1/10}J^{2+1/10}  + O \left( {J^{-8} T_{\ast}^{-8}} \right),
\end{align*}
and 
\begin{align*}
& \sup_{a\in\mathbb{N}_+} \int_{s' + a J T_{\ast}}^{s' + (a + 1) J T_{\ast}} \left\| \epsilon(s) \right\|_{L^2_{x} }^{2} \,{d}s
 \\
  \lesssim &\frac{1}{J^{1/2} T_{\ast}^{1/2}} \left(\sup_{a \in \mathbb{N}} \int_{s' + a J T_{\ast}}^{s' + (a + 1) J T_{\ast}}
 \left\| \epsilon(s) \right\|_{L^2_{x} }^{2} \, {d} s \right)^{1/2} +   \eta_{\ast}^{2-1/10}J^{2+1/10}  + O \left( {J^{-8} T_{\ast}^{-8}} \right),
\end{align*}
which implies
\begin{align}\label{qmzy}
	& \sup_{a\in\mathbb{N}} \int_{s' + a J T_{\ast}}^{s' + (a + 1) J T_{\ast}} \left\| \epsilon(s) \right\|_{L^2_{x} }^{2} \,{d}s
\\\nonumber
 \lesssim& \left\|\epsilon \left(s' \right) \right\|_{L^2_x }+\frac{1}{J^{1/2} T_{\ast}^{1/2}}
\left(\sup_{a \in \mathbb{N}} \int_{s' + a J T_{\ast}}^{s' + (a + 1) J T_{\ast}}  \left\| \epsilon(s) \right\|_{L^2_{x} }^{2}  \, {d} s \right)^{1/2}\\\notag
&+  \eta_{\ast}^{2-1/10}J^{2+1/10}  + O \left( {J^{-8} T_{\ast}^{-8}} \right).
\end{align}
Taking $s' = s_{j_{\ast}}$ and applying the Cauchy-Schwarz inequality, for sufficiently small $\eta_{\ast}$, we obtain
\begin{equation}\label{f8.14}
\sup_{a \in \mathbb{N}} \int_{s_{j_{\ast}} + a J T_{\ast}}^{s_{j_{\ast}} + (a + 1) J T_{\ast}}  \left\| \epsilon(s) \right\|_{L^2_{x} }^{2} \, {d} s
 \lesssim 2^{-j_{\ast}} \eta_{\ast} +  \eta_{\ast}^{2-1/10}J^{2+1/10}  + O \left( {J^{-8} T_{\ast}^{-8}} \right)\lesssim 2^{-j_{\ast}}\eta_{\ast}.
\end{equation}
This together with the triangle inequality implies
\begin{equation*}
\sup_{s' \geq s_{j_{\ast}}} \int_{s'}^{s' + J T_{\ast}}  \left\| \epsilon(s) \right\|_{L^2_{x} }^{2}  \, { d} s \lesssim 2^{-j_{\ast}} \eta_{\ast},
\end{equation*}
and by H{\"o}lder's inequality,
\begin{equation*}
\sup_{s' \geq s_{j_{\ast}}} \int_{s'}^{s' + J T_{\ast}}  \left\| \epsilon(s) \right\|_{L^2_{x} } \, {d} s \lesssim 1.
\end{equation*}
Repeating this argument, Theorem $\ref{t8.15}$ can be proved by induction. Indeed, fix a constant $C < \infty$ and suppose that there exists a positive integer $n_{0}$ such that for all integers $0 \leq n \leq n_{0}$,
\begin{equation}\label{f8.17}
\sup_{s' \geq s_{nj_{\ast}}} \int_{s'}^{s' + J^{n} T_{\ast}}
\left\| \epsilon(s) \right\|_{L^2_{x} } \, {d} s \leq C, \quad \sup_{s' \geq s_{nj_{\ast}}}
\int_{s'}^{s' + J^{n} T_{\ast}}  \left\| \epsilon(s) \right\|_{L^2_{x} }^{2} \, {d} s \leq C J^{-n} \eta_{\ast}.
\end{equation}
Then, by \eqref{f8.5}, for $s' \geq s_{n j_{\ast}}$,
\begin{equation}\label{f8.18}
\frac{\sup_{s \in  \left[s', s' + 3 J^{n + 1} T_{\ast} \right]} \lambda(s)}{\inf_{s \in  \left[s', s' + 3 J^{n + 1} T_{\ast} \right]} \lambda(s)} \ll T_{\ast}^{\frac{1}{100}}.
\end{equation}
\begin{remark}
The $C$ in \eqref{f8.17} will ultimately be given by the implicit constants in Theorems \ref{t7.13} and  Theorem \ref{t7.14}, so for $T_{\ast}$ sufficiently large, \eqref{f8.18} will hold.
\end{remark}
Using \eqref{f8.17} and arguing as in the proof of  Theorem \ref{t7.14}, it is easy to see that for any $1\leq \gamma\leq 3$ and $s'\geq s_{nj_{\ast}}$,
\begin{align}\label{f8.19}
& 	\int_{s'}^{s'+\gamma J^{n+1}T_{\ast}}
\left\| \epsilon(s) \right\|_{L^2_{x} }^{2} \, {d} s \\\notag
 \lesssim&  \left\| \epsilon \left(s' \right)  \right\|_{L^2_{x} }
+  \left\| \epsilon \left(s'+\gamma J^{n+1}T_{\ast} \right)  \right\|_{L^2_{x} }+C J^{\frac{n+1}{10}}T_{\ast}^{\frac{1}{10}}J^{1-n}\eta_{\ast}\int_{s'}^{s'+\gamma J^{n+1}T_{\ast}}  \left\| \epsilon(s)  \right\|_{L^2_{x} }^{2} \, {d} s\\
&+ O \left({J^{-8(n+1)} T_{\ast}^{-8}} \right). \notag
\end{align}
Replacing \eqref{f8.9} with \eqref{f8.19} and repeating the proof of \eqref{qmzy}, one can prove that
\begin{align*}
	& \sup_{a\in\mathbb{N}} \int_{s' + a J^{n+1} T_{\ast}}^{s' + (a + 1) J^{n+1} T_{\ast}}
 \left\| \epsilon(s) \right\|_{L^2_{x} }^{2}  \,{d}s \\
   \lesssim & \left\|\epsilon \left(s' \right) \right\|_{L_x^2}+\frac{1}{J^{\frac{n+1}{2}} T_{\ast}^{1/2}} \left(\sup_{a \in \mathbb{N}} \int_{s' + a J^{n+1} T_{\ast}}^{s' + (a + 1) J^{n+1} T_{\ast}}  \left\| \epsilon(s) \right\|_{L^2_{x} }^{2} \, { d} s \right)^{1/2} \\
&     +   C J^{(n+1)/10}T_{\ast}^{1/10}J^{1-n}\eta_{\ast}\int_{s'}^{s'+\gamma J^{n+1}T_{\ast}}  \left\| \epsilon(s)  \right\|_{L^2_{x} }^{2} \, { d} s + O \left({J^{-8(n+1)} T_{\ast}^{-8}} \right), 
\end{align*}
which has the same implicit constant as \eqref{f8.9}.
Since $J\gg 1$ and $j_{\ast}\geq 0$, under the conditions required by \eqref{f8.14} ($\eta_{\ast}$ is sufficiently small), we get
\begin{equation*}
	\sup_{s' \geq s_{(n + 1) j_{\ast}}} \int_{s'}^{s' + J^{n + 1} T_{\ast}} \left\| \epsilon(s) \right\|_{L^2_{x} }^{2} \, {d} s \leq C J^{-(n + 1)} T_{\ast}^{-1},
\end{equation*}
and the implicit constant can be considered to be the same as in \eqref{f8.14}.
Then by H{\"o}lder's inequality,
\begin{equation*}
\sup_{s' \geq s_{(n + 1) j_{\ast}}} \int_{s'}^{s' + J^{n + 1} T_{\ast}}  \left\| \epsilon(s) \right\|_{L^2_{x} }  \, {d} s \leq C.
\end{equation*}
Therefore, \eqref{f8.17} holds for any integer $n > 0$.

Now, take any $j \in \mathbb{N}_+$ and suppose $n j_{\ast} < j \leq (n + 1) j_{\ast}$. Then \begin{equation}\label{f8.23}
	\sup_{s' \geq s_{j}} \int_{s'}^{s' + J^{n+1} T_{\ast}}  \left\| \epsilon(s) \right\|_{L^2_{x} }  \, {d} s \leq CJ, \quad \sup_{s' \geq s_{j}}
 \int_{s'}^{s' + J^{n+1} T_{\ast}}  \left\| \epsilon(s) \right\|_{L^2_{x} }^{2}  \, {d} s \leq C J^{1-n} \eta_{\ast}.
\end{equation}
Once again, we can prove that \eqref{f8.23} implies a sharper bound:
\begin{equation}\label{f8.24}
\sup_{a \geq 0} \int_{s_{j} + a J^{n + 1} T_{\ast}}^{s_{j} + (a + 1) J^{n + 1} T_{\ast}}
\left\| \epsilon(s) \right\|_{L^2_{x} }^{2} \, {d} s \lesssim 2^{-j} \eta_{\ast},
\end{equation}
and by H\"older's inequality,
\begin{equation}\label{f8.25}
\sup_{s' \geq s_{j}} \int_{s'}^{s' + 2^{j} T_{\ast}}  \left\| \epsilon(s) \right\|_{L^2_{x} } \, {d} s \lesssim 1,
\end{equation}
with the bound independent of $j$. Finally, \eqref{f8.24} and \eqref{f8.25} also yield a sharper bound:
\begin{equation}\label{f8.26}
\int_{s_{j}}^{s_{j} + 2^{j} J T_{\ast}}  \left\| \epsilon(s) \right\|_{L^2_{x} }^{2}  \, {d} s \lesssim 2^{-j} \eta_{\ast},
\end{equation}
and therefore, by the integral mean value theorem,
\begin{equation*}
\inf_{s \in  \left[s_{j}, s_{j} + 2^{j} J T_{\ast} \right]}  \left\| \epsilon(s) \right\|_{L^2_{x} } \lesssim 2^{-j} \eta_{\ast} J^{-1/2},
\end{equation*}
which implies
\begin{equation*} 
s_{j + 1} \in  \left[s_{j}, s_{j} + 2^{j} J T_{\ast} \right].
\end{equation*}
Thus, by \eqref{f8.26} and H{\"o}lder's inequality,
\begin{equation*}
\int_{s_{j}}^{s_{j + 1}}  \left\| \epsilon(s) \right\|_{L^2_{x} }^{2} \, {d} s
 \lesssim 2^{-j} \eta_{\ast}, \quad \text{and} \quad \int_{s_{j}}^{s_{j + 1}} \left\| \epsilon(s) \right\|_{L^2_{x} } \, {d} s \lesssim 1,
\end{equation*}
with constants independent of $j$. Summing over $j$ yields \eqref{f8.2} and \eqref{f8.4}.
\end{proof}

As a consequence of Theorem \ref{t8.15} and \eqref{f4.56}, we get the $L_s^p$-bound.
\begin{corollary}\label{cor:epsLp}
	Let $u$ be the solution to \eqref{1.1} satisfying $M(u_0)=M(Q)$, and suppose that
	\begin{equation*}
	\left\| \epsilon(0)  \right\|_{L^2_{x} } = \eta_{\ast},\quad \sup_{s\in [0,\infty)} \left\| \epsilon(s)  \right\|_{L^2_{x} }\leq \eta_{\ast}.
	\end{equation*}
Then, for any $p>1$,
	\begin{equation}\label{equ:epsiLpbound}
		\int_{0}^{\infty}  \left\| \epsilon(s)  \right\|_{L^2_{x} }^{p} \, { d} s \lesssim \eta_{\ast}^{p-1},
	\end{equation}
where the implicit constant depends only on $p$ when $\eta_{\ast} \ll 1$ is sufficiently small. 
\end{corollary}
\begin{proof}
	By \eqref{f4.56} and \eqref{f8.2}, we have 
	\begin{equation}\label{f8.32}
		\lim_{s \rightarrow \infty} \left\| \epsilon(s) \right\|_{L^2_{x} } = 0.
	\end{equation}
For $p\geq2$, \eqref{f8.1} and \eqref{f8.2} give the assertion.
For $1<p<2$, the definition of $s_j$ and \eqref{f8.4} imply
$s_{j+1}-s_j\lesssim2^jT_\ast$, where $T_\ast=\eta_\ast^{-1}$.
Thus, by H\"older's inequality,
\begin{align*}
 \int_{s_j}^{s_{j+1}}\|\epsilon(s)\|_{L_x^2}^p\,ds
 &\leq (s_{j+1}-s_j)^{1-p/2}
 \left(\int_{s_j}^{s_{j+1}}\|\epsilon(s)\|_{L_x^2}^2\,ds\right)^{p/2}\\
 &\lesssim \eta_\ast^{p-1}2^{-j(p-1)}.
\end{align*}
Summing over $j$ proves \eqref{equ:epsiLpbound}.
\end{proof}

\begin{remark}\label{remarkz}
Comparing with Theorem \ref{t2.3}, these theorems require an additional condition:
\begin{align*}
\left\|\epsilon(0) \right\|_{L^2_x }=\sup_{s\in[0,\infty)} \left\|\epsilon(s) \right\|_{L^2_x }=\eta_{\ast},
\end{align*}
which, one may check directly (using \eqref{f4.4}), that can be relaxed to
\begin{align}\label{weakerc}
\left\|\epsilon(0) \right\|_{L^2_x }\geq\frac{99}{100}\sup_{s\in[0,\infty)} \left\|\epsilon (s)  \right\|_{L^2_x }>0.
\end{align}
\eqref{weakerc} can be satisfied due to the time translation and rescaling if $u$ is not a soliton. Similarly, Theorem \ref{t9.16} also holds under the weaker condition \eqref{weakerc}.
\end{remark}

Now, we utilize Corollary \ref{cor:epsLp} to show that $\lambda(s)$ (or $\lambda(t)$) is an approximately monotone decreasing function.
\begin{theorem}\label{t9.16}
Let $u$ be the solution in Theorem \ref{t2.3}, and suppose that
\begin{equation*}
	\left\| \epsilon(0)  \right\|_{L^2_{x} } = \eta_{\ast},\qquad \sup_{t\in I} \left\| \epsilon(t)  \right\|_{L^2_{x} }\leq \eta_{\ast}.
\end{equation*}
 For any $s \geq 0$, let
\begin{equation*} 
\tilde{\lambda}(s) = \inf_{\tau \in [0, s]} \lambda(\tau).
\end{equation*}
Then, for any $s \geq 0$,
\begin{equation*} 
1 \leq \frac{\lambda(s)}{\tilde{\lambda}(s)} <e.
\end{equation*}

\end{theorem}

\begin{proof}
By contradiction, we suppose that there exist $0 \leq s_{-} \leq s_{+} < \infty$ satisfying
\begin{equation*} 
\frac{\lambda(s_{+})}{\lambda(s_{-})} = e.
\end{equation*}
Then we will show that $u(t)$ is a soliton solution to \eqref{1.1}, which is a contradiction since $\lambda(s)$ is constant in that case.
A simple computation shows
\begin{align*}
\left(\nabla Q, |x|^2 Q \right)_{L^2_{x}} =&  \left( iQ, |x|^2 Q \right)_{L^2_{x}} =  \left( i \nabla Q, |x|^2 Q \right)_{L_x^2} = 0,\\
  \left( \frac{d}{2}Q + x \cdot \nabla Q, |x|^{2} Q \right)_{L^2_{x}} =& -\| |x|Q \|_{L^2_{x}}^{2},\\
 L_{-} \left(|x|^2 Q \right) =& -2dQ - 4 x \cdot \nabla Q.
\end{align*}
Using these identities and Theorem \ref{t8.15}, and by performing modulation calculations similar to those in \cite{D2, D1}, we have
\begin{equation*} 
\begin{aligned}
&(\epsilon(s_+),|x|^2Q)_{L_x^2}
-(\epsilon(s_-),|x|^2Q)_{L_x^2}+\||x|Q\|_{L_x^2}^2\\
&\quad+4\int_{s_-}^{s_+}
\left(\epsilon_2(s),\frac d2Q+x\cdot\nabla Q\right)_{L_x^2}\,ds
=O(\eta_\ast).
\end{aligned}
\end{equation*}
This together with
\begin{equation*}
	\left( \epsilon(s_+), |x|^2Q \right)_{L^2_x } ,  \left( \epsilon(s_-), |x|^2Q \right)_{L^2_x }
\lesssim \sup_{s\in[s_-,s_+]} \left\|\epsilon(s) \right\|_{L^2_x } \left\| |x|^2Q \right\|_{L_x^2}=O(\eta_{\ast})
\end{equation*}
yields
\begin{equation}\label{f9.8}
	\big\| |x| Q  \big\|_{L^2_{x}}^{2} + 4 \int_{s_{-}}^{s_{+}}  \left( \epsilon_2(s), \frac{d}{2}Q + x\cdot\nabla Q \right)_{L^2_{x}} \,{d}s
= O (\eta_{\ast}).
\end{equation}
Therefore, there exists $s' \in [s_{-}, s_{+}]$ such that
\begin{equation}\label{f9.9}
\left( \epsilon_2(s'),  \frac{d}{2}Q + x \cdot \nabla Q \right)_{L^2_{x} } < 0.
\end{equation}
Let $T_{\ast}=\frac{1}{\eta_{\ast}}$, and let $s_j$ have the same meaning as in Theorem \ref{t8.15}. 
Since $s' \geq 0$, there exists a unique $j \geq 0$ such that $s_{j} \leq s'< s_{j + 1}$. Observing that $|s_{j+1}-s_j|\leq C 2^{j}T_{\ast}$, we have  $s_{j+1}\in \left[s',s'+C2^jT_{\ast} \right]$.
By \eqref{f8.4} and the integral mean value theorem, there exists $\tilde{s}_1\in  \left[0,2^jT_{\ast} \right]$ satisfying
\begin{align}\label{f9.10}
	\|\epsilon \left(s'+C2^jT_{\ast}+\tilde{s}_1 \right)\|_{L^2}\lesssim& 2^{-j}\eta_{\ast},\\\label{f9.11}
\int_{s'}^{s'+ C2^{j}T_{\ast}+\tilde{s}_1}  \left\|\epsilon(s) \right\|^2_{L^2_x } \, {d} s
\lesssim 2^{-j}\eta_{\ast},&\quad\int_{s'}^{s'+ C2^{j}T_{\ast}+\tilde{s}_1}  \left\|\epsilon(s) \right\|_{L^2_x }  \, {d} s \lesssim 1,
\end{align}
with implicit constants in \eqref{f9.10} and \eqref{f9.11}  independent of $\eta_{\ast}$. 
Similarly, for any $k\in \mathbb{N}_+$, there exists $\tilde{s}_k\in  \left[0,2^{j+k}T_{\ast} \right]$ such that
\begin{equation*}
s_{j+k+1}\in \Big[s'+C\sum_{i=j}^{j+k-1}2^iT_{\ast}+\sum_{m=1}^{k-1}\tilde{s}_m,\  s'+C\sum_{i=j}^{j+k-1}2^iT_{\ast}+\sum_{m=1}^{k-1}\tilde{s}_m+C2^{k+j}T_{\ast}+\tilde{s}_k\Big],
\end{equation*}
 \begin{equation}\label{f9.13}
	\bigg\|\epsilon \Big(s'+C\sum_{i=j}^{j+k-1}2^iT_{\ast}+\sum_{m=1}^{k-1}\tilde{s}_m+C2^{k+j}T_{\ast}+\tilde{s}_k \Big)\bigg\|_{L_x^2}
\lesssim 2^{-j-k}\eta_{\ast},
\end{equation}
\begin{equation}\label{f9.14}
\begin{aligned}
	\int_{s'+C\sum\limits_{i=j}^{j+k-1}2^iT_{\ast}+\sum\limits_{m=1}^{k-1}\tilde{s}_m}^{s'+C\sum\limits_{i=j}^{j+k}2^iT_{\ast}+\sum\limits_{m=1}^{k}\tilde{s}_m}
\left\|\epsilon(s) \right\|^2_{L^2_x }\, { d} s
&\lesssim 2^{-j-k}\eta_{\ast},\\\int_{s'+C\sum\limits_{i=j}^{j+k-1}2^iT_{\ast}+\sum\limits_{m=1}^{k-1}\tilde{s}_m}^{s'+C\sum\limits_{i=j}^{j+k}2^iT_{\ast}+\sum\limits_{m=1}^{k}\tilde{s}_m} \left\|\epsilon(s) \right\|_{L^2_x } \, {d} s &\lesssim 1,
\end{aligned}
\end{equation}
with implicit constants in \eqref{f9.13} and \eqref{f9.14} independent of $\eta_{\ast}$ and $k$. 
Let $T_k=C\sum_{i=j}^{j+k}2^iT_{\ast}+\sum_{m=1}^{k}\tilde{s}_m$, and  suppose that for any $1\leq k \leq n$, there exists constant $C_0>0$ so that
\begin{equation}\label{qmm}
	\int_{s'}^{s'+T_k}  \left\|\epsilon(s) \right\|^2_{L^2_x } \, {d} s
 \leq C_02^{-j-k}\eta_{\ast},\quad\int_{s'}^{s'+T_k}  \left\|\epsilon(s) \right\|_{L^2_x } \, {d} s \leq C_0,
\end{equation}
then we claim that for $C_{0}$ sufficiently large,
\begin{equation}\label{induction10B}
	\int_{s'}^{s'+T_{k+1}}  \left\|\epsilon(s) \right\|^2_{L^2_x } \, {d} s \leq C_02^{-j-k-1}\eta_{\ast}.
\end{equation}
In fact, by  \eqref{f8.4}, \eqref{f9.14} and \eqref{qmm}, we have
\begin{equation*} 
\int_{s'}^{s'+T_{k+1}}
\left\|\epsilon(s) \right\|^2_{L^2_x } \, {d} s
\leq C_02^{-j-k}\eta_{\ast}+C_12^{-j-k}\eta_{\ast},\quad\int_{s'}^{s'+T_{k+1}}  \left\|\epsilon(s) \right\|_{L^2_x } \, {d} s \leq C_0+C_1
\end{equation*}
for some constant $C_1$, which is also independent of $\eta$ and $k$. As in the proof of Theorem \ref{t7.14}, we can set
\begin{equation*}
	1\leq \lambda(s)\leq T_{k+1}^{1/100},
\quad   s\in \left[s',s'+T_{k+1} \right],\quad \xi \left(s'+T_{k+1} \right)=x \left(s' \right)=0,
\end{equation*}
and by \eqref{f4.54} and \eqref{f4.55},
\begin{equation*}
	\sup_{s\in \left[s',s'+T_{k+1} \right]}\frac{|\xi(s)|}{\lambda(s)}
\lesssim \left(\int_{s'}^{s'+T_{k+1}}  \left\| \epsilon(s) \right\|_{L^2_{x} }^{2} \, { d} s \right)\ll\eta_0.
\end{equation*}
\begin{equation*}
\sup_{s\in[s',s'+T_{k+1}]}|x(s)|\ll T_{k+1}^{1/25}.
\end{equation*}
Noticing that $C2^{j+k+1}T_{\ast}\leq T_{k+1}\leq 4C2^{j+k+1}T_{\ast}$, by Theorem \ref{t7.13} and \eqref{f9.9},
\begin{align}\label{8.45}
& \int_{s'}^{s'+T_{k+1}}  \left\| \epsilon(s) \right\|_{L^2_{x} }^{2} \, {d} s\\\notag
\lesssim&   
\left( \epsilon_2\left(s' \right),  \frac{d}{2}Q + x \cdot \nabla Q \right)_{L^2_{x} } 
+ \left\|\epsilon \left(s'+T_{k+1} \right) \right\|_{L^2_x }\\\notag
&+T_{k+1}^{1/10} \left(\int_{s'}^{s'+T_{k+1}}  \left\| \epsilon(s) \right\|_{L^2_{x} }^{2}  \, {d} s \right)^2+O \left( 
{T_{k+1}^{-8}} \right) \\\notag
\lesssim& 2^{-j-k-1}\eta_{\ast}+T_{k+1}^{1/10} \left(\int_{s'}^{s'+T_{k+1}}  \left\| \epsilon(s) \right\|_{L^2_{x} }^{2}  \, {d} s \right)^2+O \left( { T_{k+1}^{-8}} \right) \\\notag
\lesssim& 2^{-j-k-1}\eta_{\ast}+ \left(4C2^{j+k+1}T_{\ast} \right)^{1/10}\eta_{\ast}^2(C_0+C_1)^22^{-2j-2k}+\frac{1}{C^82^{8j+8k+8}T^8_{\ast}}.
\end{align}
 Moreover, the implicit constant in \eqref{8.45} only relies on that in  \eqref{f9.13} and Theorem \ref{t7.14}, which are both independent of $\eta_{\ast}$ and $k$. 
 Thus, by choosing $\eta_{\ast}$ sufficiently small and $C_0$ sufficiently large, \eqref{induction10B} holds.

Now we let $k \nearrow \infty$, 
 then \eqref{induction10B} implies 
\begin{equation*} 
\int_{s'}^{\infty}  \left\| \epsilon(s) \right\|_{L^2_{x} }^{2} \, {d} s = 0,
\end{equation*}
which directly yields that $ \epsilon (s) = 0$ for all $s \geq s'$. This is a contradiction since $u$ is not a soliton.
\end{proof}

\subsection{Proof of Theorem \ref{t2.3} }
Finally, we turn to prove Theorem \ref{t2.3}. We first consider the infinite time blowup case:
\subsubsection{Part A: infinite time blowup}
\begin{theorem}\label{t9.17}
If $u$ is the solution in Theorem \ref{t2.3} and
\begin{equation*} 
\sup(I) = \infty,
\end{equation*}
then $u$ is a soliton solution of the form \eqref{soliton}.
\end{theorem}
\begin{proof}
If $u$ is not a soliton of the form \eqref{soliton}, then $\sup\limits_{s\in[0,\infty)} \left\|\epsilon(s) \right\|_{L^2_x }>0$. By Remark~\ref{remarkz}, we may first translate time so that
\eqref{weakerc} holds. Then \eqref{f8.32} implies that the remainder
tends to zero. Its positive maximum on the remaining half-line is
therefore attained. Translating once more and relabeling this
maximum as $\eta_\ast$, we obtain 
\begin{align}\label{contra}
\left\|\epsilon(0) \right\|_{L^2_x }=\sup_{s\in[0,\infty)} \left\|\epsilon(s) \right\|_{L^2_x }=\eta_{\ast}.
\end{align}
Additionally, by rescaling, we can set $\tilde{\lambda}(s)\leq  1$ for all $s\geq0$.

\textbf{Case 1.} $\tilde{\lambda}(s)\searrow 0$. In this case, for any integer $k \geq 0$, let
\begin{equation}\label{f10.3}
I(k) =  \left\{ s \geq 0 : 2^{-k + 2} \leq \tilde{\lambda}(s) \leq 2^{-k + 3}  \right\}.
\end{equation}
Since $\sup(I)=\infty$, there exists a sequence $k_{n} \nearrow \infty$ such that
\begin{equation*} 
|I(k_{n})| 2^{-2k_{n}} \geq \frac{1}{k_{n}^{2}},
\end{equation*}
and that $|I(k_{n})| \geq |I(k)|$ for all $k \leq k_{n}$.

Let $I(k_n)=[a_n,b_n]$. We then claim that for any sufficiently large $n$, there exists $s_{n} \in [\frac{3 a_{n} + b_{n}}{4} ,\frac{a_{n} + 3 b_{n}}{4}]\subset I(k_n)$ such that
\begin{equation}\label{f10.5}
\left\| \epsilon(s_{n})  \right\|_{L^2_{x} } \lesssim \frac{1}{|I(k_n)|}.
\end{equation}
Indeed, by Theorem \ref{t8.15},
\begin{equation*}
\int_{I(k_{n})}  \left\| \epsilon(s)  \right\|_{L^2_{x} }^{2}  \, {d} s \lesssim \eta_{\ast},
\end{equation*}
then by the virial identity in \eqref{f9.8}, we have
\begin{equation}\label{f10.7}
\int_{a_{n}}^{\frac{3 a_{n} + b_{n}}{4}}  \left( \epsilon_2, \frac{d}{2}Q + x \cdot \nabla Q \right)_{L^2_{x} } \, {d} s
= O \left(\eta_{\ast} \right) + O(1).
\end{equation}
Therefore, by the integral mean value theorem, there exists $s_{n}^{-} \in  \left[a_{n}, \frac{3 a_{n} + b_{n}}{4} \right]$ such that
\begin{equation}\label{f10.8}
 \left| \left( \epsilon_2 \left(s_{n}^{-} \right), \frac{d}{2}Q + x \cdot \nabla Q \right)_{L^2_{x} } \right| \lesssim \frac{1}{|I(k_{n})|}.
\end{equation}
Similarly, there exists $s_{n}^{+} \in  \left[\frac{a_{n} + 3 b_{n}}{4}, b_{n} \right]$ such that
\begin{equation}\label{f10.9}
 \left| \left( \epsilon_2 \left(s_{n}^{+} \right),\frac{d}{2}Q + x \cdot \nabla Q \right)_{L^2_{x} } \right|
  \lesssim \frac{1}{|I(k_{n})|}.
\end{equation}
Splitting the interval $ \left[s_n^-,s_n^+ \right]$ into $z_n$ small intervals of length $T_{\ast}=\frac{1}{\eta_{\ast}}$, we proceed as follows. 
		
\textbf{Scenario I.}
If 
 $z_n \leq J=2^{j_{\ast}}$, then, as in \eqref{f8.5}-\eqref{f8.8}, we can similarly prove that
\begin{equation*}
\frac{\sup\limits_{s \in  \left[s_n^-, s_n^+ \right]} \lambda(s)}{\inf\limits_{s \in  \left[s_n^-, s_n^+ \right]} \lambda(s)}
\lesssim e^{\ln (T_{\ast})^{1/4}} \lesssim e^{\ln \left(T_{\ast}^{1/1000} \right)} \ll T_{\ast}^{1/100}.
\end{equation*}
Therefore, after normalizing as in the proof of Theorem~\ref{t7.14}, Theorem~\ref{t7.13} together with \eqref{f10.8} and \eqref{f10.9} implies
\begin{equation*} 
\int_{s_{n}^{-}}^{s_{n}^{+}}  \left\| \epsilon(s)  \right\|_{L^2_{x} }^{2} \, {d} s \lesssim \frac{1}{|I(k_{n})|}.
\end{equation*}

\textbf{Scenario II.}
If  $z_n > J=2^{j_{\ast}}$, then there exists $\ell_n\in\mathbb{N}$ such that $J^{\ell_n}<z_n\leq J^{\ell_n+1}$. Repeating the induction in the proof of Theorem~\ref{t8.15}, using \eqref{f8.17} and \eqref{f8.19}, we obtain
\begin{align*}
	\int_{s_{n}^{-}}^{s_{n}^{+}} \left\| \epsilon(s)  \right\|_{L^2_{x} } \, {d} s\lesssim 1
\end{align*}
and
\begin{align*}
	\int_{s_{n}^{-}}^{s_{n}^{+}} \left\| \epsilon(s) \right\|_{L^2_{x} }^{2} \, {d} s\lesssim J^{1-\ell_n}\eta_{\ast} .
\end{align*}
Following the proof of Theorem \ref{t7.14}, we see that Theorem \ref{t7.13} can be applied on the interval $ \left[s_n^-,s_n^+ \right]$.
Thus \eqref{f10.8} and \eqref{f10.9} imply  
\begin{equation*}
\int_{s_{n}^{-}}^{s_{n}^{+}} \left\| \epsilon(s)  \right\|_{L^2_{x} }^{2}  \, {d} s
\lesssim \frac{1}{|I(k_{n})|} + J^{(\ell_n+1)/10} T_{\ast}^{1/10}J^{1-\ell_n}\eta_{\ast}
 \left(\int_{s_{n}^{-}}^{s_{n}^{+}} \left\| \epsilon(s)  \right\|_{L^2_{x} }^{2}  \, {d} s \right),
\end{equation*}
which further implies
\begin{align*}
\int_{s_{n}^{-}}^{s_{n}^{+}}  \left\| \epsilon(s)  \right\|_{L^2_{x} }^{2} \, {d} s	\lesssim \frac{1}{|I(k_{n})|}, 
\end{align*}
provided $k_n$ is sufficiently large. Thus \eqref{f10.5} holds.

Now, let $m$ be the smallest integer such that
 \begin{equation}\label{f10.12}
 	\frac{2^{2k_{n}}}{k_{n}^{2}} 2^{m} \geq |I(k_{n})|.
 \end{equation}
 Since $|I(k)| \leq |I(k_{n})|$ for all $0 \leq k \leq k_{n}$, \eqref{f10.12} implies that
 \begin{equation}\label{f10.13}
 	|s_{n}| \leq 2^{2k_{n} + m + 1}.
 \end{equation}
 Moreover, as $s_{n} \in [\frac{3 a_{n} + b_{n}}{4} ,\frac{a_{n} + 3 b_{n}}{4}]\subset I(k_n)$, we also have
 \begin{equation}\label{f10.13b}
      |s_{n}| \geq \frac{2^{2k_{n}+m}}{8k_n^2}.
 \end{equation}
 
 Let $r_{n}$ be the smallest integer satisfying 
 \begin{equation*} 
 	2^{\frac{2k_{n} + m + 1}{3}} \leq 2^{r_{n}}.
 \end{equation*}
Let $t_n=s^{-1}(s_n), \lambda_n=\frac{1}{4}2^{-k_n}$, and set $v_n(t,x)=\lambda_n^{\frac{d}{2}}u \left(\lambda_n^2t,\lambda_nx \right)$. 
After a spatial translation and a Galilean transformation, we continue to denote the resulting solution by $v_n$. Its modulation parameters $\tilde{\xi}(t)$, $\tilde{x}(t)$, and $\tilde{\gamma}(t)$ satisfy
	\begin{align} \label{yyz}
	&e^{i\tilde{\gamma}(t)}e^{ix\cdot\tilde{\xi}(t)}\tilde{\lambda}(t)^{d/2}v_n \left(t,\tilde{\lambda}(t)x+\tilde{x}(t) \right)
=\tilde{\epsilon}(t)+ Q, \\\notag 
	&\tilde{\lambda}(t)=\lambda_n^{-1}{\lambda \left(\lambda_n^2t \right)},\   t\in \left[0,\lambda_n^{-2}t_n \right],   \\\notag
	&\tilde{x}\left(\lambda_n^{-2}t_n\right)=\tilde{\xi} \left(\lambda_n^{-2}t_n \right)=0,\  \frac{ \left|\tilde{\xi}(t) \right|}{\tilde{\lambda}(t)}
 \lesssim \eta_{\ast},\    \left\|\tilde{\epsilon}(t) \right\|_{L^2_x }
 = \left\|\epsilon \left(\lambda_n^2t \right) \right\|_{L^2_x }, \\\notag
 &1\leq \tilde{\lambda}(t)\leq 16\cdot 2^{k_n},
 \  t\in  \left[0,\lambda_n^{-2}t_n \right],\\
	&\int_{0}^{\lambda_n^{-2}t_n} \frac{1}{\tilde{\lambda}(t)^2} \,{d}t =s_n  \ \mbox{and}\quad\tilde{\epsilon}(t) \mbox{ satisfies }\eqref{orthod}. \notag
\end{align}
The bound for $\widetilde\xi/\widetilde\lambda$ in \eqref{yyz}
follows from \eqref{f4.54} and \eqref{f8.2}.
Moreover, integrating \eqref{f4.54} and \eqref{f4.55} backwards from
$\lambda_n^{-2}t_n$, and using \eqref{f8.4} and Theorem~\ref{t9.16},
we obtain\footnote{Although this spatial-center bound is weaker than the
one assumed in Theorem~\ref{longtime}, the bound involving
$\tilde{\lambda}(t)^2$ gives only logarithmic factors in the spatial sums
\eqref{th1} and \eqref{zznn1}, which can be absorbed into $s_n^{1/25}$.
Thus the proof of Lemma~\ref{integ1} still yields the recursive estimate below.}
\begin{equation*}
 |\tilde{x}(t)|\lesssim\tilde{\lambda}(t)^2[1+\log(2+s_n)]
 \lesssim 2^{2k_n}[1+\log(2+s_n)],
 \quad t\in[0,\lambda_n^{-2}t_n].
\end{equation*}
Arguing as in the proof of Lemma~\ref{integ1}, with $u=v_n$
and $t_0=\lambda_n^{-2}t_n$ in \eqref{vvv1}, we obtain, for
$i\geq\lfloor\log_2(s_n)/4\rfloor+5$,
\begin{align*}
   & \|P_{\geq i}v_n\|_{U_{\Delta}^2\left([0,\lambda_{n}^{-2}t_n]\right)}\lesssim_{\delta} \left\|\tilde{\epsilon} \left(\lambda_n^{-2}t_n \right) \right\|_{L_x^2 } +O \left(s_n^{-10} \right)\\
    &+({\eta_{\ast}^{1-}}+\eta_{\ast}^{\frac{1}{2}-}+[s_n^{\frac{d}{d+2}\delta}+s_n^{\frac{d+4}{2(d+2)}\delta}](2^{-i}s_n^{1/25})^{1-\delta})\|P_{\geq i-4}v_n\|_{U_{\Delta}^2\left([0,\lambda_{n}^{-2}t_n]\right)}.
\end{align*}
Using \eqref{f10.5}, \eqref{f10.12}, \eqref{f10.13}, \eqref{f10.13b} and then making an induction on $i$, we obtain 
	\begin{equation}\label{f10.18}
	\left\| P_{\geq r_{n} + \frac{k_{n}}{4} + \frac{m}{4}} v_n  \right\|_{U^2_{\Delta} \left( \left[0,\lambda_{n}^{-2}t_n \right] \right)}
 \lesssim k_{n}^{2} 2^{-2k_{n}} 2^{-m},\, \mbox{ if } k_n \mbox{ is  sufficiently large. }
\end{equation}
Since $\tilde\xi(\lambda_n^{-2}t_n)=0$, the computations in
\eqref{f6.3}, together with \eqref{f10.5}, imply that
\begin{equation*}
 E\big(P_{\leq r_n+\frac{k_n}{4}+\frac{m}{4}+4}
 v_n(\lambda_n^{-2}t_n)\big)
 \lesssim
 \big(k_n^2 2^{-2k_n-m}
 2^{r_n+\frac{k_n}{4}+\frac{m}{4}}\big)^2.
\end{equation*}
Following the computations in \eqref{bbv} and \eqref{bbv2}, and using
\eqref{f10.18} to estimate the two high-frequency factors, we obtain
\begin{align*}
 &\sup_{t\in[0,\lambda_n^{-2}t_n]}
 \Big|E\big(P_{\leq r_n+\frac{k_n}{4}+\frac{m}{4}+4}v_n(t)\big)
 -E\big(P_{\leq r_n+\frac{k_n}{4}+\frac{m}{4}+4}
 v_n(\lambda_n^{-2}t_n)\big)\Big|\\
 &\qquad\lesssim
 2^{2r_n+\frac{k_n}{2}+\frac{m}{2}}
 \big\|P_{\geq r_n+\frac{k_n}{4}+\frac{m}{4}}v_n
 \big\|_{U_\Delta^2([0,\lambda_n^{-2}t_n])}^2
 \lesssim
 \big(k_n^2 2^{-2k_n-m}
 2^{r_n+\frac{k_n}{4}+\frac{m}{4}}\big)^2.
\end{align*}
Consequently,
\begin{equation}\label{f10.19}
 \sup_{t\in[0,\lambda_n^{-2}t_n]}
 E\big(P_{\leq r_n+\frac{k_n}{4}+\frac{m}{4}+4}v_n(t)\big)
 \lesssim k_n^4 2^{-\frac{13}{6}k_n-\frac{5}{6}m}
 \ll 2^{-\frac{21}{10}k_n}.
\end{equation}
Finally, the frequency-tail estimates in \eqref{ccc0}--\eqref{cccc0}
are uniform for $\tilde\lambda\geq1$, and the terms containing
$|\tilde\xi|^2/\tilde\lambda^2$ are absorbed by the positive term
in \eqref{ccc0}. Keeping the factor $\tilde\lambda^{-2}$ in
\eqref{ccc1}, the computations in \eqref{ccc0}, \eqref{ccc1},
\eqref{cccc0} and \eqref{cccc1}, together with \eqref{yyz} and
\eqref{f10.19}, give
\begin{equation}\label{f10.20}
 \|\epsilon(0)\|_{L_x^2}^2
 =\|\tilde\epsilon(0)\|_{L_x^2}^2
 \lesssim k_n^4 2^{-\frac16k_n-\frac56m}
 \ll 2^{-\frac1{10}k_n}.
\end{equation}
Letting $k_n\nearrow\infty$, we obtain $\|\epsilon(0)\|_{L_x^2}=0$,
which contradicts \eqref{contra}.

\textbf{Case 2.} There exists $c>0$ such that $\tilde{\lambda}(s)\geq c$.
In this case, Theorem \ref{t9.16} implies that $\lambda(s)>\tilde{c}$ for some constant $\tilde{c}>0$.
By rescaling, we can set
\begin{align}\label{qdz}
1\leq \lambda(s)\leq \frac{1}{\tilde{c}},\quad \forall\,  s>0.
\end{align}
Let $s'_n=\eta_{\ast}^{-1}2^{n}$, and let $\{s_j\}$ be the sequence defined in \eqref{f8.3}.
By \eqref{f8.4}, $s_j\lesssim2^jT_\ast$, and the integral mean value theorem,
there exist $0<\theta<1$ and a sequence $ \left\{\tilde{s}_n \right\}$ so that
$$ \left\|\epsilon \left(\tilde{s}_n \right) \right\|_{L^2_x }\lesssim \eta_{\ast} 2^{-n},\quad\tilde{s}_n\in  \left[\theta s'_n,s'_n \right]. $$
Setting $\tilde{t}_n=s^{-1} \left(\tilde{s}_n \right)$, we can apply a Galilean transform so that
$\xi \left(\tilde{t}_n \right)=0$. Arguing as in \eqref{f10.18} and \eqref{f10.19}, we have
	\begin{equation*}
	\left\| P_{\geq \tilde{r}_{n} + \frac{n}{8}} u(t)  \right\|_{U^2_{\Delta} \left( \left[0,\tilde{t}_n \right] \right)}
 \lesssim_{\eta_{\ast}} 2^{-n},\quad \mbox{ if } n \mbox{ sufficiently large. }
\end{equation*}
\begin{equation}\label{f10.23}
	\sup_{t\in \left[0,\tilde{t}_n \right]}E \left(P_{\leq \tilde{r}_{n} + \frac{n}{4}}  u(t) \right) \lesssim_{\eta_{\ast}} 2^{-\frac{5}{12}n},
\end{equation}
where $\tilde{r}_n=\left\lfloor\frac{1}{3}\log_2\tilde{s}_n\right\rfloor$.
Then as in \eqref{f10.20}, inequalities \eqref{qdz} and \eqref{f10.23} imply  
\begin{equation*} 
	 \left\|\epsilon(0) \right\|^2_{L^2_x }
\leq \sup_{t\in \left[0,\tilde{t}_n \right]} \left\|\epsilon(t) \right\|^2_{L^2_x }
\lesssim_{\tilde{c}} 2^{-\frac{5}{12}n}.
\end{equation*}
Letting $n\nearrow \infty$, we find that  $ \left\|\epsilon(0) \right\|_{L^2_x }=0$, which contradicts \eqref{contra}.
Thus, $u$ must be a soliton of the form \eqref{soliton}. \end{proof}
\subsubsection{Part B: finite time blowup}
Next, we address the finite time blowup case. 
Heuristically, we can use the pseudo-conformal transformation 
\[
u\rightarrow\frac{1}{|t|^{\frac{d}{2}}}\overline{u(\frac{1}{t},\frac{x}{t})}e^{\frac{i|x|^2}{4t}}
\]
to reduce the problem to the infinite time blowup case.

	\begin{theorem}\label{t9.18}
		If $u$ is the solution in Theorem \ref{t2.3} and
		\begin{equation*}
			\sup(I) <\infty,
		\end{equation*}
		then $u$ is a pseudo-conformal transformation of the soliton solution (of the form \eqref{pseudosoliton}).
	\end{theorem}
    
	\begin{proof}
	Without loss of generality, we assume $-1\in I$, $\sup(I) = 0$, and
		\begin{align*}
			\sup_{-1<t<0} \|\epsilon (t)\|_{L^2_x }\leq \eta_{\ast}.
		\end{align*}
		Applying the pseudo-conformal transformation, we obtain
\[
v(t,x)=|t|^{-d/2}
\overline{u\left(\frac1t,\frac{x}{t}\right)}
e^{i|x|^2/(4t)},\qquad t\leq-1.
\]
This is a solution of \eqref{1.1} which blows up backward in infinite
time. By \eqref{decomp}, with
$y=(x-tx(1/t))/(t\lambda(1/t))$, we have
\[
v(t,x)=\frac{e^{i\gamma(1/t)}e^{iy\cdot\xi(1/t)}e^{i|x|^2/(4t)}}
{(|t|\lambda(1/t))^{d/2}}
\left[Q(y)+\overline{\epsilon(1/t,y)}\right].
\]
Since
\[
\frac{|x|^2}{4t}
=\frac{|x-tx(1/t)|^2}{4t}
+\frac{x\cdot x(1/t)}2-\frac{t|x(1/t)|^2}4,
\]
removing the first term on the right from the phase of the $Q$ term
gives a function in the ground-state orbit. Thus, by \eqref{ppp}
and the radiality of $Q$,
\begin{align}\label{finite:orbit-distance}
\delta(v(t))
&\leq\|\epsilon(1/t)\|_{L_x^2}
+\left\|Q(y)\left(e^{it\lambda(1/t)^2|y|^2/4}-1\right)\right\|_{L_y^2}
\notag\\
&\leq\|\epsilon(1/t)\|_{L_x^2}
+\frac{|t|\lambda(1/t)^2}{4}\big\||y|^2Q\big\|_{L_y^2}.
\end{align}
Translate the origin of renormalized time to the physical time $-1$
and rescale so that $\lambda(s)\leq1$ for $s\geq0$.
Define $I(k)$ as in \eqref{f10.3}. By \eqref{f4.45} and
\eqref{f8.32}, there is a sequence $c_k\nearrow\infty$ such that
\[
|I(k)|\geq c_k,\qquad \lambda(s)\sim2^{-k}\quad(s\in I(k)),
\]
for all sufficiently large $k$.
For $s\in I(k)$, \eqref{f4.23} gives
\[
|t(s)|=\int_s^\infty\lambda(\sigma)^2\,d\sigma
\geq\int_{I(k+1)}\lambda(\sigma)^2\,d\sigma
\gtrsim c_{k+1}2^{-2k}.
\]
Consequently,
\[
\frac{\lambda(t)^2}{|t|}\longrightarrow0\quad(t\nearrow0),
\qquad |t|\lambda(1/t)^2\longrightarrow0\quad(t\to-\infty).
\]
Together with \eqref{f8.32} and \eqref{finite:orbit-distance},
this implies $\delta(v(t))\to0$ as $t\to-\infty$.
For some sufficiently negative $t_0$, the solution $v$ therefore
satisfies the backward-in-time version of the hypotheses of
Theorem~\ref{t2.3} on $(-\infty,t_0]$.
By time reversal and Theorem~\ref{t9.17}, $v$ is a solitary wave.
Hence $u$ is a pseudo-conformal transform of a solitary wave of the
form \eqref{pseudosoliton}.
\end{proof}
Now, Theorems \ref{t9.17} and \ref{t9.18} imply Theorem \ref{t2.3}, and by the reduction in Section \ref{reductionsec},   Theorem \ref{main1} holds. 

\appendix

\section{Modulation analysis}\label{sec-mod}
In this appendix, we collect the modulation estimates used in the proof of
Theorem~\ref{t2.3}. We first recall the spectral properties of the linearized
operators and then state the modulation decomposition and the coercivity estimate.

We need to linearize \eqref{1.1} at $Q$. For $u(t) = e^{it}(Q + \varepsilon(t))$  in \eqref{1.1} with $ \varepsilon = \varepsilon_1 + i \varepsilon_2$, we have
\begin{align} \label{lin-flow}
 \partial_t \begin{pmatrix}
    \varepsilon_1\\
    \varepsilon_2
\end{pmatrix} + \begin{pmatrix}
    0 &  -L_-\\
    L_+ & 0
\end{pmatrix} \begin{pmatrix}
    \varepsilon_1\\
    \varepsilon_2
\end{pmatrix} = \begin{pmatrix}
     \text{Im}(\mathcal{N}(Q, \varepsilon))\\
     -\text{Re}(\mathcal{N}(Q, \varepsilon))
\end{pmatrix},
\end{align}
where $\mathcal{N}(Q, \varepsilon)$ is nonlinear in $\varepsilon$ and the linearized operators $L_{\pm}$ are given by
\begin{align}\label{lin-op}
L_{+} &= -\Delta + 1 - \phi_{Q^2} - 2 Q \phi_{(\cdot\;  Q)},\\ \label{lin-op2}
L_{-} &= -\Delta + 1 - \phi_{Q^2},
\end{align}
with
\[ 
\phi_{Q^2} = |x|^{-2}\ast Q^2,\;\;\phi_{(\cdot \; Q)} =|x|^{-2} \ast (\;\cdot \; Q).
\]
Thus $L_{\pm}$ are posed on  $L^2(\R^d)$ with domain $ D = H^2(\R^d)\subset L^2(\R^d)$ and hence the operator $L_+$ is non-local.

It has been shown that the ground state $Q$ is non-degenerate in $d=4$, i.e.
\[
\ker L_{+}=\operatorname{span} \{\partial_{x_1}Q,\cdots,\partial_{x_d}Q\}
\]
We refer the readers to \cite{Chen,Li}. However, for other dimensions, the uniqueness and non-degeneracy of the ground state remain unknown, but they are widely believed to hold. Since Proposition~\ref{positive1} relies on non-degeneracy, we must make the uniqueness and non-degeneracy  assumptions in other dimensions, as in \cite{LZ28,Wang-Gao}.

In \cite{CMP1}, the authors established some fundamental properties of the linearized operators $L_{+}$ and $L_{-}$ for the case $d=4$. One can easily verify, however, that these properties remain valid in other dimensions under the assumption that the ground state $Q$ is unique and non-degenerate. We summarize these results in the following theorem:
\begin{theorem}[\cite{CMP1}]\label{posit}
The following four properties hold:
	\begin{enumerate}
		\item ${L}_{+}$ is a self-adjoint operator and $\sigma_{ess}({L}_{+}) = [1, \infty)$.

		\item $L_{+}$ has a unique negative eigenvalue $\lambda$ associated with an eigenfunction $\chi\in H^1$. 

        \item $( {L}_{-} u,u )_{L^2}\geq 0$ for any real-valued $u\in H^1$.
		
		\item There exists a constant $c_d>0$ such that if $u\in H^1(\mathbb{R}^d)$ is real-valued and satisfies $u\perp_{L^2}\{\chi,\partial_{x_1}Q,\cdots,\partial_{x_d}Q\}$, then $( {L}_{+} u,u )_{L^2}\geq c_d\|u\|^2_{H^1}$.
	\end{enumerate}

\end{theorem}

We also need the following regularity estimate of the eigenfunction $\chi$:
\begin{lemma}[Regularity of the eigenfunction]
For any $p\in (1,\infty)$ and integer $k\geq0$, we have
\begin{equation}\label{equ:nablakchi}
    \||\nabla|^k \chi\|_{L^p}\lesssim_{p,k}1.
\end{equation}
Moreover, there exists $\delta_0>0$ such that
\begin{equation}\label{equ:chidecay}
    |\chi(x)|\lesssim e^{-\delta_0|x|}.
\end{equation}
\end{lemma}
\begin{proof}
    Clearly, $\chi$ satisfies the elliptic equation
    \begin{equation}\label{egen}
        -\Delta \chi+(1-\lambda)\chi=\phi_{Q^2} \chi+2Q\phi_{\chi Q}.
    \end{equation}
By Proposition~\ref{grondstate1}, $\phi_{Q^2}\in L^\infty$ and there
exists $\delta>0$ such that
\[
 |Q(x)|+|\nabla Q(x)|\lesssim e^{-\delta|x|}.
\]
We first bootstrap the integrability of $\chi$. Since $d\geq3$,
the restrictions of $K(x)=|x|^{-2}$ to $\{|x|\leq1\}$ and
$\{|x|>1\}$ belong to $L^1$ and $L^\infty$, respectively.
Using $Q\in L^q$ for every $1\leq q\leq\infty$, Young's and
H\"older's inequalities therefore give
\[
 \|\phi_{Q^2}f+2Q\phi_{Qf}\|_{L^p}
 \lesssim_p \|f\|_{L^p},\qquad 1<p<\infty.
\]
Starting with $\chi\in L^2$, global elliptic regularity applied to
\eqref{egen}, followed by Sobolev embedding, yields by iteration
$\chi\in W^{2,p}$ for every $2\leq p<\infty$.
Taking $p>d$, we obtain
$\chi\in C^{1,\mu}(\mathbb{R}^d)$ with $\mu=1-d/p\in(0,1)$.
In particular, $\chi$ is uniformly Lipschitz continuous.

Next we prove the exponential decay of $\chi$. To this end, we multiply both sides of \eqref{egen} by the function $\theta_{\epsilon}\bar{\chi}:=e^{\frac{\delta}{2}\frac{|x|}{1+\epsilon|x|}}\bar{\chi}$ and then integrate by parts, then we have
\begin{equation}\label{regu1}
\begin{aligned}
&\Re\int\nabla\chi\cdot\nabla(\theta_\epsilon\bar\chi)\,dx
+(1-\lambda)\int\theta_\epsilon|\chi|^2\,dx\\
&\quad=\int\theta_\epsilon\phi_{Q^2}|\chi|^2\,dx
+2\Re\int\theta_\epsilon\phi_{\chi Q}Q\bar\chi\,dx.
\end{aligned}
\end{equation}

Noticing that $\nabla (\theta_{\epsilon} \bar{\chi}) = \bar{\chi} \nabla \theta_{\epsilon} + \theta_{\epsilon} \nabla \bar{\chi}$ and $|\nabla \theta_{\epsilon}|\leq\theta_{\epsilon}$ for all small $\delta,\epsilon>0$, using the Cauchy--Schwarz inequality, it is easy to see that
\[
\Re \big( \nabla \chi \cdot \nabla (\theta_{\epsilon} \bar{\chi}) \big) \geq \theta_{\epsilon} |\nabla \chi|^2 - \theta_{\epsilon} |\chi| |\nabla \chi| \geq \theta_{\epsilon} |\nabla \chi|^2 - \left[ \frac{1}{2} \theta_{\epsilon} |\chi|^2 + \frac{1}{2} \theta_{\epsilon} |\nabla \chi|^2 \right].
\]
Plugging it into \eqref{regu1}, we obtain
\begin{equation}\label{regu2}
  \int_{\mathbb{R}^d} \theta_{\epsilon} |\chi|^2 \, dx \leq 2\int_{\mathbb{R}^d} \theta_{\epsilon} \phi_{Q^2}|\chi|^{2} \, dx+4\Re\int_{\mathbb{R}^d} \theta_{\epsilon} \phi_{\chi Q}Q\bar{\chi} \, dx.  
\end{equation}
By Proposition \ref{grondstate1}, we know that $\phi_{Q^2}\lesssim \frac{1}{|x|^2}$ as $|x|\to\infty$. And hence there exists $R>0$ such that
\[
|\phi_{Q^2}(x)| \leq \frac{1}{4}, \quad |x| \geq R.
\]
Using this fact together with $|\theta_{\epsilon}(x)| \leq e^{\frac{\delta}{2}|x|}$ we obtain
\begin{equation}\label{regu3}
   2 \int_{\mathbb{R}^d} \theta_{\epsilon} \phi_{Q^2}|\chi|^{2} \, dx \leq 2 \int_{|x| \leq R} e^{\frac{\delta}{2}|x|} \phi_{Q^2}|\chi|^{2} \, dx + \frac{1}{2} \int_{\mathbb{R}^d} \theta_{\epsilon} |\chi|^2 \, dx. 
\end{equation}
Substituting this inequality into \eqref{regu2} gives
\begin{align*}
\int_{\mathbb{R}^d}\theta_{\epsilon}|\chi|^2\,dx
&\leq 4\int_{|x|\leq R}
 e^{\frac{\delta}{2}|x|}\phi_{Q^2}|\chi|^2\,dx
 +8\left|\int_{\mathbb{R}^d}
 \theta_{\epsilon}\phi_{\chi Q}Q\bar{\chi}\,dx\right|\\
&\leq C(\delta)<\infty,
\end{align*}
where $C(\delta)$ is independent of $\epsilon$.
Indeed, the kernel decomposition above and
$\chi Q\in L^1\cap L^\infty$ imply
$\phi_{\chi Q}\in L^\infty$, and
\[
\left|\int_{\mathbb{R}^d}
\theta_{\epsilon}\phi_{\chi Q}Q\bar{\chi}\,dx\right|
\leq
\|\phi_{\chi Q}\|_{L^\infty}
\|e^{\frac{\delta}{2}|x|}Q\|_{L^2}
\|\chi\|_{L^2}.
\]
Letting $\epsilon\to0$, we get
\[\int_{\mathbb{R}^d} e^{\frac{\delta}{2}|x|} |\chi|^2 \, dx < \infty.\]
The finiteness of $\int_{\mathbb{R}^d} e^{\frac{\delta}{2}|x|} |\chi|^2 \, dx$ is sufficient for us to obtain the desired decay property. In fact,
as $\chi(x)$ is uniformly Lipschitz continuous on $\mathbb{R}^d$, there exists $L > 0$ such that
\[
|\chi(x) - \chi(y)| \leq L|x - y| \implies |\chi(x)| \leq |\chi(y)| + L|x - y|, \quad \forall y \in \mathbb{R}^d.
\]
Naturally we also have
\[
|\chi(x)|^2 \leq 2|\chi(y)|^2 + 2L^2|x - y|^2, \quad \forall y \in \mathbb{R}^d.
\]
Consequently,
\[
|\chi(x)|^2 \leq 4|\chi(y)|^2, \quad \forall y \in B_{0}(x) := \{y : |x - y| \leq \tfrac{1}{2L}|\chi(x)|\}.
\]
Integrating both sides over $y \in B_0(x)$ yields
\begin{equation}\label{equ:chixcondchiy}
    \omega_d \left( \frac{1}{2L}|\chi(x)| \right)^d |\chi(x)|^2 \leq 4 \int_{B_0(x)} |\chi(y)|^2 \, dy.
\end{equation}
Since $\chi\in L^2(\mathbb{R}^d)$ is uniformly continuous,
$\chi(x)\to0$ as $|x|\to\infty$. Thus there exists $R_0\gg1$
such that $|\chi(x)|\leq1$ whenever $|x|\geq R_0$.
For such $x$, we have
\[
|x - y| \leq \frac{1}{2L} |\chi(x)| \leq \frac{1}{2L} \implies |y| - |x| + \frac{1}{2L} \geq 0, \quad  y \in B_0(x),
\]
which implies
\[
e^{|x|} \lesssim e^{|y|} \implies e^{|y| - |x|} \gtrsim 1.
\]
This together with \eqref{equ:chixcondchiy} yields
\[
|\chi(x)|^{d+2} e^{\frac{\delta}{2}|x|} \lesssim \int_{B_0(x)} e^{\frac{\delta}{2}|y|} |\chi(y)|^2 \, dy \leq \int_{\mathbb{R}^d} e^{\frac{\delta}{2}|y|} |\chi(y)|^2 \, dy < \infty, \quad |x| \geq R_0,
\]
which implies $|\chi|\lesssim e^{-\delta_0|x|}$ with $\delta_0=\frac{\delta}{2(d+2)}$. And so \eqref{equ:chidecay} follows.

It remains to prove $\chi\in W^{k,p}$ for any $k\geq0$ and $p>1$. To this end, we observe that as $|\chi|\lesssim e^{-\delta_0|x|}$, $\chi\in L^p$ for any $p>1$ and hence by the Hardy-Littlewood-Sobolev inequality and H\"older's inequality, implies $\phi_{Q^2} \chi+2Q\phi_{\chi Q}\in L^p$ for any $p>1$, which again by the Calderón-Zygmund estimate implies that 
    $\chi\in W^{2,p}$. Now using chain rule and the fact that $Q\in W^{k_0,p}$ for all $k_0\geq0$ and $p>1$,  the desired estimate comes from Calderón-Zygmund estimate by iteration.
\end{proof}

Our next goal is  to present a decomposition of the solution $u(t)$ and its initial properties. This decomposition is pivotal in deriving the uniform control of the remainder in the modulation analysis. Recalling Lemma \ref{att}, if $u(t)$ is the blowup solution in Theorem \ref{t2.3}, it admits the decomposition:
\begin{align*}
	& u(t)=e^{i\gamma_{0} (t)}e^{ix\cdot\xi_0(t)}\lambda_{0}^{d/2}(t)\left[Q(\lambda_{0}(t)( x + x_0(t)) )+\epsilon_{0}(t,\lambda_{0}(t)( x + x_0(t)) )\right], 
      \\
    & \mbox{ with }  \|\epsilon_0(t)\|_{L^2_x }=\delta\big(u(t)\big)\leq \eta_{\ast}.
\end{align*}
We then claim that we can refine this decomposition via  using the implicit function theorem, which can be stated as the following theorem:
\begin{theorem}\label{modulation}
 Let $u\in L^2(\R^d)$. There exists $\alpha > 0$ sufficiently small such that if there exist $\lambda_{0} > 0$, $\gamma_0\in [0,2\pi]$, $x_{0} \in \mathbb{R}^d$, and $\xi_{0} \in \mathbb{R}^d$ satisfying 
\begin{equation*} 
\| e^{i \gamma_{0}} e^{ix \cdot \xi_{0}} \lambda_{0}^{\frac{d}{2}} u (\lambda_{0} x + x_{0}) -
Q(x) \|_{L^{2}} \leq \alpha,
\end{equation*}
then there exist unique parameters $\lambda>0$, $\gamma\in\R$,
$\tilde x,\xi\in\R^d$ in a sufficiently small neighborhood of
$(\lambda_0,\gamma_0,x_0,\xi_0)$, measured by the left-hand side of
\eqref{f4.4}, such that if 
\begin{equation*} 
\epsilon(x):= e^{i \gamma} e^{ix \cdot \xi} \lambda^{d/2} u \left(\lambda x + \tilde{x} \right) - Q(x),
\end{equation*}
then the following orthogonality conditions hold: 
\begin{align}\label{f4.3}
( \epsilon,\chi  )_{L^2}
&= ( \epsilon,i\chi  )_{L^2}
=( \epsilon,\partial_{x_j}Q  )_{L^2}=( \epsilon,i\partial_{x_j}Q  )_{L^2}=0, \quad j=1,\cdots,d,
\end{align}
and, with the phase chosen near
$\gamma_0+\lambda_0^{-1}\xi_0\cdot(\tilde x-x_0)$, we have
\begin{align}\label{f4.4}
\aligned
& \|\epsilon\|_{L^2} +   \left|\frac{\lambda}{\lambda_{0}} - 1 \right|
+  \left|\gamma - \gamma_{0} - \frac{\xi_{0}}{\lambda_{0}} \cdot \left(\tilde{x} - x_{0} \right) \right|
 +  \left|\xi - \frac{\lambda}{\lambda_{0}} \xi_{0} \right| +  \left|\frac{\tilde{x} - x_{0}}{\lambda_{0}} \right|  \\
\lesssim& \| e^{i \gamma_{0}} e^{ix \cdot \xi_{0}} \lambda_{0}^{d/2} u \left(\lambda_{0} x + x_{0} \right) - Q(x) \|_{L^{2}}.
\endaligned
\end{align}
\end{theorem}
The proof of this theorem is in the same spirit as that of \cite{D2,D1}, we therefore omit the proof. Based on this theorem,  if $u(t)$ is the solution as in Theorem \ref{t2.3}, then there exist functions $\lambda(t)>0$, $\gamma(t)\in\mathbb R$, $x(t)\in\R^d$ and $\xi(t)\in \R^d$ on $[0,\sup I)$ such that if
\begin{equation}\label{decomp}
\epsilon(t,x) = e^{i \gamma(t)} e^{ix \cdot \xi(t)} \lambda(t)^{d/2} u\big(t,\lambda(t) x + x(t)\big) - Q(x),
\end{equation}
then the following orthogonality conditions hold: 
\begin{align}\label{orthod}
\aligned
( \epsilon,\chi )_{L^2_x }
&= ( \epsilon,i\chi )_{L^2_x}
=( \epsilon,\partial_{x_j}Q  )_{L^2_x}=( \epsilon,i\partial_{x_j}Q  )_{L^2_x}=0, \quad j=1,\cdots,d.
\endaligned
\end{align}
Then, by the standard modulation argument (cf. \cite{D2,D1}),
$\lambda(t)$, $\gamma(t)$, $x(t)$ and $\xi(t)$ are locally absolutely
continuous on $[0,\sup I)$. Moreover,
if we define the monotone function $s : [0, \sup I ) \rightarrow \mathbb{R}$ as
\begin{equation}\label{f4.23}
	s(t) = \int_{0}^{t} \lambda(\tau)^{-2} \, { d} \tau,
\end{equation}
and let
\begin{equation}\label{f4.24}
	\aligned
	\lambda(s) = \lambda(t(s)), \qquad \gamma(s) = \gamma(t(s)), \qquad \\
	x(s) = x(t(s)),  \qquad \xi(s) = \xi(t(s))
	\endaligned
\end{equation}
and 
\begin{equation}\label{f4.25}
	\epsilon (s, x) = e^{i \gamma(s)} e^{ix \cdot\xi(s)} \lambda^{d/2}(s) u(t(s), \lambda(s) x + x(s)) - Q(x),
\end{equation}
then we have the following control of these parameters:
\begin{equation}\label{f4.45}
	\int_{a}^{a + 1}  \left|\frac{\lambda_{s}}{\lambda} \right| \, {d} s
\lesssim \int_{a}^{a + 1}  \| \epsilon(s)  \|_{L^2_{x}} \,  {d} s.
\end{equation}
\begin{equation}\label{f4.53}
	\int_{a}^{a + 1}  \left|(\gamma)_s + 1 - \frac{x_{s}}{\lambda} \cdot\xi(s) - |\xi(s)|^{2}  \right|  \, {d} s
 \lesssim \int_{a}^{a + 1}  \|\epsilon(s) \|^2_{L^2_{x}} \, {d} s,
\end{equation}
\begin{equation}\label{f4.54}
	\int_{a}^{a + 1}  \left|\xi_{s} - \frac{\lambda_{s}}{\lambda} \xi(s) \right| \, {d} s
\lesssim \int_{a}^{a + 1}  \| \epsilon(s) \|_{L^2_{x}}^{2} \, {d} s,
\end{equation}
and
\begin{equation}\label{f4.55}
	\int_{a}^{a + 1}  \left|\frac{x_{s}}{\lambda} + 2 \xi (s) \right| \, {d} s
 \lesssim \int_{a}^{a + 1}  \| \epsilon(s)\|_{L^2_{x}} \, {d} s.
\end{equation}
\begin{equation}\label{f4.56}
		\sup_{s\in [a,a+1]} \|\epsilon(s) \|_{L^2_x }
\sim\inf_{s\in [a,a+1]} \|\epsilon(s) \|_{L^2_x }.
	\end{equation}
Finally, we record the following positivity result, which is established in the radial setting(cf. \cite{Wang-Gao}). Their proof only uses the positivity of the operators  $L_{+}$  and $L_{-}$(cf. Theorem \ref{posit}), and is therefore also valid in the non-radial case. We therefore only state the theorem and omit the proof.
\begin{proposition}\label{positive1}
    Suppose that $\epsilon=\epsilon_1+i\epsilon_2\in H^1$  satisfies  $\|\epsilon\|_{L^2}\ll1$, $\|Q+\epsilon\|_{L^2}=\|Q\|_{L^2}$ and
    \[
    \epsilon\perp_{L^2}\{\chi,i\chi,\partial_{x_1}Q,\cdots,\partial_{x_d}Q,i\partial_{x_1}Q,\cdots,i\partial_{x_d}Q\},
    \]
    then there exists $c_d>0$ such that
    \[
    (L_{+}\epsilon_1,\epsilon_1)+(L_-\epsilon_2,\epsilon_2)\geq c_d\|\epsilon\|^2_{H^1}.
    \]
\end{proposition}

\end{document}